\documentclass{article}

\usepackage{amsfonts}
\usepackage{amsmath}
\usepackage{mathtools}
\usepackage{multirow}
\usepackage{epstopdf}
\usepackage{algorithm}
\usepackage{algorithmic}
\usepackage{bm}
\usepackage{microtype}
\usepackage{amssymb}
\usepackage{enumerate}
\usepackage{marvosym}
\usepackage{color}
\usepackage{marvosym}
\usepackage{subcaption}
\usepackage{stmaryrd}

\newtheorem{theorem}{Theorem}

\newtheorem{remark}{Remark}
\newtheorem{lemma}{Lemma}
\newtheorem{proof}{Proof}

\newtheorem{definition}{Definition}

\newtheorem{fact}{Fact}

\newcommand{\x}{\mathbf{x}}
\newcommand{\y}{\mathbf{y}}

\newcommand{\g}{\mathbf{g}}

\newcommand{\m}{\mathbf{m}}
\newcommand{\1}{\mathbf{1}}
\newcommand{\0}{\mathbf{0}}

\newcommand{\vv}{\mathbf{v}}

\newcommand{\E}{\mathbb{E}}
\newcommand{\bO}{{\cal O}}
\newcommand{\bbR}{\mathbb{R}}

\newcommand{\F}{\mathcal{F}}

\newcommand{\G}{\mathbf{G}}
\newcommand{\bR}{\mathbf{R}}
\newcommand{\I}{\mathbf{I}}
\newcommand{\U}{\mathbf{U}}
\newcommand{\V}{\mathbf{V}}
\newcommand{\M}{\mathbf{M}}

\newcommand{\X}{\mathbf{X}}
\newcommand{\Y}{\mathbf{Y}}
\newcommand{\A}{\mathbf{A}}
\newcommand{\B}{\mathbf{B}}

\newcommand{\bL}{\mathbf{L}}
\newcommand{\bLambda}{\mathbf{\Lambda}}

\newcommand{\<}{\left\langle}
\renewcommand{\>}{\right\rangle}

\DeclareMathOperator*{\diag}{diag}
\DeclareMathOperator*{\tr}{tr}
\renewcommand{\vec}{\mbox{vec}}

\usepackage[accepted]{icml2020}

\icmltitlerunning{Convergence Rate Analysis of the AdamW-style Shampoo}

\begin{document}

\onecolumn
\icmltitle{Convergence Rate Analysis of the AdamW-style Shampoo: \\Unifying One-Sided and Two-Sided Preconditioning}

\begin{icmlauthorlist}
\icmlauthor{Huan Li}{to}
\icmlauthor{Yiming Dong}{goo}
\icmlauthor{Zhouchen Lin}{goo}
\end{icmlauthorlist}

\icmlaffiliation{to}{Institute of Robotics and Automatic Information Systems, College of Artificial Intelligence, Nankai University, Tianjin, China.\\}
\icmlaffiliation{goo}{National Key Lab of General AI, School of Intelligence Science and Technology, Peking University, Beijing, China.\\}
\icmlcorrespondingauthor{Huan Li and Zhouchen Lin}{lihuanss@nankai.edu.cn, zlin@pku.edu.cn}





\vskip 0.3in



\printAffiliationsAndNotice{}  

\begin{abstract}
  This paper studies AdamW-style Shampoo, an effective variant of the classical Shampoo that won the external tuning track of the AlgoPerf neural network training competition \citep{AlgoPerf-competition-25}. Our analysis unifies one-sided and two-sided preconditioning. When the exponents of the two preconditioners sum to $1/2$, we establish the convergence rate $\frac{1}{K}\sum_{k=1}^K\E\left[\|\nabla f(\X_k)\|_*\right]\leq \bO(\frac{\sqrt{m+n}C}{K^{1/4}})$, where $K$ represents the number of iterations, $(m,n)$ denotes the dimensions of the matrix-valued parameters, and $C$ matches the constant appearing in the optimal convergence rate of SGD. Theoretically, the nuclear norm and Frobenius norm satisfy $\|\nabla f(\X)\|_F\leq \|\nabla f(\X)\|_*\leq \sqrt{\min\{m,n\}}\|\nabla f(\X)\|_F$. This suggests that our convergence rate is analogous to the optimal $\frac{1}{K}\sum_{k=1}^K\E\left[\|\nabla f(\X_k)\|_F\right]\leq \bO(\frac{C}{K^{1/4}})$ convergence rate of SGD in the ideal case where $\|\nabla f(\X)\|_*= \Theta(\sqrt{\min\{m,n\}})\|\nabla f(\X)\|_F$ and $m$ and $n$ are of comparable magnitude. Then, we extend our analysis to settings where the preconditioning exponents do not sum to $1/2$, and establish convergence with an explicit but more involved rate.
\end{abstract}

\section{Introduction}

Adaptive gradient methods have become the predominant optimizers for training deep neural networks, especially in large language models. The development of adaptive gradient algorithms has followed two distinct directions: diagonal preconditioning and non‑diagonal preconditioning. The former has developed through AdaGrad \citep{Duchi-2011-jmlr,McMahan-2010-colt}, RMSProp \citep{RMSProp-2012-hinton}, Adam \citep{adam-15-iclr}, and finally AdamW \citep{adanw-2019-iclr}, which have served as the de facto optimizers for training deep networks over the past decade. The latter lineage, originating from full-matrix AdaGrad \citep{Duchi-2011-jmlr} and advancing to methods such as K-FAC \citep{kfac-15}, Shampoo \citep{shampoo-icml-18}, SOAP \citep{soap-2025-iclr}, and Muon \citep{muon2024}, has recently demonstrated the potential to outperform its diagonal counterparts.

Diagonally preconditioned methods typically apply coordinate-wise scaling to the gradient. For instance, AdaGrad treats the network’s parameters as a single high-dimensional vector and updates them according to the following procedure
\begin{eqnarray}
\begin{aligned}\notag
\x_{k+1}=\x_k-\eta\Lambda_k^{-1/2}\g_k,\quad\Lambda_k=\diag\left(\sum_{t=1}^k\diag(\g_t\g_t^T)\right),
\end{aligned}
\end{eqnarray} 
where $\Lambda_k$ is a diagonal preconditioning matrix whose entries are the coordinate-wise sum of squared historical gradients. 

\begin{algorithm}[t]
    \caption{AdamW-style Shampoo}
    \label{alg1}
    \begin{algorithmic}
       \STATE Hyperparameters: $\eta,\theta,\beta,\lambda,\varepsilon$. Let $p,q\in(0,+\infty]$ with $\frac{1}{p}+\frac{1}{q}=2-\omega$ and $\omega\in(0,2]$.
       \STATE Denote $\bL_{k,\varepsilon}^{\pm\frac{1}{\infty}}=\I_m$ and $\bR_{k,\varepsilon}^{\pm\frac{1}{\infty}}=\I_n$.
       \STATE Initialize $\X_1$, $\M_0=\0$, $\bL_0=\0$, $\bR_0=\0$.
       \FOR{$k=1,2,\cdots,K$}
       \STATE $\G_k=\mbox{GradOracle}(\X_k)$
       \STATE $\M_k=\theta\M_{k-1}+(1-\theta)\G_k$
       \STATE $\bL_k=\beta \bL_{k-1}+(1-\beta)\G_k\G_k^T$
       \STATE $\bR_k=\beta \bR_{k-1}+(1-\beta)\G_k^T\G_k$
       \STATE $\bL_{k,\varepsilon}=\bL_k+\varepsilon\I_m,\quad \bR_{k,\varepsilon}=\bR_k+\varepsilon\I_n$
       \STATE $\X_{k+1}=(1-\lambda\eta)\X_k - \eta\bL_{k,\varepsilon}^{-\frac{1}{2p}}\M_k\bR_{k,\varepsilon}^{-\frac{1}{2q}}$
       \ENDFOR
    \end{algorithmic}
\end{algorithm}

In contrast, non-diagonally preconditioned methods exploit the inherent matrix structure of neural network parameters. For example, Shampoo operates according to the following formulation with two-sided preconditioning
\begin{eqnarray}
\begin{aligned}\label{shampoo}
&\hspace*{0.3cm}\X_{k+1}=\X_k-\eta\bL_k^{-1/4}\G_k\bR_k^{-1/4},\\
&\bL_k=\sum_{t=1}^k\G_t\G_t^T,\quad\bR_k=\sum_{t=1}^k\G_t^T\G_t,
\end{aligned}
\end{eqnarray} 
where the gradient $\G_k\in \bbR^{m\times n}$ is a matrix and $\bL_k\in\bbR^{m\times m}$ and $\bR_k\in\bbR^{n\times n}$ are non-diagonal preconditioners consisting of sum of historical gradient outer products. This can be regarded as using the Kronecker product $\bR_k^{1/2}\otimes \bL_k^{1/2}$ to approximate the full-matrix AdaGrad preconditioner $\sum_{t=1}^k \g_t\g_t^T$, where $\g_t=\vec(\G_t)$.

A key advantage of non-diagonally preconditioned methods is their ability to capture the cross-parameter correlations in gradient, thereby yielding a more informed search direction and potentially superior convergence compared to diagonal approaches. Recently, an implementation based on the distributed Shampoo \citep{distribute-shampoo-20,distribute-shampoo-23} won the external tuning track of the AlgoPerf neural network training competition \citep{AlgoPerf-competition-25}, demonstrating that non-diagonally preconditioned training algorithms can outperform currently popular diagonal preconditioning methods, such as Adam. The winning implementation achieved significantly accelerated training, with an average speedup of $28\%$ over the NAdamW \citep{Dozat-nadam-iclr-2016} baseline across eight deep learning workloads. Algorithm \ref{alg1} presents the core characteristics of the Shampoo variant implemented in \citep{distribute-shampoo-20,distribute-shampoo-23} in a non-distributed manner, including the exponential moving average of the first and second moment matrices, decoupled weight decay, and two-sided preconditioning with a tunable exponent. Recently, \citet{Runa-shampoo-26} empirically observed that setting the exponents $\frac{1}{2p}=\frac{1}{2q}=\frac{1}{2}$ outperforms the classical choice $\frac{1}{2p}=\frac{1}{2q}=\frac{1}{4}$.

Theoretically, convergence of diagonally preconditioned methods has been extensively studied \citep{bottou-2022-tmlr,luo-2020-iclr,lihuan-rmsprop-2024,luo-2022-nips,hong-2024-adam,haochuanli-2023,Li-2025-nips}. For non-diagonal methods, Muon represents the first method to receive a rigorous convergence analysis for nonconvex optimization due to its simple structure \citep{muon-hongmingyi-25,Muon-NS-ICLR-26,muon-shen-25,muon-chen-25,muon-sato-25}. Analyses of other optimizers in this class, such as Shampoo, have largely been confined to convex settings. For example, \citet{shampoo-icml-18} established the regret bound of the classical Shampoo (\ref{shampoo}) within the online convex optimization framework, \citet{shampoo-xie-icml-25} provided a unified convergence analysis including full-matrix AdaGrad and \underline{\textit{one-sided}} variant of Shampoo for convex problems, where the update $\bL_k^{-1/4}\G_k\bR_k^{-1/4}$ in the original Shampoo is replaced by $\bL_k^{-1/2}\G_k$, \citet{KangAn-nips-25} proposed ASGO, effectively equivalent to \underline{\textit{one-sided}} Shampoo, and studied its convergence for convex programming. 

To the best of our knowledge, \citep{shampoo-xie-25} appears to be the only work prior to ours that establishes the convergence of Shampoo in the nonconvex setting.  However, their study is limited to the \underline{\textit{one-sided}} variant of Shampoo in the \underline{\textit{AdaGrad-style}} and \underline{\textit{RMSProp-style}}, and does not address the more complex yet more commonly used two-sided preconditioning, which has been empirically observed to perform better than the one-sided variant \citep[Takeaway \#3]{Runa-shampoo-26}. Furthermore, their analysis does not incorporate momentum or decoupled weight decay. While other works \citep{Sketchy-2023,Morwani-2025-iclr,shampoo-Runa-25,KL-Shampoo-2025} have explored Shampoo from different perspectives, none has provided a convergence guarantee for the nonconvex case.

\subsection{Contributions}

In this paper, we study the AdamW-style Shampoo method presented in Algorithm \ref{alg1}, which provides a unified treatment of two-sided ($p,q<+\infty$) and one-sided ($p=1,q=+\infty$ or $p=+\infty,q=1$) preconditioning. This formulation captures key components of the most effective practical implementations of Shampoo \citep{distribute-shampoo-20,distribute-shampoo-23}, as evidenced by its success in the AlgoPerf competition \citep{AlgoPerf-competition-25}. Our contributions are twofold.

1. In the classical setting where $\frac{1}{p}+\frac{1}{q}=1$, we establish the following convergence rate of Algorithm \ref{alg1} for nonconvex programming
\begin{eqnarray}
\begin{aligned}\label{shampoo-rate1}
\frac{1}{K}\sum_{k=1}^K\E\left[\left\|\nabla f(\X_k)\right\|_*\right]\leq \bO\left(\sqrt{m+n}\max\left\{\sqrt[4]{\frac{\sigma^2L\left(f(\X_1)-f^*\right)}{K}},\sqrt{ \frac{L\left(f(\X_1)-f^*\right)}{K} }\right\}\right) 
\end{aligned}
\end{eqnarray} 
and show that $\lambda\|\X_k\|_{op}<1$ for all $k=1,2,\cdots,K$, where the notations can be found in Section \ref{sec:assumption}. For comparison, the classical convergence rate of SGD is \citep{Bottou-2018-siam}
\begin{eqnarray}
\begin{aligned}\label{rate-sgd}
&\frac{1}{K}\sum_{k=1}^K\E\left[\left\|\nabla f(\X_k)\right\|_F\right]\leq \bO\left(\sqrt[4]{\frac{\sigma^2L\left(f(\X_1)-f^*\right)}{K}}\right),
\end{aligned}
\end{eqnarray} 
which matches the lower bound of nonconvex stochastic optimization \citep{Arjevani-2023-mp}. Since Frobenius norm and nuclear norm satisfy 
\begin{eqnarray}
\begin{aligned}\notag
&\|\nabla f(\X)\|_F\leq \|\nabla f(\X)\|_*\leq \sqrt{\min\{m,n\}}\|\nabla f(\X)\|_F,
\end{aligned}
\end{eqnarray} 
our convergence rate also aligns with the same lower bound with respect to all the coefficients in the ideal case where $\|\nabla f(\X)\|_*=\Theta(\sqrt{\min\{m,n\}})\|\nabla f(\X)\|_F$ and $m$ and $n$ are of comparable magnitude, which is verified empirically on real training of GPT-2 in our experiment.

2. We then extend the analysis to the more general setting $\frac{1}{p}+\frac{1}{q}=2-\omega$ for any $\omega\in(0,2]$, and establish the following convergence rate for Algorithm 1 without decoupled weight decay (that is, $\lambda=0$)
\begin{eqnarray}
\begin{aligned}\label{shampoo-rate2}
\frac{1}{K}\sum_{k=1}^K\E\left[\left\||\nabla f(\X_k)|^{\omega}\right\|_*\right]
\leq \bO\left(\frac{(m+n)^{1-\frac{\omega}{2}}}{\omega K^{\frac{\omega}{4}}} \right),
\end{aligned}
\end{eqnarray}
where $|\A|^{\omega}=(\A^T\A)^{\omega/2}$ and $\||\A|^{\omega}\|_*=\sum_i\sigma_i(\A)^{\omega}$. When $\omega=1$, the convergence rate (\ref{shampoo-rate2}) reduces to the same order as rate (\ref{shampoo-rate1}). When $\omega=2$, Algorithm \ref{alg1} reduces to momentum SGD and the rate (\ref{shampoo-rate2}) reduces to rate (\ref{rate-sgd}) via the identity $\left\||\nabla f(\X_k)|^2\right\|_*=\left\|\nabla f(\X_k)\right\|_F^2$. For other $\omega$ in $(0,2]$, comparing the convergence rates (\ref{shampoo-rate2}), (\ref{shampoo-rate1}), and (\ref{rate-sgd}) is not straightforward, as $\omega$ appears on both sides of (\ref{shampoo-rate2}).

\subsection{Problem Setting, Notation, and Assumptions}\label{sec:assumption}

In this paper, we study the following nonconvex problem with matrix-valued parameters 
\begin{equation}\label{problem}
\min_{\X\in\bbR^{m\times n}} f(\X),
\end{equation}
where $f(\X)=\E_{\zeta\in\mathcal{P}}[f(\X;\zeta)]$ and $\zeta$ is the sample drawn from the data distribution $\mathcal{P}$.

We denote vectors by lowercase bold letters and matrices by uppercase bold letters. We use $\I_m$ for the identity matrix in $\mathbb R^{m\times m}$. For vectors, we denote $\|\cdot\|$ as the $\ell_2$ Euclidean norm. For matrices, denote $\|\cdot\|_F$, $\|\cdot\|_{op}$, and $\|\cdot\|_*$ as the Frobenius norm, spectral norm (largest singular value), and nuclear norm (sum of singular values), respectively. The trace of a square matrix is written as $\tr(\cdot)$. Let $\F_k=\sigma(\G_1,\G_2,\cdots,\G_k)$ denote the sigma-algebra of the stochastic gradients up to $k$, $\E_{\F_k}[\cdot]$ the expectation with respect to $\F_k$, and $\E_k[\cdot|\F_{k-1}]$ the conditional expectation with respect to $\G_k$ given $\F_{k-1}$. For brevity, we write $\E[\cdot]$ for $\E_{\F_K}[\cdot]$. Let $f^*$ denote the lower bound of $f(\X)$. Denote the singular values of $\A\in \bbR^{m\times n}$ by $\sigma_1(\A),\cdots,\sigma_r(\A)$ in a nonincreasing order with $r=\min\{m,n\}$. Finally, following \citep{rajendra-book}, we denote $|\A|=(\A^T\A)^{1/2}$, where the matrix powers are defined as follows
\begin{definition}\label{def-X-p}
For symmetric positive semidefinite matrix $\X\in\mathbb R^{m\times m}$, let $\U\bLambda\U^T$ be its eigenvalue decomposition with $\bLambda=\diag(\lambda_1,\lambda_2,\cdots,\lambda_m)$, then $\X^p$ is defined to be $\U\diag(\lambda_1^p,\lambda_2^p,\cdots,\lambda_m^p)\U^T$. 
\end{definition}
Let $\U\Sigma\V^T$ be the singular value decomposition (SVD) of $\A$, then we have
\begin{eqnarray}
\begin{aligned}\label{matrix-power}
|\A|=\V\Sigma\V^T\quad\mbox{and}\quad |\A|^p=\V\Sigma^p\V^T.
\end{aligned}
\end{eqnarray}
We make the following assumptions throughout this paper:
\begin{enumerate}
\item Smoothness: $\|\nabla f(\Y)-\nabla f(\X)\|_F\leq L\|\Y-\X\|_F, \forall \X,\Y$,
\item Unbiased estimator: $\E_k\left[\G_k\big|\F_{k-1}\right]=\nabla f(\X_k)$,
\item Bounded row-wise and column-wise second central moment matrices: \\
$\E_k\left[\left(\G_k-\nabla f(\X_k)\right)\left(\G_k-\nabla f(\X_k)\right)^T\big|\F_{k-1}\right]\preceq \Sigma_L$,\\
$\E_k\left[\left(\G_k-\nabla f(\X_k)\right)^T\left(\G_k-\nabla f(\X_k)\right)\big|\F_{k-1}\right]\preceq \Sigma_R$\\ 
for some symmetric positive semidefinite matrices $\Sigma_L$ and $\Sigma_R$.
\end{enumerate}
The first two assumptions are identical to the standard assumptions used in the analysis of SGD, while the third assumption is more restrictive than that in SGD analysis. In fact, from the third assumption, it readily follows that
\begin{equation}
\E_k\left[\left\|\G_k-\nabla f(\X_k)\right\|_F^2\big|\F_{k-1}\right]\leq \frac{\tr\left(\Sigma_L\right)+\tr\left(\Sigma_R\right)}{2}\equiv\sigma^2,\label{bounded-variance}
\end{equation}
which is the standard bounded variance assumption in SGD analysis.

\section{Convergence Rate of AdamW-style Shampoo for $\omega=1$}

Based on Assumptions 1-3, we establish the convergence rate of Algorithm \ref{alg1} with $\omega=1$ in the following theorem. By the definitions of $\bL_{k,\varepsilon}$ and $\bR_{k,\varepsilon}$, condition (\ref{L-assump}) always holds with $\hat\varepsilon=\varepsilon$. 

\begin{theorem}\label{main-theorem}
Suppose that Assumptions 1-3 and condition
\begin{eqnarray}
\begin{aligned}\label{L-assump}
\bL_{k,\varepsilon}\succeq\hat\varepsilon\I_m,\quad \bR_{k,\varepsilon}\succeq\hat\varepsilon\I_n
\end{aligned}
\end{eqnarray}
hold for some $\hat\varepsilon\geq\varepsilon$. Let $\hat\sigma^2=\max\left\{\sigma^2,\frac{L\left(f(\X_1)-f^*\right)}{K\gamma^2}\right\}$ with any $\gamma\in(0,1)$, $\frac{1}{p}+\frac{1}{q}=1$, $1-\theta=\sqrt{\frac{L\left(f(\X_1)-f^*\right)}{K\hat\sigma^2}}$, $\theta\leq\beta\leq\sqrt{\theta}$, $\varepsilon=\frac{\tau\hat\sigma^2}{m+n}$ where $\tau\leq 1$ is the hyperparameter to control the small numerical value of $\varepsilon$ used in practice, $\eta=\sqrt{\frac{\hat\varepsilon\left(f(\X_1)-f^*\right)}{4LK\hat\sigma^2}}$, $0\leq\lambda\leq\frac{1}{\sqrt{1152\hat\varepsilon}K^{3/4}}\sqrt[4]{\frac{L^3\hat\sigma^2}{f(\X_1)-f^*}}$, and $\|\X_1\|_{op}\leq\sqrt{\frac{\hat\varepsilon K\left(f(\X_1)-f^*\right)}{L\hat\sigma^2}}$. Then for Algorithm \ref{alg1}, we have $\lambda\|\X_k\|_{op}<1$ for all $k=1,2,\cdots,K$ and
\begin{eqnarray}
\begin{aligned}\label{rate1}
\frac{1}{K}\hspace*{-0.05cm}\sum_{k=1}^K\hspace*{-0.05cm}\E\hspace*{-0.05cm}\left[\left\|\nabla f(\X_k)\right\|_*\right]\hspace*{-0.07cm}\leq\hspace*{-0.07cm} \left(\hspace*{-0.07cm}8\sqrt{m\hspace*{-0.05cm}+\hspace*{-0.05cm}n}\hspace*{-0.05cm}+\hspace*{-0.05cm}\frac{119\hat\sigma}{\sqrt{\hat\varepsilon}} \hspace*{-0.05cm}\right)\hspace*{-0.05cm}\max\hspace*{-0.05cm}\left\{\hspace*{-0.07cm}\sqrt[4]{\frac{\sigma^2L\left(f(\X_1)\hspace*{-0.05cm}-\hspace*{-0.05cm}f^*\right)}{K}},  \hspace*{-0.05cm}\sqrt{ \frac{ L\left(f(\X_1)\hspace*{-0.05cm}-\hspace*{-0.05cm}f^*\right)}{K\gamma} }\right\}\hspace*{-0.05cm}.\hspace*{-0.13cm} 
\end{aligned}
\end{eqnarray} 
In the worst case where only $\hat\varepsilon=\varepsilon$ is guaranteed, we have
\begin{eqnarray}
\begin{aligned}\label{rate2}
\frac{1}{K}\sum_{k=1}^K\E\left[\left\|\nabla f(\X_k)\right\|_*\right]\leq 127\sqrt{\frac{m+n}{\tau}}\max\left\{\sqrt[4]{\frac{\sigma^2L\left(f(\X_1)-f^*\right)}{K}},  \sqrt{ \frac{ L\left(f(\X_1)-f^*\right)}{K\gamma} }\right\}. 
\end{aligned}
\end{eqnarray} 
Furthermore, when $\tau=1$, we achieve the best theoretical convergence rate
\begin{eqnarray}
\begin{aligned}\label{rate3}
\frac{1}{K}\sum_{k=1}^K\E\left[\left\|\nabla f(\X_k)\right\|_*\right]\leq 127\sqrt{m+n}\max\left\{\sqrt[4]{\frac{\sigma^2L\left(f(\X_1)-f^*\right)}{K}},  \sqrt{ \frac{ L\left(f(\X_1)-f^*\right)}{K\gamma} }\right\}. 
\end{aligned}
\end{eqnarray} 
\end{theorem}

\subsection{Discussion on $\varepsilon$ and $\hat\varepsilon$}

In practice, the parameter $\varepsilon$ is typically set to a very small value, such as $10^{-12}$, as used in \citep{distribute-shampoo-23}. On the other hand, in modern large language models, the dimensions of weight matrices optimized by non-diagonally preconditioned methods such as Shampoo/SOAP/Muon are moderate in size. For instance, in GPT-3 with $175$ billion parameters, the QKV projection matrices have dimensions $m=n=12288$, while the weight matrices in the feed-forward network layer have dimensions $(m,n)=(12288,49152)$ or $(49152,12288)$. Consequently, the quantity $\frac{\hat\sigma^2}{m+n}$ is several orders of magnitude larger than the typical practical value of $\varepsilon$. To account for this gap, we parameterize $\varepsilon$ using a scaling factor $\tau$, which better aligns the analysis with practical configurations. This adjustment, however, yields a weaker convergence bound, as shown in (\ref{rate2}), which depends explicitly on $\tau$. The impractical setting $\tau=1$ yields the best convergence rate given in (\ref{rate3}). 

To bridge this gap between theory and practice, we further introduce condition (\ref{L-assump}). Informally, the preconditioners $\bL_k$ and $\bR_k$ can be regarded as approximating $\E\left[\G\G^T\right]$ and $\E\left[\G^T\G\right]$, respectively. In the training of modern large language models, empirical evidence indicates that the gradient norm remains $\bO(1)$ \citep[Figure 7]{kaiyue-compare-2025}. Consequently, Lemma \ref{condition2-lemma} suggests that condition (\ref{L-assump}) can reasonably hold with $\hat\varepsilon=\bO(\frac{1}{m+n})$. 
As illustrated in Figure \ref{fig:min_sval_noise} in Section \ref{sec-exp}, this condition is empirically satisfied during GPT‑2 training for a moderate value of $\hat\varepsilon$, which remains orders of magnitude larger than $\varepsilon$. 
With condition (\ref{L-assump}), our derived convergence rate (\ref{rate1}) depends only on $\hat\varepsilon$, rather than $\varepsilon$ or $\tau$. When $\hat\varepsilon\geq\frac{\hat\sigma^2}{m+n}$, convergence rate (\ref{rate1}) matches (\ref{rate3}) , even for arbitrarily small $\tau$. This represents a trade-off between theoretical guarantees and practical behavior. From the perspective of worst-case theoretical bounds, (\ref{rate3}) yields the optimal convergence rate, and condition (\ref{L-assump}) can be removed since it is always satisfied with $\hat\varepsilon=\varepsilon$. From a practical standpoint, (\ref{rate1}) better explains why an extremely small $\varepsilon$ does not hinder fast convergence in real-world scenarios.

\begin{lemma}\label{condition2-lemma}
When each entry of $\G\in\mathbb R^{m\times n}$ is generated independently from Gaussian distribution with mean $\mu$ and variance $\xi^2$, we have
\begin{eqnarray}
\begin{aligned}\notag
&\E\left[\G\G^T\right]\succeq n\xi^2\I_m=\frac{\xi^2}{m(\xi^2+\mu^2)}\E\left[\|\G\|_F^2\right]\I_m,\\
&\E\left[\G^T\G\right]\succeq m\xi^2\I_n=\frac{\xi^2}{n(\xi^2+\mu^2)}\E\left[\|\G\|_F^2\right]\I_n.
\end{aligned}
\end{eqnarray}
\end{lemma}

\subsection{Optimality of Our Convergence Rate}

Comparing the optimal convergence rate (\ref{rate-sgd}) of SGD with our best theoretical convergence rate (\ref{rate3}), we observe that our result is measured in the nuclear norm and contains an additional factor of $\sqrt{m+n}$. Let $\sigma_1,\sigma_2,\cdots,\sigma_r$ denote the singular values of $\nabla f(\X)$ with $r=\min\{m,n\}$, the Frobenius norm and nuclear norm satisfy 
\begin{eqnarray}
\begin{aligned}\notag
&\|\nabla f(\X)\|_F=\sqrt{\sum_{i=1}^r\sigma_i^2}\leq \sum_{i=1}^r\sigma_i=\|\nabla f(\X)\|_*,\\
&\|\nabla f(\X)\|_*=\sum_{i=1}^r\sigma_i\leq \sqrt{r\sum_{i=1}^r\sigma_i^2}=\sqrt{r}\|\nabla f(\X)\|_F.
\end{aligned}
\end{eqnarray} 
This means that our rate also aligns with the lower bound in nonconvex stochastic optimization \citep{Arjevani-2023-mp} in the ideal case where $\|\nabla f(\X)\|_*=\Theta(\sqrt{\min\{m,n\}})\|\nabla f(\X)\|_F$ and $m$ and $n$ are of comparable magnitude, as verified empirically in our GPT-2 training experiments demonstrated in Figure \ref{fig:nuc_fro_ratio} in Section \ref{sec-exp}. 

According to the recent work \citep{jiang-adagrad-2024}, the lower bound for AdaGrad is $\E\left[\min_{1\leq k\leq K}\|\nabla f(\x_k)\|_1\right]=\Theta\left(\frac{\sqrt{d}}{K^{1/4}}\sqrt[4]{\sigma^2L\left(f(\x_1)-f^*\right)}\right)$
for a constructed objective function $f(\x)$ under standard assumptions of $L$-smoothness and $\sigma^2$-bounded variance (\ref{bounded-variance}), where $d$ denotes the dimension of the variable. Consider the following simplified scenario: during the iteration process, the left and right singular matrices of $\G_k$ remain fixed, with $\U\Sigma_k\V^T$ being its compact singular value decomposition (SVD), then Shampoo is equivalent to running AdaGrad on the singular values in the sense of 
$\bL_k^{-1/4}\G_k\bR_k^{1/4}=\U\left(\sum_{t=1}^k\Sigma_t^2\right)^{-1/4}\Sigma_k\left(\sum_{t=1}^k\Sigma_t^2\right)^{-1/4}\V^T$. 
Since AdamW-style Shampoo further extends Shampoo, and given the unavoidable $\sqrt{d}$ dependence in AdaGrad's lower bound, we conjecture that the convergence rate derived in our paper is sharp and that the factor $\sqrt{m+n}$ cannot be eliminated.

\subsection{Unifying Two-Sided and One-Sided Preconditioning}

When $p,q<+\infty$, Algorithm \ref{alg1} employs two-sided preconditioning; for instance, setting $p=q=2$ recovers the original Shampoo update proposed in \citep{shampoo-icml-18}. If either $p$ or $q$ is infinite, the algorithm reduces to the one-sided preconditioning analyzed in \citep{shampoo-xie-icml-25,KangAn-nips-25,shampoo-xie-25}. Specifically, the update $\bL_{k,\varepsilon}^{-\frac{1}{2p}}\M_k\bR_{k,\varepsilon}^{-\frac{1}{2q}}$ reduces to the left-sided preconditioning $\bL_{k,\varepsilon}^{-\frac{1}{2}}\M_k$ when $p=1$ and $q=+\infty$, and to the right-sided preconditioning $\M_k\bR_{k,\varepsilon}^{-\frac{1}{2}}$ when $p=+\infty$ and $q=1$. Our analysis framework permits any positive values $p$ and $q$ (including infinite ones) satisfying $\frac{1}{p}+\frac{1}{q}=1$. Under this constraint, decreasing $p$ (and thus increasing $q$) strengthens the left preconditioner and weakens the right one; increasing $p$ has the opposite effect. At the extremes $p=1,q=+\infty$ and $p=+\infty,q=1$, one of the two preconditioners disappears, leaving only the other active. Empirically, \citet{distribute-shampoo-20} and \citet{distribute-shampoo-23} observed that treating the exponents $1/2p$ and $1/2q$ as tunable hyperparameters can lead to improved performance. Intuitively, the left preconditioning $\bL_{k,\varepsilon}^{-\frac{1}{2}}\M_k$ captures within-column correlations of $\M_k$, while the right preconditioning $\M_k\bR_{k,\varepsilon}^{-\frac{1}{2}}$ captures within-row correlations of $\M_k$. Two-sided preconditioning $\bL_{k,\varepsilon}^{-\frac{1}{2p}}\M_k\bR_{k,\varepsilon}^{-\frac{1}{2q}}$ combines these advantages and captures correlations both across rows and across columns of $\M_k$.

\subsection{AdamW-style Shampoo as AdamW in the Singular Value Geometry}

\citet{Li-2025-nips} established, for AdamW, the convergence rate 
\begin{eqnarray}
\begin{aligned}\notag
\frac{1}{K}\sum_{k=1}^K\E\left[\left\|\nabla f(\x_k)\right\|_1\right]\leq\bO\left(\frac{\sqrt{d}}{K^{1/4}}\sqrt[4]{\sigma^2L(f(\x_1)-f^*)}+\sqrt{\frac{dL(f(\x_1)-f^*)}{K}}\right)
\end{aligned}
\end{eqnarray} 
while ensuring $\lambda\|\x_k\|_{\infty}<1$ for all $k=1,2,\cdots,K$, where $d$ is the dimension. Comparing with (\ref{rate3}), we observe that AdamW-style Shampoo employs the nuclear norm in its convergence rate and spectral norm in its implicit bias, which correspond respectively to the $\ell_1$ and $\ell_{\infty}$ norms of the singular values, respectively. Consequently, AdamW-style Shampoo can be interpreted as exhibiting theoretical behavior analogous to that of AdamW, but in the space of singular values. 

\subsection{AdamW-style Shampoo as Muon with STR Preconditioning}

Let $\U_k\Sigma_k\V_k^T$ be the compact singular value decomposition (SVD) of $\M_k$. Using the identities $\left(\M_k\M_k^T\right)^{\frac{1}{2p}}=\U_k\Sigma_k^{\frac{1}{p}}\U_k^T$, $\left(\M_k^T\M_k\right)^{\frac{1}{2q}}=\V_k\Sigma_k^{\frac{1}{q}}\V_k^T$, and $\frac{1}{p}+\frac{1}{q}=1$, we can rewrite the update $\bL_{k,\varepsilon}^{-\frac{1}{2p}}\M_k\bR_{k,\varepsilon}^{-\frac{1}{2q}}$ in AdamW-style Shampoo as\footnote{Concurrent with our conference paper, we  noticed that \citet{Runa-shampoo-26} discovered a similar formula.} 
\begin{eqnarray}
\begin{aligned}\notag
\bL_{k,\varepsilon}^{-\frac{1}{2p}}\left(\M_k\M_k^T\right)^{\frac{1}{2p}}\U_k\V_k^T\left(\M_k^T\M_k\right)^{\frac{1}{2q}}\bR_{k,\varepsilon}^{-\frac{1}{2q}}.
\end{aligned}
\end{eqnarray} 
Informally, $\M_k$, $\bL_k$, and $\bR_k$ can be interpreted as approximations to the first moment matrix $\E\left[\G\right]$, the row-wise second raw moment matrix $\E\left[\G\G^T\right]$, and the column-wise second raw moment matrix $\E\left[\G^T\G\right]$, respectively. Noting the decomposition $\E\left[\G\G^T\right]=(\E[\G])(\E[\G])^T+\E\left[(\G-\E[\G])(\G-\E[\G])^T\right]$, the two quantities $\bL_{k,\varepsilon}^{-\frac{1}{2p}}\left(\M_k\M_k^T\right)^{\frac{1}{2p}}$ and $\left(\M_k^T\M_k\right)^{\frac{1}{2q}}\bR_{k,\varepsilon}^{-\frac{1}{2q}}$ can be interpreted as row-wise and column-wise signal-to-total-energy ratio (STR) matrices, respectively. This conceptually extends the scalar notation of the original Adam to matrices \citep{adam-15-iclr}. Consequently, AdamW-style Shampoo can be viewed as an STR preconditioned variant of Muon. This relationship is analogous to that between AdamW and SignSGD \citep{Orvieto-2025}.

\subsection{Restriction on the Weight Decay Parameter}

Our theory requires the weight decay parameter $\lambda$ to be sufficiently small. To illustrate why some restriction on $\lambda$ is necessary for convergence, we follow \citep{Li-2025-nips} to consider a simple stochastic convex problem
\begin{eqnarray}
\begin{aligned}\label{toy-example}
f(\X)=\frac{\|\X-\X^*\|_F^2}{200}\quad\mbox{with}\quad \X^*=\left[\begin{array}{cc}4,&4\\4,&4\end{array}\right].
\end{aligned}
\end{eqnarray}
The stochastic gradient oracle is given by
\begin{eqnarray}
\begin{aligned}\notag
\G(\X)=\left\{
\begin{array}{ll}
\X-\X^*-\A, & \mbox{with probability 0.1}, \\
-\frac{1}{10}(\X-\X^*-\frac{10}{9}\A), & \mbox{with probability 0.9}. \\\end{array}
\right.
\end{aligned}
\end{eqnarray}
We initialize $\X_1\hspace*{-0.06cm}=\hspace*{-0.06cm}\X^*\hspace*{-0.03cm}+\hspace*{-0.03cm}\left[\hspace*{-0.2cm}\begin{array}{cc}-1,&-3\\3,&1\end{array}\hspace*{-0.2cm}\right]$ and set $\A\hspace*{-0.06cm}=\hspace*{-0.06cm}\left[\hspace*{-0.2cm}\begin{array}{cc}10,&10\\10,&10\end{array}\hspace*{-0.2cm}\right]$, $K=10^{9}$, $\theta=1-\frac{1}{\sqrt{K}}$, $\beta=\sqrt{\theta}$, $\eta=\frac{1}{\sqrt{K}}$, $\varepsilon=10^{-12}$, $p=q=2$, $\M_0=\0$, $\bL_0=\0$, $\bR_0=\0$ for Algorithm \ref{alg1}. We test $\lambda=\{10^{-1},10^{-2},10^{-3},10^{-4},10^{-5},0\}$. For these choices, $\|\X^*\|_{op}=8<\frac{1}{\lambda}$ and $\|\X_1\|_{op}\leq9.2<\frac{1}{\lambda}$. Figure \ref{figure5} demonstrates that Algorithm \ref{alg1} fails to converge to $\X^*$ for $\lambda\in\{10^{-1},10^{-2},10^{-3},10^{-4}\}$. This indicates that, even for a simple convex problem, convergence to the minimum solution is not guaranteed if $\lambda$ exceeds a certain threshold, although the tightness of our upper bound on $\lambda$ remains unclear.

\begin{figure}
\hspace*{2cm}\begin{tabular}{@{\extracolsep{0.001em}}c@{\extracolsep{0.001em}}c}
   \includegraphics[width=0.36\linewidth,keepaspectratio]{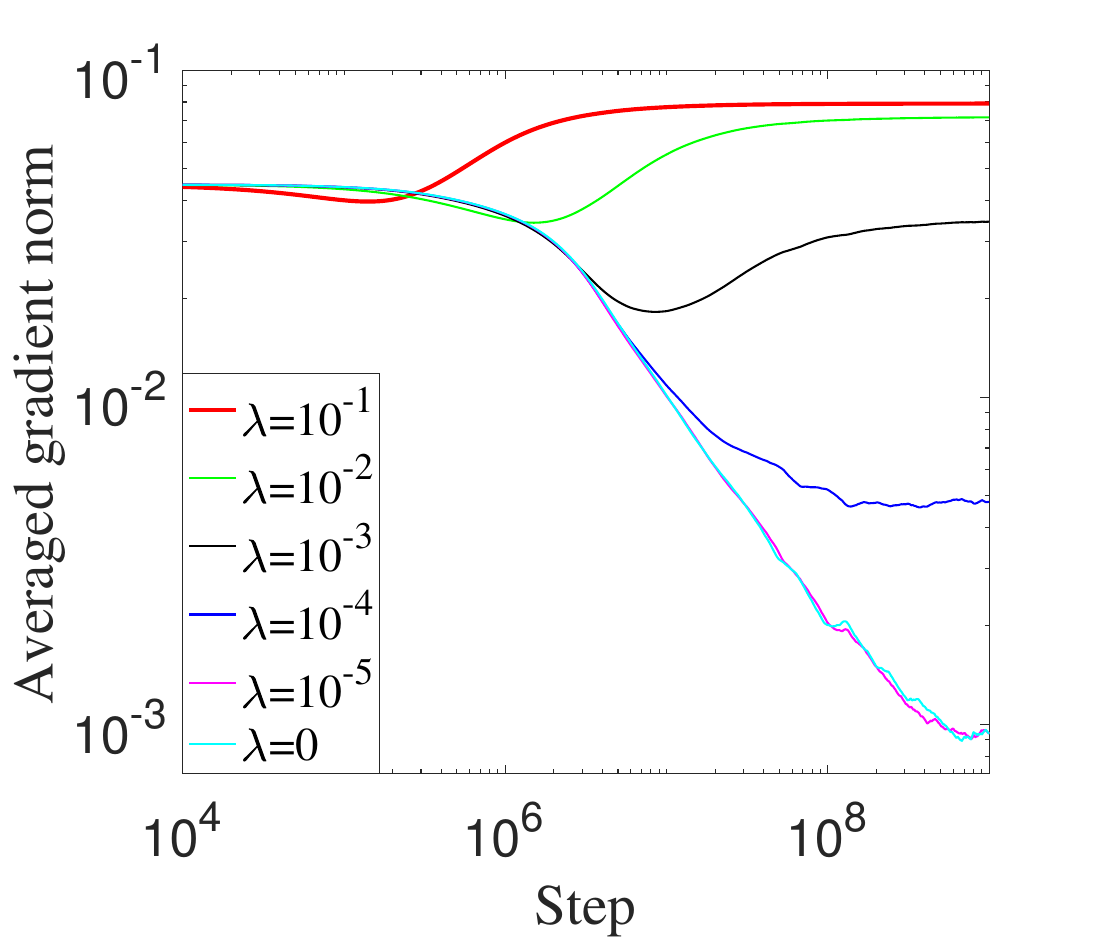}
   &\includegraphics[width=0.36\linewidth,keepaspectratio]{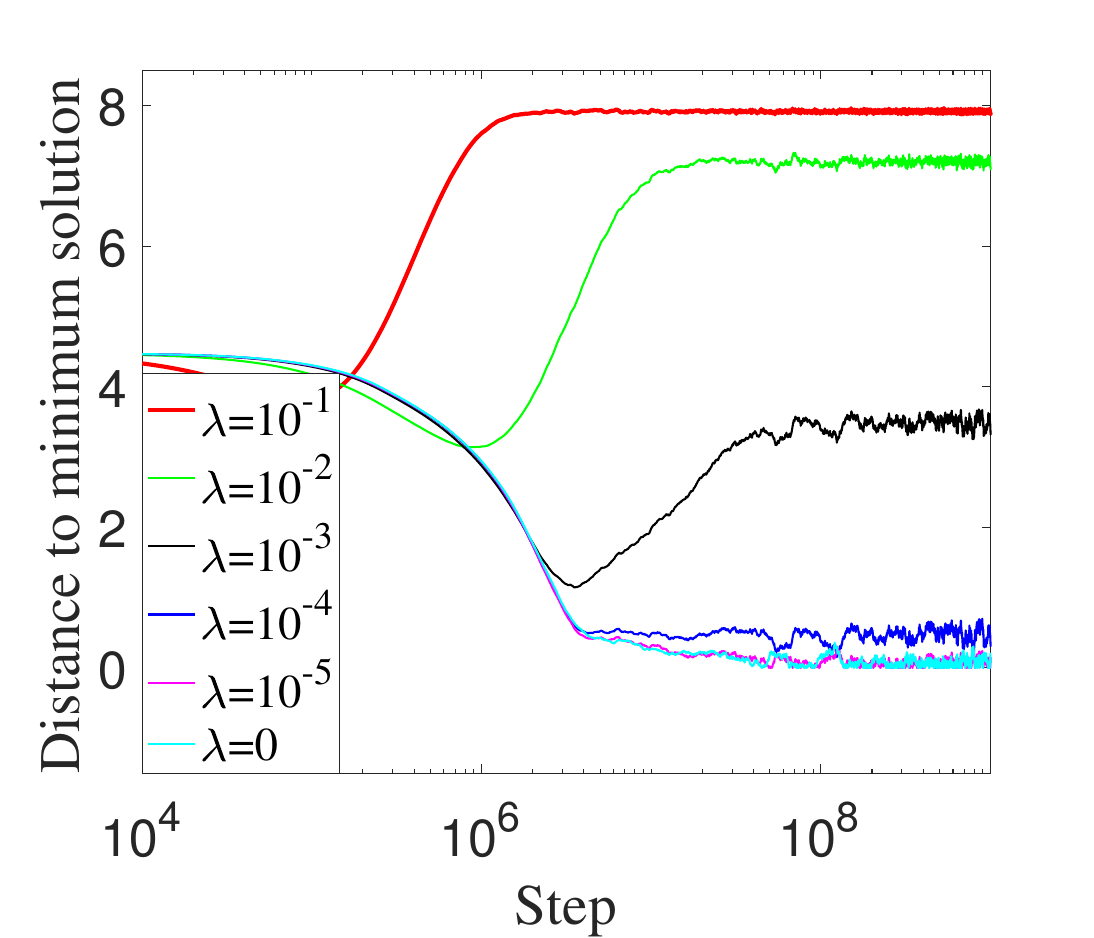}
   \end{tabular}
   \caption{Illustrations of $\frac{1}{k}\sum_{t=1}^k\|\nabla f(\X_t)\|_F$ (left) and $\|\X_k-\X^*\|_F$ (right) over steps on the toy example (\ref{toy-example}).}
   \label{figure5}
\end{figure}

\subsection{Comparison with \citep{shampoo-xie-25} and \citep{Gratton-2026}}

While \citet{shampoo-xie-25} established convergence for adaptive optimizers under a nonconvex, adaptive smoothness framework, our work differs in several key aspects. First, we provide a unified treatment of two-sided and one-sided Shampoo, addressing the more complex and widely used former variant, while \citet{shampoo-xie-25} only studied the latter. Second, we incorporate practical components such as momentum, decoupled weight decay, and tunable exponents satisfying $\frac{1}{2p}+\frac{1}{2q}=1-\frac{\omega}{2}$ for any $\omega\in(0,2]$, none of which are considered in \citep{shampoo-xie-25}. Third, we utilize a stochastic expectation-based assumption, which is weaker than the deterministic assumption $-\Sigma_L\preceq\G\G^T-\nabla f(\X)\nabla f(\X)^T\preceq\Sigma_L$ used by \citet{shampoo-xie-25}. Finally, while both analyses achieve the $\bO(\frac{C}{K^{1/4}})$ convergence rate, the constant $C$ in our convergence rate (\ref{rate3}) is simpler and matches the one appearing in the optimal convergence rate of SGD under the standard Euclidean smoothness assumption.

\citet{Gratton-2026}\footnote{This paper appeared online three months after the ICML submission and first appearance of our paper on arXiv.} provided a unified convergence theory for adaptive first-order methods, including AdaNorm, full-matrix AdaGrad, diagonal AdaGrad, Shampoo, and AdaGo \citep{AdaGo-2025}. However, at the algorithmic level, \citet{Gratton-2026} only considered the AdaGrad-style update of the preconditioners in Shampoo—rather than the more practical RMSProp-style—and they omitted decoupled weight decay. At the theoretical level, they relied on the stronger assumption $\E_k\left[\left\|\G_k-\nabla f(\X_k)\right\|_F^2|\F_{k-1}\right]\leq\frac{\sigma^2}{(k+1)^{\alpha}}+c\E_k\left[\left\|\bL_k^{-1/4}\G_k\bR_k^{-1/4}\right\|_F^2|\F_{k-1}\right]$ for some $\alpha,c>0$, and obtained a convergence rate of $\bO(\frac{1}{K^{\alpha/2}})$ for $\alpha<1$ and $\bO(\frac{1}{\sqrt{K}})$ when $\alpha\geq 1$, where the dependence on $d$ is hidden in the $\bO(\cdot)$ notation. Letting $\alpha\rightarrow0$ and $c\rightarrow0$ reduces this assumption to the one commonly used in nonconvex stochastic optimization, but under that limit their convergence rate $\bO(\frac{1}{K^{\alpha/2}})$ becomes extremely slow.

\section{Convergence Rate of Algorithm \ref{alg1} without Decoupled Weight Decay for Any $\omega\in(0,2]$}

In this section, we extend our analysis framework to the more general setting $\frac{1}{p}+\frac{1}{q}=2-\omega$ for any $\omega\in(0,2]$. When $\omega=2$, we have $p=q=+\infty$, and Algorithm \ref{alg1} reduces to momentum SGD. As $\omega$ decreases, the sum of the exponents $\frac{1}{2p}+\frac{1}{2q}$ approaches 1; in this regime, \citet{Runa-shampoo-26,distribute-shampoo-20,distribute-shampoo-23} empirically observed that larger exponents $\frac{1}{2p}$ and $\frac{1}{2q}$ yield better performance than the traditional choice $\frac{1}{2p}=\frac{1}{2q}=\frac{1}{4}$. The following theorem establishes the convergence rate for Algorithm \ref{alg1} without decoupled weight decay (that is, $\lambda=0$) for any $\omega\in(0,2]$.

\begin{theorem}\label{main-theorem2}
Suppose that Assumptions 1-3 and conditions $\bL_{k,\varepsilon}\succeq\hat\varepsilon\I_m$ and $\bR_{k,\varepsilon}\succeq\hat\varepsilon\I_n$ hold for some $\hat\varepsilon\geq\varepsilon$. Let $\hat\sigma^2=\max\left\{\sigma^2,\frac{L\left(f(\X_1)-f^*\right)}{K\gamma^2}\right\}$ with any $\gamma\in(0,1)$, $\frac{1}{p}+\frac{1}{q}=2-\omega$ with any $\omega\in(0,2]$, $1-\theta=\sqrt{\frac{L\left(f(\X_1)-f^*\right)}{K\hat\sigma^2}}$, $\eta=\hat\varepsilon^{1-\omega/2}\sqrt{\frac{f(\X_1)-f^*}{4LK\hat\sigma^2}}$, $\varepsilon=\frac{\tau\hat\sigma^2}{m+n}$ where $\tau\leq 1$ is the hyperparameter to control the small numerical value of $\varepsilon$ used in practice, and $\beta\in(0,1)$ be a constant independent of $K$. Then for Algorithm \ref{alg1} with $\lambda=0$, we have 
\begin{eqnarray}
\begin{aligned}\label{rate4}
&\frac{1}{K}\sum_{k=1}^K\E\left[\left\||\nabla f(\X_k)|^{\omega}\right\|_*\right]\\
\leq& \frac{20\hat\sigma^{\frac{3\omega}{2}-\frac{\omega^2}{2}}(m+n)^{(1-\frac{\omega}{2})^2}}{\omega \hat\varepsilon^{\frac{\omega}{2}-\frac{\omega^2}{4}}}\left(\frac{L\left(f(\X_1)-f^*\right)}{K}\right)^{\frac{\omega}{4}} + \frac{10\hat C}{\hat\varepsilon^{1-\frac{\omega}{2}}}\left(\frac{\hat\sigma^2L\left(f(\X_1)-f^*\right)}{K}\right)^{\frac{1}{2}},
\end{aligned}
\end{eqnarray}
where $\hat C=\left(\frac{2(1-\beta)^{\omega/2}}{1-\beta^{\omega/2}}\right)^{2/\omega-1}$ is a constant independent of $K$. In the worst case where only $\hat\varepsilon=\varepsilon$ is guaranteed, we have
\begin{eqnarray}
\begin{aligned}\label{rate5}
&\frac{1}{K}\hspace*{-0.06cm}\sum_{k=1}^K\hspace*{-0.06cm}\E\hspace*{-0.06cm}\left[\left\||\nabla f(\X_k)|^{\omega}\right\|_*\right]\hspace*{-0.06cm}\leq\hspace*{-0.06cm} \frac{20(m\hspace*{-0.06cm}+\hspace*{-0.06cm}n)^{1\hspace*{-0.03cm}-\hspace*{-0.03cm}\frac{\omega}{2}}}{\omega\tau^{\frac{\omega}{2}-\frac{\omega^2}{4}}}\hspace*{-0.06cm}\max\hspace*{-0.06cm}\left\{\hspace*{-0.06cm}\sigma^{\frac{\omega}{2}}\hspace*{-0.06cm}\left(\hspace*{-0.06cm}\frac{L\hspace*{-0.06cm}\left(f(\X_1)\hspace*{-0.06cm}-\hspace*{-0.06cm}f^*\right)}{K}\hspace*{-0.06cm}\right)^{\hspace*{-0.06cm}\frac{\omega}{4}}\hspace*{-0.06cm},\hspace*{-0.06cm}\left(\hspace*{-0.06cm}\frac{L\hspace*{-0.06cm}\left(f(\X_1)\hspace*{-0.06cm}-\hspace*{-0.06cm}f^*\right)}{K\gamma}\right)^{\hspace*{-0.06cm}\frac{\omega}{2}}\hspace*{-0.06cm}\right\}\\
&\hspace*{1.2cm}+\hspace*{-0.06cm}\frac{10\hat C(m\hspace*{-0.06cm}+\hspace*{-0.06cm}n)^{1\hspace*{-0.03cm}-\hspace*{-0.03cm}\frac{\omega}{2}}}{\tau^{1-\frac{\omega}{2}}}\hspace*{-0.06cm}\left\{\hspace*{-0.06cm}\begin{array}{cc}
\max\hspace*{-0.06cm}\left\{\hspace*{-0.06cm}\sigma^{\omega-1}\hspace*{-0.06cm}\left(\hspace*{-0.06cm}\frac{L\left(f(\X_1)-f^*\right)}{K}\right)^{\hspace*{-0.06cm}\frac{1}{2}}\hspace*{-0.06cm},\hspace*{-0.06cm}\frac{1}{\gamma^{\omega-1}}\hspace*{-0.06cm}\left(\hspace*{-0.06cm}\frac{L\left(f(\X_1)-f^*\right)}{K}\right)^{\hspace*{-0.06cm}\frac{\omega}{2}}\hspace*{-0.06cm}\right\}\hspace*{-0.06cm}, & 1\hspace*{-0.06cm}\leq\hspace*{-0.06cm} \omega\hspace*{-0.06cm}\leq\hspace*{-0.06cm} 2,\\
\frac{1}{\gamma^{\omega-1}}\left(\frac{L\left(f(\X_1)-f^*\right)}{K}\right)^{\frac{\omega}{2}}, & 0\hspace*{-0.06cm}<\hspace*{-0.06cm}\omega\hspace*{-0.06cm}<\hspace*{-0.06cm} 1.
\end{array}\right.  \hspace*{-1cm}
\end{aligned}
\end{eqnarray}
\end{theorem}

Theorem \ref{main-theorem2} establishes the $\bO(\frac{1}{K^{\omega/4}})$ convergence rate for $\frac{1}{K}\sum_{k=1}^K\E\left[\left\||\nabla f(\X_k)|^{\omega}\right\|_*\right]$. When $\omega=1$, convergence rate (\ref{rate5}) becomes
\begin{eqnarray}
\begin{aligned}\notag
&\frac{1}{K}\hspace*{-0.05cm}\sum_{k=1}^K\hspace*{-0.05cm}\E\hspace*{-0.05cm}\left[\left\|\nabla f(\X_k)\right\|_*\right]
\hspace*{-0.05cm}\leq\hspace*{-0.05cm} 20\sqrt{m\hspace*{-0.05cm}+\hspace*{-0.05cm}n}\sqrt[4]{\frac{\sigma^2L\hspace*{-0.05cm}\left(f(\X_1)\hspace*{-0.05cm}-\hspace*{-0.05cm}f^*\right)}{K\tau}} \hspace*{-0.05cm}+\hspace*{-0.05cm} (20\hspace*{-0.05cm}+\hspace*{-0.05cm}10\hat C)\sqrt{\frac{(m\hspace*{-0.05cm}+\hspace*{-0.05cm}n)L\hspace*{-0.05cm}\left(f(\X_1)\hspace*{-0.05cm}-\hspace*{-0.05cm}f^*\right)}{ K\tau\gamma}},
\end{aligned}
\end{eqnarray}
which is the same order as rate (\ref{rate2}). When $\omega=2$, Algorithm \ref{alg1} reduces to momentum SGD and convergence rate (\ref{rate5}) becomes
\begin{eqnarray}
\begin{aligned}\notag
\frac{1}{K}\sum_{k=1}^K\E\left[\left\|\nabla f(\X_k)\right\|_F^2\right]
\leq (10+10\hat C)\max\left\{\sqrt{\frac{\sigma^2L\left(f(\X_1)-f^*\right)}{K}},\frac{L\left(f(\X_1)-f^*\right)}{ K\gamma}\right\},
\end{aligned}
\end{eqnarray}
since $\left\||\nabla f(\X_k)|^2\right\|_*=\left\|\nabla f(\X_k)\right\|_F^2$, which is the same as the rate of SGD. When $\omega<1$, the term $\frac{1}{K^{\omega/4}}$ is larger than $\frac{1}{K^{1/4}}$; when $1<\omega<2$, it becomes smaller. However, due to the presence of $\omega$ on the left-hand side of (\ref{rate5}), it is difficult to determine whether the convergence rate (\ref{rate5}) is faster or slower than rate (\ref{rate3}).

\begin{remark}
We explain why we do not consider decoupled weight decay for the setting $\omega\in(0,2]$. As discussed above, when $\omega=2$, Algorithm \ref{alg1} reduces to momentum SGD. With the momentum further removed, decoupled weight decay is equivalent to $\ell_2$ regularization, which minimizes problem
\begin{eqnarray}
\begin{aligned}\label{wd-p1}
f(\X)+\frac{\lambda}{2}\|\X\|_F^2.
\end{aligned}
\end{eqnarray}
When $\omega=1$, for the classical AdamW, \citet{xie-2024-adamw-icml} proved that if the iterates of AdamW converge to some $\x_{\infty}$, then $\x_{\infty}$ is a KKT point of the constrained problem 
\begin{eqnarray}
\begin{aligned}\notag
\min_{\x\in\mathbb R^d} f(\x),\quad \mbox{s.t.}\quad \|\x\|_{\infty}\leq\frac{1}{\lambda}.
\end{aligned}
\end{eqnarray}
Accordingly, by extending from the vector case to the matrix case, we conjecture that AdamW-style Shampoo solves problem
\begin{eqnarray}
\begin{aligned}\label{wd-p2}
\min_{\X\in\mathbb R^{m\times n}} f(\X),\quad \mbox{s.t.}\quad \|\X\|_{op}\leq\frac{1}{\lambda}
\end{aligned}
\end{eqnarray}
when $\omega=1$. Our proof framework unifies the settings for all $\omega\in(0,2]$. If we were to include decoupled weight decay, then, in addition to the technical challenges, we would need to unify the optimality conditions of problems (\ref{wd-p1}) and (\ref{wd-p2}), whether such a unification is possible remains unclear.
\end{remark}

\section{Proof of the Theorems}

In Section \ref{sec:basic}, we first recall several basic facts from matrix analysis that are used in our proof. We then present our techniques to address the challenging two-sided preconditioning in Sections \ref{sec-two-sided-preconditioning1} and \ref{sec-two-sided-preconditioning2}, which constitute the primary technical contribution of this paper relative to existing literature. Lemmas \ref{finite-lemma}, \ref{gradient-bound}, and \ref{second-moment-bound} in Section \ref{sec-two-sided-preconditioning1} hold for any $\omega\in(0,2]$, while Lemmas \ref{bound-X}, \ref{bound-update-lemma}, and \ref{spectral-norm-bound} in Section \ref{sec-two-sided-preconditioning2} only hold for $\omega=1$. Finally, we give the complete proofs of Theorems \ref{main-theorem} and \ref{main-theorem2} in Sections \ref{main-theorem-proof-sec} and \ref{sec:proof2}, respectively, with supporting lemmas given in Section \ref{sec:supprt-lemmas}.

\subsection{Basic Properties from Matrix Analysis}\label{sec:basic}

We first introduce some basic properties from matrix analysis. For matrices $\A,\B\in\bbR^{m\times n}$, it holds that
\begin{eqnarray}
\begin{aligned}\notag
&\<\A,\B\>=\sum_{i=1}^m\sum_{j=1}^n \A_{i,j}\B_{i,j}=\tr(\A^T\B),\\
&\|\A\|_F^2=\sum_{i=1}^m\sum_{j=1}^n \A_{i,j}^2=\tr(\A^T\A),\\
&\tr(\A^T\B)=\tr(\B\A^T).
\end{aligned}
\end{eqnarray}

\begin{lemma}\citep[Lemma V.1.5, Proposition V.1.6]{rajendra-book}\label{matrix-inequality-lemma}
For any matrices $\A\in\mathbb R^{m\times s}$ and $\X,\Y\in\mathbb R^{m\times m}$ with $\X\preceq\Y$, it holds that
\begin{eqnarray}
\begin{aligned}\label{matrix-inequality}
\A^T\X\A\preceq \A^T\Y\A,\quad\tr\left(\X\right)\leq \tr\left(\Y\right).
\end{aligned}
\end{eqnarray}
For symmetric positive definite matrices $\X,\Y\in\mathbb R^{m\times m}$ with $\X\preceq\Y$, it holds that
\begin{eqnarray}
\begin{aligned}\label{matrix-inequality2}
\X^{-1}\succeq\Y^{-1}.
\end{aligned}
\end{eqnarray}
\end{lemma}

\begin{fact}\label{PSD-eig-singular}
For a symmetric positive semidefinite matrix $\X\in\mathbb R^{m\times m}$, the singular values coincide with the eigenvalues and thus $\tr(\X)=\sum_{i=1}^m\sigma_i(\X)$.
\end{fact}

\begin{lemma}\label{matrix-root-inequality}\citep{Ando-1999}
For symmetric positive semidefinite matrices $\X,\Y\in\bbR^{m\times m}$ and $0<p\leq 1$, it holds that
\begin{eqnarray}
\begin{aligned}\notag
\tr\left(\left(\X+\Y\right)^p\right)\leq \tr\left(\X^p+\Y^p\right).
\end{aligned}
\end{eqnarray}
In particular, for $p=1/2$, we have $\tr\left(\left(\X+\Y\right)^{1/2}\right)\leq \tr\left(\X^{1/2}+\Y^{1/2}\right)$.
\end{lemma}

\begin{lemma}\label{trace-nuclear-equation}
For any matrix $\X\in\bbR^{m\times n}$, it holds that $\left\||\X|^{\omega}\right\|_*=\tr\left((\X\X^T)^{\omega/2}\right)$.
\end{lemma}
\begin{proof}
Let $\U\Sigma\V^T$ be the compact singular value decomposition (SVD) of $\X$ with $r=\mbox{rank}(\X)$, $\U\in\bbR^{m\times r}$, $\Sigma\in\bbR^{r\times r}$, and $\V\in\bbR^{n\times r}$. From (\ref{matrix-power}), we have $|\X|^{\omega}=\V\Sigma^{\omega}\V^T$. On the other hand, since $(\X\X^T)^{\omega/2}=(\U\Sigma^2\U^T)^{\omega/2}=\U\Sigma^{\omega}\U^T$, we have $\left\||\X|^{\omega}\right\|_*=\tr\left(\Sigma^{\omega}\right)=\tr\left((\X\X^T)^{\omega/2}\right)$.
\end{proof}
Since $x^p$ is operator monotonic and operator concave on $[0,\infty)$ for $p\in[0,1]$, we have the following two properties \citep[Theorem V.1.9, Theorem V.2.5]{rajendra-book}:
\begin{lemma}\label{matrix-monotonic}
For symmetric positive semidefinite matrices $\X,\Y\in\bbR^{m\times m}$ with $\X\preceq\Y$, it holds that $\X^p\preceq \Y^p$ with $p\in[0,1]$.
\end{lemma}
\begin{lemma}\label{matrix-concavity}
For symmetric positive semidefinite matrix $\X\in\mathbb R^{m\times m}$, it holds that $\E\left[\X^p\right]\preceq(\E[\X])^p$ with $p\in[0,1]$.
\end{lemma}
From Assumptions 2-3, we have
\begin{eqnarray}
\begin{aligned}\label{G-assump1}
\E_k\left[\left\|\G_k-\nabla f(\X_k)\right\|_F^2\big|\F_{k-1}\right]=&\E_k\left[\tr\left(\left(\G_k-\nabla f(\X_k)\right)\left(\G_k-\nabla f(\X_k)\right)^T\right)\big|\F_{k-1}\right]\\
=&\E_k\left[\tr\left(\left(\G_k-\nabla f(\X_k)\right)^T\left(\G_k-\nabla f(\X_k)\right)\right)\big|\F_{k-1}\right]\\
\leq& \frac{\tr\left(\Sigma_L\right)+\tr\left(\Sigma_R\right)}{2}\equiv\sigma^2,
\end{aligned}
\end{eqnarray}
and
\begin{eqnarray}
\begin{aligned}\label{G-assump2}
\Sigma_L\succeq&\E_k\left[\left(\G_k-\nabla f(\X_k)\right)\left(\G_k-\nabla f(\X_k)\right)^T\big|\F_{k-1}\right]\\
=&\E_k\left[\G_k\G_k^T+\nabla f(\X_k)\nabla f(\X_k)^T-\G_k\nabla f(\X_k)^T-\nabla f(\X_k)\G_k^T\big|\F_{k-1}\right]\\
=&\E_k\left[\G_k\G_k^T\big|\F_{k-1}\right]-\nabla f(\X_k)\nabla f(\X_k)^T.
\end{aligned}
\end{eqnarray}

\subsection{Bounding the Nuclear Norm of Gradient by Holder's Inequality in Singular Value Space}\label{sec-two-sided-preconditioning1}

In the analysis of AdamW, \citet{Li-2025-nips} employs the following inequality to measure the expected gradient $\ell_1$ norm summation
\begin{eqnarray}
\begin{aligned}\label{adam-inequality}
\left(\sum_{k=1}^K\E\left[\|\nabla f(\x_k)\|_1\right]\right)^2\leq\left(\sum_{k=1}^K\sum_{i=1}^d\E\left[\frac{\left|\nabla_i f(\x_k)\right|^2}{\sqrt{\widetilde\vv_{k,i}+\varepsilon}}\right]\right)\hspace*{-0.15cm}\left(\sum_{k=1}^K\sum_{i=1}^d\E\left[\sqrt{\widetilde\vv_{k,i}+\varepsilon}\right]\right)\hspace*{-0.5cm}
\end{aligned}
\end{eqnarray}
by Holder's inequality for some $\widetilde\vv_k$ to approximate the second moment, where $\widetilde\vv_{k,i}$ denotes the $i$-th element of vector $\widetilde\vv$ at the $k$-th iteration. To handle the more complex matrix case with two-sided preconditioning, we instead utilize the following lemma. 

\begin{lemma}\label{UIN-Holder-inequality}\citep[(IV.33), Exercise IV.2.7]{rajendra-book}
Let $\A_i\in\mathbb R^{m\times m}$ and $p_i$, $r$ be positive real numbers ($i=1,2,\cdots,t$) such that $\sum_{i=1}^t\frac{1}{p_i}=\frac{1}{r}$. Then
\begin{eqnarray}
\begin{aligned}\label{UIN-Holder}
\left\|\left|\Pi_{i=1}^t\A_i\right|^{r}\right\|_*^{1/r}\leq \Pi_{i=1}^t\left\|\left|\A_i\right|^{p_i}\right\|_*^{1/p_i}.
\end{aligned}
\end{eqnarray}
\end{lemma}
Recalling (\ref{matrix-power}) and the definition of nuclear norm, inequality (\ref{UIN-Holder}) can be equivalently rewritten as
\begin{eqnarray}
\begin{aligned}\notag
\left(\sum_{j=1}^m\Big(\sigma_j\left(\Pi_{i=1}^t\A_i\right)\Big)^r\right)^{1/r}\leq \Pi_{i=1}^t\left(\sum_{j=1}^m\Big(\sigma_j\left(\A_i\right)\Big)^{p_i}\right)^{1/p_i},
\end{aligned}
\end{eqnarray}
which can be regarded as Holder's inequality in the space of singular values. Note that Lemma \ref{UIN-Holder-inequality} holds for all $r>0$, not restricted to $r\geq 1$. This result extends readily to rectangular matrices, because we can always obtain a square matrix by appending zero columns to the right and zero rows to the bottom of the original matrix, and these appended zeros do not affect the non-zero singular values and matrix multiplication, for example,
\begin{eqnarray}
\begin{aligned}\notag
\left(\hspace*{-0.03cm}\begin{array}{cc}
\A_{m\times n} & \0_{m\times (d-n)}\\
\0_{(d-m)\times n} & \0_{(d-m)\times (d-n)}\\
\end{array}\hspace*{-0.03cm}\right)\left(\hspace*{-0.03cm}\begin{array}{cc}
\B_{n\times s} & \0_{n\times (d-s)}\\
\0_{(d-n)\times s} & \0_{(d-n)\times (d-s)}\\
\end{array}\hspace*{-0.03cm}\right)=\left(\hspace*{-0.03cm}\begin{array}{cc}
(\A\B)_{m\times s} & \0_{m\times (d-s)}\\
\0_{(d-m)\times s} & \0_{(d-m)\times (d-s)}\\
\end{array}\hspace*{-0.03cm}\right).
\end{aligned}
\end{eqnarray}
\begin{remark}
In our conference paper \citep{li-2026-adamw-shampoo-icml}, we use the Schatten-$p$ norm defined as
\begin{eqnarray}
\begin{aligned}\notag
\|\A\|_{S_p}=\left\{\begin{array}{cc}
\Big(\sum_{j=1}^m\left(\sigma_j(\A)\right)^p\Big)^{1/p}, & 1\leq p<\infty,\\
\sigma_1(\A), & p=\infty.
\end{array}\right.
\end{aligned}
\end{eqnarray}
It is straightforward to verify that $\|\A\|_{S_p} = \||\A|^p\|_*^{1/p}$. However, for $0<p<1$, the quantity $\|\A\|_{S_p}$ does not satisfy the triangle inequality and is therefore not a norm, although Schatten-$p$ Holder's inequality used in our conference paper also holds for $0<p<1$. To avoid ambiguity when dealing with this regime, we follow \citep{rajendra-book} to use the notation $\||\A|^p\|_*^{1/p}$ in this paper.
\end{remark}
Since Holder's inequality $\sum_{i} x_iy_i\leq \left(\sum_{i} x_i^p\right)^{1/p}\left(\sum_{i} y_i^q\right)^{1/q}$ requires that both $\sum_{i} x_i^p$ and $\sum_{i} y_i^q$ be finite, we first establish the following lemma to guarantee that Holder's inequality is valid, although Lemma \ref{finite-lemma} appears intuitively obvious.
\begin{lemma}\label{finite-lemma}
Suppose that Assumptions 1-3 hold. Then for Algorithm \ref{alg1}, $\nabla f(\X_k)$, $\bL_k$, and $\bR_k$ are almost surely finite (that is, the absolute value of every entry is $<+\infty$ with probability one) for all $k=1,2,\cdots,K$. We also have that $\E\left[\left\|\bL_{k,\varepsilon}^{-\frac{1}{4p}}\nabla f(\X_k)\bR_{k,\varepsilon}^{-\frac{1}{4q}}\right\|_F^2\right]$, $\E\left[\tr\left(\bL_k^{\omega/2}\right)\right]$, and $\E\left[\tr\left(\bR_k^{\omega/2}\right)\right]$ are finite (that is, $<+\infty$).
\end{lemma}
\begin{proof}
We first prove by induction that $\X_k$, $\nabla f(\X_k)$, $\M_{k-1}$, $\bL_{k-1}$, and $\bR_{k-1}$ are almost surely finite for all $k=1,2,\cdots,K+1$. The base case $k=1$ holds trivially. Assume the statement is true for $k=t$, where $t>1$. Taking expectation with respect to $\F_{k-1}$ on the bounded-variance condition (\ref{G-assump1}), we know that $\left\|\G_t-\nabla f(\X_t)\right\|_F^2$ is almost surely finite, otherwise $P(\left\|\G_t-\nabla f(\X_t)\right\|_F^2=+\infty)>0$ would imply $\E\left[\left\|\G_t-\nabla f(\X_t)\right\|_F^2\right]=+\infty$, contradicting (\ref{G-assump1}). Hence $\G_t-\nabla f(\X_t)$ is almost surely finite. By the induction hypothesis $\nabla f(\X_k)$ is almost surely finite, so $\G_t$ is also almost surely finite. Using the update rules for $\M_t$, $\bL_t$ and $\bR_t$ in Algorithm \ref{alg1} together with the induction hypothesis, we conclude that $\M_t$, $\bL_t$, and $\bR_t$ are almost surely finite. From condition (\ref{L-assump}) and property (\ref{matrix-inequality2}), we have $\bL_{k,\varepsilon}^{-1}\preceq\frac{1}{\hat\varepsilon}\I_m$ and $\bR_{k,\varepsilon}^{-1}\preceq\frac{1}{\hat\varepsilon}\I_n$. Since $\frac{1}{2p}\leq 1$ and $\frac{1}{2q}\leq 1$, the matrix functions $\X^{1/2p}$ and $\X^{1/2q}$ are operator monotonic (Lemma \ref{matrix-monotonic}), which yields $\bL_{k,\varepsilon}^{-1/2p}\preceq\frac{1}{\hat\varepsilon^{1/2p}}\I_m$ and $\bR_{k,\varepsilon}^{-1/2q}\preceq\frac{1}{\hat\varepsilon^{1/2q}}\I_n$. From the update of $\X_{t+1}$ in Algorithm \ref{alg1}, and the fact that $\bL_{k,\varepsilon}^{-1/2p}\preceq\frac{1}{\hat\varepsilon^{1/2p}}\I_m$, $\bR_{k,\varepsilon}^{-1/2q}\preceq\frac{1}{\hat\varepsilon^{1/2q}}\I_n$, and $\M_k$ is almost surely finite, it follows that $\X_{t+1}$ is almost surely finite. From Assumption 1, we have $\|\nabla f(\X_{t+1})\|_F\leq\|\nabla f(\X_{t+1})-\nabla f(\0)\|_F+\|\nabla f(\0)\|_F\leq L\|\X_{t+1}\|_F+\|\nabla f(\0)\|_F$, and thus $\|\nabla f(\X_{t+1})\|_F$ is almost surely finite. Hence the claim holds for $k=t+1$. By induction, we know that $\X_k$, $\nabla f(\X_k)$, $\M_{k-1}$, $\bL_{k-1}$, and $\bR_{k-1}$ are almost surely finite for all $k=1,2,\cdots,K+1$.

Next, we prove by induction that $\E\left[\|\X_k\|_F^2\right]$, $\E\left[\|\nabla f(\X_k)\|_F^2\right]$, $\E\left[\|\M_{k-1}\|_F^2\right]$, $\E\left[\tr\left(\bL_{k-1}\right)\right]$, and $\E\left[\tr\left(\bR_{k-1}\right)\right]$ are finite for all $k=1,2,\cdots,K+1$. The base case holds trivially. Assume that all the above quantities are finite for $k=t$. From the bounded-variance condition (\ref{G-assump1}) and the inequality $\E\left[\|\G_t\|_F^2\right]\leq 2\E\left[\|\G_t-\nabla f(\X_t)\|_F^2\right]+2\E\left[\|\nabla f(\X_t)\|_F^2\right]$ together with the induction hypothesis, we deduce that $\E\left[\|\G_t\|_F^2\right]$ is finite. The update rules for $\M_t$, $\bL_t$ and $\bR_t$ in Algorithm \ref{alg1} give
\begin{eqnarray}
\begin{aligned}\notag
&\E\left[\|\M_t\|_F^2\right]\leq 2\theta^2\E\left[\|\M_{t-1}\|_F^2\right] + 2(1-\theta)^2\E\left\|\G_t\|_F^2\right],\\
&\E\left[\tr\left(\bL_t\right)\right]=\beta \E\left[\tr\left(\bL_{t-1}\right)\right]+(1-\beta)\E\left[\|\G_t\|_F^2\right],\\
&\E\left[\tr\left(\bR_t\right)\right]=\beta \E\left[\tr\left(\bR_{t-1}\right)\right]+(1-\beta)\E\left[\|\G_t\|_F^2\right].
\end{aligned}
\end{eqnarray}
By the induction hypothesis, $\E\left[\|\M_t\|_F^2\right]$, $\E\left[\tr\left(\bL_t\right)\right]$, $\E\left[\tr\left(\bR_t\right)\right]$ are finite. For the iterate $\X_{t+1}$, the update in Algorithm \ref{alg1} yields
\begin{eqnarray}
\begin{aligned}\notag
\E\left[\|\X_{t+1}\|_F^2\right]\leq& 2(1-\lambda\eta)^2\E\left[\|\X_t\|_F^2\right]+2\eta^2\E\left[\left\|\bL_{k,\varepsilon}^{-\frac{1}{2p}}\M_k\bR_{k,\varepsilon}^{-\frac{1}{2q}}\right\|_F^2\right]\\
\leq& 2(1-\lambda\eta)^2\E\left[\|\X_t\|_F^2\right]+\frac{2\eta^2}{\hat\varepsilon^{\frac{1}{p}+\frac{1}{q}}}\E\left[\|\M_k\|_F^2\right],
\end{aligned}
\end{eqnarray}
where the second inequality follows from the estimate
\begin{eqnarray}
\begin{aligned}\label{equ14}
\left\|\bL_{k,\varepsilon}^{-\frac{1}{2p}}\M_k\bR_{k,\varepsilon}^{-\frac{1}{2q}}\right\|_F^2=&\tr\left(\bR_{k,\varepsilon}^{-\frac{1}{2q}}\M_k^T\bL_{k,\varepsilon}^{-\frac{1}{p}}\M_k\bR_{k,\varepsilon}^{-\frac{1}{2q}}\right)\leq\frac{1}{\hat\varepsilon^{\frac{1}{p}}}\tr\left(\bR_{k,\varepsilon}^{-\frac{1}{2q}}\M_k^T\M_k\bR_{k,\varepsilon}^{-\frac{1}{2q}}\right)\\
=&\frac{1}{\hat\varepsilon^{\frac{1}{p}}}\tr\left(\M_k\bR_{k,\varepsilon}^{-\frac{1}{q}}\M_k^T\right)\leq\frac{1}{\hat\varepsilon^{\frac{1}{p}+\frac{1}{q}}}\tr\left(\M_k\M_k^T\right)=\frac{1}{\hat\varepsilon^{\frac{1}{p}+\frac{1}{q}}}\left\|\M_k\right\|_F^2,
\end{aligned}
\end{eqnarray}
which uses property (\ref{matrix-inequality}), $\bL_{k,\varepsilon}^{-1/p}\preceq\frac{1}{\hat\varepsilon^{1/p}}\I_m$, and $\bR_{k,\varepsilon}^{-1/q}\preceq\frac{1}{\hat\varepsilon^{1/q}}\I_n$. So $\E\left[\|\X_{t+1}\|_F^2\right]$ is finite. From Assumption 1, we have 
\begin{eqnarray}
\begin{aligned}\notag
\E\left[\|\nabla f(\X_{t+1})\|_F^2\right]\leq& 2\E\left[\|\nabla f(\X_{t+1})-\nabla f(\0)\|_F^2\right]+2\|\nabla f(\0)\|_F^2\\
\leq& 2L^2\E\left[\|\X_{t+1}\|_F^2\right]+2\|\nabla f(\0)\|_F^2,
\end{aligned}
\end{eqnarray}
and thus $\E\left[\|\nabla f(\X_{t+1})\|_F^2\right]$ is also finite. Thus the claim holds for $k=t+1$. Similar to (\ref{equ14}), we also have that $\E\left[\left\|\bL_{k,\varepsilon}^{-\frac{1}{4p}}\nabla f(\X_k)\bR_{k,\varepsilon}^{-\frac{1}{4q}}\right\|_F^2\right]$ is finite as well. 

Finally, we turn to $\E\left[\tr\left(\bL_k^{\omega/2}\right)\right]$ and $\E\left[\tr\left(\bR_k^{\omega/2}\right)\right]$. The two inequalities are analogous, and we only prove the first. For a symmetric positive semidefinite matrix $\A$, we have $\mbox{det}\left(\begin{array}{cc}
\A_{i,i} & \A_{i,j}\\
\A_{j,i} & \A_{j,j}\\
\end{array}\right)=\A_{i,i}\A_{j,j}-\A_{i,j}\A_{j,i}=\A_{i,i}\A_{j,j}-\A_{i,j}^2\geq 0$ and thus $|\A_{i,j}|\leq \sqrt{\A_{i,i}\A_{j,j}}\leq \tr(\A)$. Using that $\E\left[\tr(\bL_k)\right]$ is finite, we conclude that $\E\left[\bL_k\right]$ is finite. Since the function $\X^{\omega/2}$ is operator concave for $0<\omega/2\leq 1$ (Lemma \ref{matrix-concavity}), we have $\E\left[\bL_k^{\omega/2}\right]\preceq \left(\E\left[\bL_k\right]\right)^{\omega/2}$. Hence both $\E\left[\bL_k^{\omega/2}\right]$ and $\E\left[\tr\left(\bL_k^{\omega/2}\right)\right]$ are finite.
\end{proof}

With Lemma \ref{finite-lemma} in place, we can use Lemma \ref{UIN-Holder-inequality} and Holder's inequality to generalize (\ref{adam-inequality}) from vectors to matrices.
\begin{lemma}\label{gradient-bound}
Suppose that Assumptions 1-3 hold. Let $\frac{1}{p}+\frac{1}{q}=2-\omega$ with $0<\omega\leq 2$. Then for Algorithm \ref{alg1}, we have
\begin{eqnarray}
\begin{aligned}\notag
&\sum_{k=1}^K\E\left[\left\||\nabla f(\X_k)|^{\omega}\right\|_*\right]\\
\leq&\left(\sum_{k=1}^K\E\left[\tr\left(\bL_{k,\varepsilon}^{\omega/2}\right)+\tr\left(\bR_{k,\varepsilon}^{\omega/2}\right)\right]\right)^{1-\omega/2}\left(\sum_{k=1}^K\E\left[\left\|\bL_{k,\varepsilon}^{-\frac{1}{4p}}\nabla f(\X_k)\bR_{k,\varepsilon}^{-\frac{1}{4q}}\right\|_F^2\right]\right)^{\omega/2}.
\end{aligned}
\end{eqnarray}
\end{lemma}
\begin{proof}
From Lemma \ref{UIN-Holder-inequality} and $\frac{1}{2p\omega}+\frac{1}{2q\omega}+\frac{1}{2}=\frac{1}{\omega}$, we have
\begin{eqnarray}
\begin{aligned}\label{equ11}
\left\||\nabla f(\X_k)|^{\omega}\right\|_*^{1/\omega}=&\left\|\left|\bL_{k,\varepsilon}^{\frac{1}{4p}}\bL_{k,\varepsilon}^{-\frac{1}{4p}}\nabla f(\X_k)\bR_{k,\varepsilon}^{-\frac{1}{4q}}\bR_{k,\varepsilon}^{\frac{1}{4q}}\right|^{\omega}\right\|_*^{1/\omega}\\
\leq&\left\|\left|\bL_{k,\varepsilon}^{\frac{1}{4p}}\right|^{2p\omega}\right\|_*^{1/2p\omega}\left\|\left|\bL_{k,\varepsilon}^{-\frac{1}{4p}}\nabla f(\X_k)\bR_{k,\varepsilon}^{-\frac{1}{4q}}\right|^2\right\|_*^{1/2}\left\|\left|\bR_{k,\varepsilon}^{\frac{1}{4q}}\right|^{2q\omega}\right\|_*^{1/2q\omega}.
\end{aligned}
\end{eqnarray}
From Holder's inequality and $\frac{1}{2p}+\frac{1}{2q}+\frac{\omega}{2}=1$, we have
\begin{eqnarray}
\begin{aligned}\label{equ12}
&\sum_{k=1}^K\E\left[\left\||\nabla f(\X_k)|^{\omega}\right\|_*\right]\\
\leq&\sum_{k=1}^K\E\left[\left\|\left|\bL_{k,\varepsilon}^{\frac{1}{4p}}\right|^{2p\omega}\right\|_*^{1/2p}\left\|\left|\bL_{k,\varepsilon}^{-\frac{1}{4p}}\nabla f(\X_k)\bR_{k,\varepsilon}^{-\frac{1}{4q}}\right|^2\right\|_*^{\omega/2}\left\|\left|\bR_{k,\varepsilon}^{\frac{1}{4q}}\right|^{2q\omega}\right\|_*^{1/2q}\right]\\
\leq&\hspace*{-0.08cm}\left(\sum_{k=1}^K\hspace*{-0.05cm}\E\hspace*{-0.08cm}\left[\left\|\left|\bL_{k,\varepsilon}^{\frac{1}{4p}}\right|^{2p\omega}\right\|_*\right]\hspace*{-0.05cm}\right)^{\hspace*{-0.12cm}1/{2p}}\hspace*{-0.15cm}\left(\sum_{k=1}^K\hspace*{-0.05cm}\E\hspace*{-0.08cm}\left[\left\|\left|\bL_{k,\varepsilon}^{-\frac{1}{4p}}\nabla f(\X_k)\bR_{k,\varepsilon}^{-\frac{1}{4q}}\right|^2\right\|_*\right]\hspace*{-0.05cm}\right)^{\hspace*{-0.12cm}\omega/2}\hspace*{-0.15cm}\left(\sum_{k=1}^K\hspace*{-0.05cm}\E\hspace*{-0.08cm}\left[\left\|\left|\bR_{k,\varepsilon}^{\frac{1}{4q}}\right|^{2q\omega}\right\|_*\right]\hspace*{-0.05cm}\right)^{\hspace*{-0.12cm}1/{2q}}\hspace*{-0.1cm}.\hspace*{-0.8cm}
\end{aligned}
\end{eqnarray}
Denote $\A=\bL_{k,\varepsilon}^{-\frac{1}{4p}}\nabla f(\X_k)\bR_{k,\varepsilon}^{-\frac{1}{4q}}\in\mathbb R^{m\times n}$ with $r=\min\{m,n\}$. Recalling (\ref{matrix-power}), we have
\begin{eqnarray}
\begin{aligned}\notag
\left\||\A|^2\right\|_*=\sum_{i=1}^r \left(\sigma_i(\A)\right)^2=\tr(\A^T\A)=\|\A\|_F^2.
\end{aligned}
\end{eqnarray}
When $p,q<+\infty$, from (\ref{matrix-power}), $\sigma_i(\B^{\frac{1}{p}})=\left(\sigma_i(\B)\right)^{\frac{1}{p}}$ for any symmetric positive semidefinite matrix $\B$, the fact that $\bL_{k,\varepsilon}$ is symmetric positive definite, and Fact \ref{PSD-eig-singular}, we have
\begin{eqnarray}
\begin{aligned}\label{Schatten-p-bound}
\left\|\left|\bL_{k,\varepsilon}^{\frac{1}{4p}}\right|^{2p\omega}\right\|_*=\sum_{i=1}^m \left(\sigma_i\left(\bL_{k,\varepsilon}^{\frac{1}{4p}}\right)\right)^{2p\omega}=\tr\left(\bL_{k,\varepsilon}^{\omega/2}\right).
\end{aligned}
\end{eqnarray}
The derivation for $\bR_{k,\varepsilon}$ follows a similar approach. Finally, using $a^{\frac{1}{2p}}b^{\frac{1}{2q}}\leq(a+b)^{\frac{1}{2p}}(a+b)^{\frac{1}{2q}}=(a+b)^{1-\omega/2}$ for positive $a,b$, we have the desired bound. From Lemma \ref{finite-lemma} and the preceding argument, we know Holder's inequality is valid in (\ref{equ11}) and (\ref{equ12}). 

When $p=+\infty$ and $q=\frac{1}{2-\omega}$, we have $\bL_{k,\varepsilon}^{\frac{1}{4p}}=\I_m$ and $\bL_{k,\varepsilon}^{-\frac{1}{4p}}=\I_m$. Similar to the above analysis, we have
\begin{eqnarray}
\hspace*{-3.4cm}\begin{aligned}\notag
&\sum_{k=1}^K\E\left[\left\||\nabla f(\X_k)|^{\omega}\right\|_*\right]=\sum_{k=1}^K\E\left[\left\|\left|\nabla f(\X_k)\bR_{k,\varepsilon}^{-\frac{1}{4q}}\bR_{k,\varepsilon}^{\frac{1}{4q}}\right|^{\omega}\right\|_*\right]\\
\leq&\sum_{k=1}^K\E\left[\left\|\left|\nabla f(\X_k)\bR_{k,\varepsilon}^{-\frac{1}{4q}}\right|^2\right\|_*^{\omega/2}\left\|\left|\bR_{k,\varepsilon}^{\frac{1}{4q}}\right|^{2q\omega}\right\|_*^{\frac{1}{2q}}\right]
\end{aligned}
\end{eqnarray}
\begin{eqnarray}
\begin{aligned}\notag
\leq&\left(\sum_{k=1}^K\E\left[\left\|\left|\nabla f(\X_k)\bR_{k,\varepsilon}^{-\frac{1}{4q}}\right|^2\right\|_*\right]\right)^{\omega/2}\left(\sum_{k=1}^K\E\left[\left\|\left|\bR_{k,\varepsilon}^{\frac{1}{4q}}\right|^{2q\omega}\right\|_*\right]\right)^{1/{2q}}\\
=&\left(\sum_{k=1}^K\E\left[\left\|\nabla f(\X_k)\bR_{k,\varepsilon}^{-\frac{1}{4q}}\right\|_F^2\right]\right)^{\omega/2}\left(\sum_{k=1}^K\E\left[\tr\left(\bR_{k,\varepsilon}^{\omega/2}\right)\right]\right)^{1-\omega/2}\\
\leq&\left(\sum_{k=1}^K\E\left[\left\|\bL_{k,\varepsilon}^{-\frac{1}{4p}}\nabla f(\X_k)\bR_{k,\varepsilon}^{-\frac{1}{4q}}\right\|_F^2\right]\right)^{\omega/2}\left(\sum_{k=1}^K\E\left[\tr\left(\bL_{k,\varepsilon}^{\omega/2}\right)+\tr\left(\bR_{k,\varepsilon}^{\omega/2}\right)\right]\right)^{1-\omega/2},
\end{aligned}
\end{eqnarray}
where we use $\frac{1}{2q\omega}+\frac{1}{2}=\frac{1}{\omega}$ and the fact that $\bL_{k,\varepsilon}$ is symmetric positive definite. The case when $q=+\infty$ and $p=\frac{1}{2-\omega}$ is similar. Finally, the case when $p=+\infty$, $q=+\infty$, and $\omega=2$ holds trivially.
\end{proof} 

Next, we bound $\sum_{k=1}^K\E\left[\tr\left(\bL_{k,\varepsilon}^{\omega/2}\right)\right]$ and $\sum_{k=1}^K\E\left[\tr\left(\bR_{k,\varepsilon}^{\omega/2}\right)\right]$ in the following lemma.
\begin{lemma}\label{second-moment-bound}
Suppose that Assumptions 2-3 hold. Let $\beta< 1$ and $0<\omega\leq 2$. Then for Algorithm \ref{alg1}, we have
\begin{eqnarray}
\begin{aligned}\notag
\sum_{k=1}^K\E_{\F_k}\left[\tr\left(\bL_{k,\varepsilon}^{\omega/2}\right)\right]\leq K\tr\left(\Sigma_L^{\omega/2}\right)+Km\varepsilon^{\omega/2}+\frac{(1-\beta)^{\omega/2}}{1-\beta^{\omega/2}}\sum_{t=1}^K\E_{\F_{t-1}}\left[\left\||\nabla f(\X_t)|^{\omega}\right\|_*\right],
\end{aligned}
\end{eqnarray}
and
\begin{eqnarray}
\begin{aligned}\notag
\sum_{k=1}^K\E_{\F_k}\left[\tr\left(\bR_{k,\varepsilon}^{\omega/2}\right)\right]\leq K\tr\left(\Sigma_R^{\omega/2}\right)+Kn\varepsilon^{\omega/2}+\frac{(1-\beta)^{\omega/2}}{1-\beta^{\omega/2}}\sum_{t=1}^K\E_{\F_{t-1}}\left[\left\||\nabla f(\X_t)|^{\omega}\right\|_*\right].
\end{aligned}
\end{eqnarray}
\end{lemma}
\begin{proof}
From the recursion of $\bL_{k-t}$, we have
\begin{eqnarray}
\begin{aligned}\notag
&\E_{\F_{k-t}}\left[\tr\left(\left(\beta^t \bL_{k-t}+(1-\beta^t)\Sigma_L+\varepsilon\I_m\right)^{\omega/2}\right)\right]\\
=&\E_{\F_{k-t}}\left[\tr\left(\left(\beta^{t+1}\bL_{k-t-1}+\beta^t(1-\beta)\G_{k-t}\G_{k-t}^T+(1-\beta^t)\Sigma_L+\varepsilon\I_m\right)^{\omega/2}\right)\right]\\
=&\E_{\F_{k-t-1}}\hspace*{-0.1cm}\left[\E_{k-t}\left[\tr\left(\left(\beta^{t+1}\bL_{k-t-1}+\beta^t(1-\beta)\G_{k-t}\G_{k-t}^T+(1-\beta^t)\Sigma_L+\varepsilon\I_m\right)^{\omega/2}\right)\big|\F_{k-t-1}\right]\right]\\
\overset{a}\leq&\E_{\F_{k-t-1}}\hspace*{-0.1cm}\left[\tr\hspace*{-0.08cm}\left(\hspace*{-0.08cm}\left(\beta^{t+1}\bL_{k-t-1}\hspace*{-0.08cm}+\hspace*{-0.08cm}\beta^t(1\hspace*{-0.08cm}-\hspace*{-0.08cm}\beta)\E_{k-t}\left[\G_{k-t}\G_{k-t}^T\big|\F_{k-t-1}\right]+(1-\beta^t)\Sigma_L+\varepsilon\I_m\right)^{\omega/2}\right)\right]\\
\overset{b}\leq&\E_{\F_{k-t-1}}\hspace*{-0.1cm}\left[\tr\hspace*{-0.08cm}\left(\hspace*{-0.08cm}\left(\beta^{t+1}\bL_{k-t-1}\hspace*{-0.08cm}+\hspace*{-0.08cm}\beta^t\hspace*{-0.05cm}(\hspace*{-0.05cm}1\hspace*{-0.08cm}-\hspace*{-0.08cm}\beta)\hspace*{-0.05cm}\nabla\hspace*{-0.08cm} f\hspace*{-0.02cm}(\X_{k-t})\nabla\hspace*{-0.08cm}  f\hspace*{-0.02cm}(\X_{k-t})^T\hspace*{-0.1cm}+\hspace*{-0.08cm}\beta^t(\hspace*{-0.05cm}1\hspace*{-0.08cm}-\hspace*{-0.08cm}\beta\hspace*{-0.05cm})\Sigma_L\hspace*{-0.08cm}+\hspace*{-0.08cm}(1\hspace*{-0.08cm}-\hspace*{-0.08cm}\beta^t)\Sigma_L\hspace*{-0.08cm}+\hspace*{-0.08cm}\varepsilon\I_m\right)^{\hspace*{-0.05cm}\omega/2}\right)\right]\\
=&\E_{\F_{k-t-1}}\hspace*{-0.1cm}\left[\tr\hspace*{-0.08cm}\left(\hspace*{-0.08cm}\left(\beta^{t+1}\bL_{k-t-1}\hspace*{-0.08cm}+\hspace*{-0.08cm}\beta^t\hspace*{-0.05cm}(\hspace*{-0.05cm}1\hspace*{-0.08cm}-\hspace*{-0.08cm}\beta)\hspace*{-0.05cm}\nabla\hspace*{-0.08cm} f\hspace*{-0.02cm}(\X_{k-t})\nabla\hspace*{-0.08cm}  f\hspace*{-0.02cm}(\X_{k-t})^T\hspace*{-0.08cm}+\hspace*{-0.08cm}(1\hspace*{-0.08cm}-\hspace*{-0.08cm}\beta^{t+1})\Sigma_L+\varepsilon\I_m\right)^{\omega/2}\right)\right]\\
\overset{c}\leq&\E_{\F_{k-t-1}}\hspace*{-0.1cm}\left[\tr\hspace*{-0.08cm}\left(\hspace*{-0.08cm}\left(\beta^{t+1}\bL_{k-t-1}\hspace*{-0.08cm}+\hspace*{-0.08cm}(\hspace*{-0.05cm}1\hspace*{-0.1cm}-\hspace*{-0.1cm}\beta^{t+1}\hspace*{-0.05cm})\Sigma_L\hspace*{-0.1cm}+\hspace*{-0.08cm}\varepsilon\I_m\hspace*{-0.05cm}\right)^{\hspace*{-0.03cm}\omega\hspace*{-0.02cm}/\hspace*{-0.02cm}2}\hspace*{-0.05cm}\right)\hspace*{-0.1cm}+\hspace*{-0.12cm}\left(\hspace*{-0.04cm}\beta^t\hspace*{-0.03cm}(\hspace*{-0.05cm}1\hspace*{-0.09cm}-\hspace*{-0.09cm}\beta)\hspace*{-0.04cm}\right)^{\hspace*{-0.03cm}\omega\hspace*{-0.02cm}/\hspace*{-0.02cm}2}\hspace*{-0.08cm}\tr\hspace*{-0.1cm}\left(\hspace*{-0.1cm}\left(\nabla\hspace*{-0.08cm} f\hspace*{-0.03cm}(\X_{k-t})\nabla\hspace*{-0.08cm} f\hspace*{-0.03cm}(\X_{k-t})^T\right)^{\hspace*{-0.03cm}\omega\hspace*{-0.02cm}/\hspace*{-0.02cm}2}\hspace*{-0.05cm}\right)\hspace*{-0.08cm}\right]\\
\overset{d}=&\E_{\F_{k-t-1}}\hspace*{-0.1cm}\left[\tr\hspace*{-0.08cm}\left(\hspace*{-0.08cm}\left(\beta^{t+1}\bL_{k-t-1}\hspace*{-0.08cm}+\hspace*{-0.08cm}(1\hspace*{-0.08cm}-\hspace*{-0.08cm}\beta^{t+1})\Sigma_L\hspace*{-0.08cm}+\hspace*{-0.08cm}\varepsilon\I_m\right)^{\omega/2}\right)+\left(\beta^t(1-\beta)\right)^{\omega/2}\left\||\nabla f(\X_{k-t})|^{\omega}\right\|_*\right],
\end{aligned}
\end{eqnarray}
where we use the concavity of $\X^{\omega/2}$ with $\omega\leq 2$ presented in Lemma \ref{matrix-concavity} and $\tr(\X)\leq\tr(\Y)$ if $\X\preceq\Y$ presented in Lemma \ref{matrix-inequality-lemma} in $\overset{a}\leq$, (\ref{G-assump2}) and the monotonicity of $\X^{\omega/2}$ presented in Lemma \ref{matrix-monotonic} in $\overset{b}\leq$, and the property $\tr\left((\X+\Y)^{\omega/2}\right)\leq\tr(\X^{\omega/2})+\tr(\Y^{\omega/2})$ for symmetric positive semidefinite matrices with $\omega\in(0,2]$ presented in Lemma \ref{matrix-root-inequality} in $\overset{c}\leq$, and Lemma \ref{trace-nuclear-equation} in $\overset{d}=$. 

Applying the above inequality recursively for $t=0,1,2,\cdots,k-1$, we have
\begin{eqnarray}
\begin{aligned}\notag
&\E_{\F_k}\left[\tr\left(\bL_{k,\varepsilon}^{\omega/2}\right)\right]\\
\leq&\tr\left(\left(\beta^k\bL_0+(1-\beta^k)\Sigma_L+\varepsilon\I_m\right)^{\omega/2}\right)+(1-\beta)^{\omega/2}\sum_{t=0}^{k-1}\beta^{\frac{\omega t}{2}}\E_{\F_{k-t-1}}\left[\left\||\nabla f(\X_{k-t})|^{\omega}\right\|_*\right]\\
\overset{e}\leq&\tr\left(\Sigma_L^{\omega/2}+\varepsilon^{\omega/2}\I_m\right)+(1-\beta)^{\omega/2}\sum_{t=1}^k\beta^{\frac{\omega (k-t)}{2}}\E_{\F_{t-1}}\left[\left\||\nabla f(\X_t)|^{\omega}\right\|_*\right]\\
=&\tr\left(\Sigma_L^{\omega/2}\right)+m\varepsilon^{\omega/2}+(1-\beta)^{\omega/2}\sum_{t=1}^k\beta^{\frac{\omega (k-t)}{2}}\E_{\F_{t-1}}\left[\left\||\nabla f(\X_t)|^{\omega}\right\|_*\right],
\end{aligned}
\end{eqnarray}
where we use $\bL_0=\0$, monotonicity of $\X^{\omega/2}$, and $\tr\left((\X+\Y)^{\omega/2}\right)\leq\tr(\X^{\omega/2})+\tr(\Y^{\omega/2})$ in $\overset{e}\leq$. Summing over $k=1,2,\cdots,K$, we have
\begin{eqnarray}
\begin{aligned}\notag
\sum_{k=1}^K\hspace*{-0.07cm}\E_{\F_k}\hspace*{-0.1cm}\left[\tr\hspace*{-0.07cm}\left(\bL_{k,\varepsilon}^{\omega/2}\right)\right]\hspace*{-0.1cm}\leq& K\tr\hspace*{-0.07cm}\left(\hspace*{-0.07cm}\Sigma_L^{\omega/2}\right)\hspace*{-0.07cm}+\hspace*{-0.07cm}Km\varepsilon^{\omega/2}\hspace*{-0.07cm}+\hspace*{-0.07cm}(1-\beta)^{\omega/2}\sum_{k=1}^K\sum_{t=1}^k\beta^{\frac{\omega (k-t)}{2}}\E_{\F_{t-1}}\hspace*{-0.07cm}\left[\left\||\nabla f(\X_t)|^{\omega}\right\|_*\right]\\
=& K\tr\hspace*{-0.07cm}\left(\hspace*{-0.07cm}\Sigma_L^{\omega/2}\right)\hspace*{-0.07cm}+\hspace*{-0.07cm}Km\varepsilon^{\omega/2}\hspace*{-0.07cm}+\hspace*{-0.07cm}(1-\beta)^{\omega/2}\sum_{t=1}^K\sum_{k=t}^K\beta^{\frac{\omega (k-t)}{2}}\E_{\F_{t-1}}\hspace*{-0.07cm}\left[\left\||\nabla f(\X_t)|^{\omega}\right\|_*\right]\\
\leq& K\tr\hspace*{-0.07cm}\left(\hspace*{-0.07cm}\Sigma_L^{\omega/2}\right)\hspace*{-0.07cm}+\hspace*{-0.07cm}Km\varepsilon^{\omega/2}\hspace*{-0.07cm}+\hspace*{-0.07cm}\frac{(1-\beta)^{\omega/2}}{1-\beta^{\omega/2}}\sum_{t=1}^K\E_{\F_{t-1}}\hspace*{-0.07cm}\left[\left\||\nabla f(\X_t)|^{\omega}\right\|_*\right].
\end{aligned}
\end{eqnarray}
Similarly, we also have the bound for $\sum_{k=1}^K\E\left[\tr\left(\bR_{k,\varepsilon}^{\omega/2}\right)\right]$.
\end{proof} 
Combining Lemmas \ref{gradient-bound} and \ref{second-moment-bound}, we finally have
\begin{eqnarray}
\begin{aligned}\label{gradient-bound-final}
&\sum_{k=1}^K\E\left[\left\||\nabla f(\X_k)|^{\omega}\right\|_*\right]\\
\leq&\hspace*{-0.07cm}\left(\hspace*{-0.07cm}KC\hspace*{-0.07cm}+\hspace*{-0.07cm}\frac{2(1-\beta)^{\omega/2}}{1-\beta^{\omega/2}}\sum_{t=1}^K\E\left[\left\||\nabla f(\X_t)|^{\omega}\right\|_*\right]\right)^{1-\omega/2}\hspace*{-0.07cm}\left(\sum_{k=1}^K\E\hspace*{-0.07cm}\left[\left\|\bL_{k,\varepsilon}^{-\frac{1}{4p}}\nabla f(\X_k)\bR_{k,\varepsilon}^{-\frac{1}{4q}}\right\|_F^2\right]\right)^{\omega/2}\hspace*{-0.07cm},\hspace*{-0.07cm}
\end{aligned}
\end{eqnarray} 
where $C=\tr\left(\Sigma_L^{\omega/2}\right)+\tr\left(\Sigma_R^{\omega/2}\right)+(m+n)\varepsilon^{\omega/2}$. It therefore remains to bound the last term.

Beyond bounding the gradient, we also use Holder's inequality in the space of singular values to prove the following lemma, which is then employed to handle weight decay.
\begin{lemma}\label{bound-X}
Let $\frac{1}{p}+\frac{1}{q}=1$. Then for Algorithm \ref{alg1}, we have
\begin{eqnarray}
\begin{aligned}\notag
\left\|\bL_{k,\varepsilon}^{\frac{1}{4p}}\X_k\bR_{k,\varepsilon}^{\frac{1}{4q}}\right\|_F^2\leq\|\X_k\|_{op}^2 \left(\tr\left(\bL_{k,\varepsilon}^{1/2}\right)+\tr\left(\bR_{k,\varepsilon}^{1/2}\right)\right). 
\end{aligned}
\end{eqnarray}
\end{lemma}
\begin{proof}
We first consider the case of $p,q<+\infty$. Based on basic matrix analysis, we have
\begin{eqnarray}
\begin{aligned}\label{equ-m1}
\left\|\bL_{k,\varepsilon}^{\frac{1}{4p}}\X_k\bR_{k,\varepsilon}^{\frac{1}{4q}}\right\|_F^2=&\tr\left(\bL_{k,\varepsilon}^{\frac{1}{4p}}\X_k\bR_{k,\varepsilon}^{\frac{1}{2q}}\X_k^T\bL_{k,\varepsilon}^{\frac{1}{4p}}\right)\overset{a}=\left\|\left|\bL_{k,\varepsilon}^{\frac{1}{4p}}\X_k\bR_{k,\varepsilon}^{\frac{1}{2q}}\X_k^T\bL_{k,\varepsilon}^{\frac{1}{4p}}\right|\right\|_*\\
\overset{b}\leq&\left\|\left|\bL_{k,\varepsilon}^{\frac{1}{4p}}\right|^{2p}\right\|_*^{1/2p}\left\|\left|\X_k\bR_{k,\varepsilon}^{\frac{1}{2q}}\X_k^T\right|^q\right\|_*^{1/q}\left\|\left|\bL_{k,\varepsilon}^{\frac{1}{4p}}\right|^{2p}\right\|_*^{1/2p}\\
\overset{c}=&\left(\tr\left(\bL_{k,\varepsilon}^{1/2}\right)\right)^{\frac{1}{p}}\left\|\left|\X_k\bR_{k,\varepsilon}^{\frac{1}{2q}}\X_k^T\right|^q\right\|_*^{1/q},
\end{aligned}
\end{eqnarray}
where we use Fact \ref{PSD-eig-singular}, the fact that $\bL_{k,\varepsilon}^{\frac{1}{4p}}\X_k\bR_{k,\varepsilon}^{\frac{1}{2q}}\X_k^T\bL_{k,\varepsilon}^{\frac{1}{4p}}$ is symmetric positive semidefinite, and (\ref{matrix-power}) in $\overset{a}=$, $\frac{1}{p}+\frac{1}{q}=1$ and Lemma \ref{UIN-Holder-inequality} in $\overset{b}\leq$, and (\ref{Schatten-p-bound}) with $\omega=1$ in $\overset{c}=$. Denoting $r$ to be the rank of $\X_k\bR_{k,\varepsilon}^{\frac{1}{2q}}\X_k^T$, we have
\begin{eqnarray}
\begin{aligned}\notag
\left\|\left|\X_k\bR_{k,\varepsilon}^{\frac{1}{2q}}\X_k^T\right|^q\right\|_*^{1/q}\overset{d}=&\left(\sum_{i=1}^r\left(\sigma_i\left(\X_k\bR_{k,\varepsilon}^{\frac{1}{2q}}\X_k^T\right)\right)^{q}\right)^{\frac{1}{q}}\overset{e}\leq\left(\sum_{i=1}^n\left(\|\X_k\|_{op}^2\sigma_i\left(\bR_{k,\varepsilon}^{\frac{1}{2q}}\right)\right)^{q}\right)^{\frac{1}{q}}\\
\overset{f}=&\left(\|\X_k\|_{op}^{2q}\sum_{i=1}^n\sigma_i\left(\bR_{k,\varepsilon}^{1/2}\right)\right)^{\frac{1}{q}}\overset{g}=\|\X_k\|_{op}^{2}\left(\tr\left(\bR_{k,\varepsilon}^{1/2}\right)\right)^{\frac{1}{q}},
\end{aligned}
\end{eqnarray}
where we use (\ref{matrix-power}) in $\overset{d}=$, $r\leq\min\{m,n\}$ and the properties $\sigma_i(\A\B)\leq\sigma_i(\A)\|\B\|_{op}$ and $\sigma_i(\A\B)\leq\|\A\|_{op}\sigma_i(\B)$ of singular values in $\overset{e}\leq$, $\sigma_i(\B^{\frac{1}{q}})=\left(\sigma_i(\B)\right)^{\frac{1}{q}}$ for any symmetric positive semidefinite matrix $\B$ in $\overset{f}=$, and Fact \ref{PSD-eig-singular} in $\overset{g}=$. Plugging into (\ref{equ-m1}), we have
\begin{eqnarray}
\begin{aligned}\notag
\left\|\bL_{k,\varepsilon}^{\frac{1}{4p}}\X_k\bR_{k,\varepsilon}^{\frac{1}{4q}}\right\|_F^2\leq&\|\X_k\|_{op}^{2}\left(\tr\left(\bL_{k,\varepsilon}^{1/2}\right)\right)^{\frac{1}{p}}\left(\tr\left(\bR_{k,\varepsilon}^{1/2}\right)\right)^{\frac{1}{q}}\overset{h}\leq \|\X_k\|_{op}^2 \left(\tr\left(\bL_{k,\varepsilon}^{1/2}\right)+\tr\left(\bR_{k,\varepsilon}^{1/2}\right)\right), 
\end{aligned}
\end{eqnarray}
where we use $a^{\frac{1}{p}}b^{\frac{1}{q}}\leq (a+b)^{\frac{1}{p}+\frac{1}{q}}=a+b$ for positive $a,b$ in $\overset{h}\leq$. 

When $p=+\infty$ and $q=1$, we have $\bL_{k,\varepsilon}^{\frac{1}{4p}}=\I_m$ and $\bL_{k,\varepsilon}^{-\frac{1}{4p}}=\I_m$. Similar to above analysis, we have
\begin{eqnarray}
\begin{aligned}\notag
\left\|\bL_{k,\varepsilon}^{\frac{1}{4p}}\X_k\bR_{k,\varepsilon}^{\frac{1}{4q}}\right\|_F^2\hspace*{-0.04cm}=\hspace*{-0.04cm}\tr\hspace*{-0.04cm}\left(\hspace*{-0.04cm}\bR_{k,\varepsilon}^{1/4}\X_k^T\X_k\bR_{k,\varepsilon}^{1/4}\hspace*{-0.04cm}\right)\hspace*{-0.04cm}
\overset{i}\leq\hspace*{-0.04cm}\|\X_k\|_{op}^{2}\tr\hspace*{-0.04cm}\left(\hspace*{-0.04cm}\bR_{k,\varepsilon}^{1/2}\right)\hspace*{-0.04cm}\leq\hspace*{-0.04cm}\|\X_k\|_{op}^2\hspace*{-0.04cm} \left(\hspace*{-0.04cm}\tr\hspace*{-0.04cm}\left(\hspace*{-0.04cm}\bL_{k,\varepsilon}^{1/2}\right)\hspace*{-0.04cm}+\hspace*{-0.04cm}\tr\hspace*{-0.04cm}\left(\hspace*{-0.04cm}\bR_{k,\varepsilon}^{1/2}\right)\hspace*{-0.04cm}\right),
\end{aligned}
\end{eqnarray}
where we use Lemma \ref{matrix-inequality-lemma} in $\overset{i}\leq$. The case when $p=1$ and $q=+\infty$ is similar.

\end{proof}

\subsection{Bounding the Spectral Norm of the Update by Matrix Cauchy-Schwarz Inequality}\label{sec-two-sided-preconditioning2}

In the analysis of AdamW \citep{Li-2025-nips}, we can bound the update $\frac{|\m_{k,i}|}{\sqrt{\vv_{k,i}}}$ coordinatewise, where $\m$ and $\vv$ are the first and second moments, respectively. However, the matrix case is not as simple, especially with two-sided preconditioning. To address this challenge, we use the following matrix Cauchy-Schwarz inequality.
\begin{lemma}\citep[Corollary IX.5.3]{rajendra-book}\label{matrix-Schwarz}
For $\M\in\mathbb R^{m\times n}$ and symmetric positive definite matrices $\bL\in\mathbb R^{m\times m}$ and $\bR\in\mathbb R^{n\times n}$, $0\leq\alpha\leq1$, we have 
\begin{eqnarray}
\begin{aligned}\notag
\|\bL^{\alpha}\M\bR^{1-\alpha}\|_{op}\leq \|\bL\M\|_{op}^{\alpha}\|\M\bR\|_{op}^{1-\alpha}.
\end{aligned}
\end{eqnarray}
\end{lemma}

Based on the above lemma, we can bound the update in Algorithm \ref{alg1} measured by spectral norm. Note that Lemma \ref{bound-update-lemma} still holds even when $\bL_{k,\varepsilon}$ and $\bR_{k,\varepsilon}$ are ill-conditioned.
\begin{lemma}\label{bound-update-lemma}
Let $\theta\leq\beta\leq\sqrt{\theta}<1$ and $\frac{1}{p}+\frac{1}{q}=1$. Then for Algorithm \ref{alg1}, we have
\begin{eqnarray}
\begin{aligned}\notag
\left\|\bL_{k,\varepsilon}^{-\frac{1}{2p}}\M_k\bR_{k,\varepsilon}^{-\frac{1}{2q}}\right\|_{op}\leq 2.
\end{aligned}
\end{eqnarray}
\end{lemma}
\begin{proof}
From Lemma \ref{matrix-Schwarz} and $\frac{1}{p}+\frac{1}{q}=1$, we have
\begin{eqnarray}
\begin{aligned}\notag
\left\|\bL_{k,\varepsilon}^{-\frac{1}{2p}}\M_k\bR_{k,\varepsilon}^{-\frac{1}{2q}}\right\|_{op}\leq \left\|\bL_{k,\varepsilon}^{-\frac{1}{2}}\M_k\right\|_{op}^{\frac{1}{p}}\left\|\M_k\bR_{k,\varepsilon}^{-\frac{1}{2}}\right\|_{op}^{\frac{1}{q}}.
\end{aligned}
\end{eqnarray}
So we only need to prove
\begin{eqnarray}
\begin{aligned}\notag \left\|\bL_{k,\varepsilon}^{-\frac{1}{2}}\M_k\right\|_{op}\leq 2\quad\mbox{and}\quad \left\|\M_k\bR_{k,\varepsilon}^{-\frac{1}{2}}\right\|_{op}\leq 2.
\end{aligned}
\end{eqnarray}
The two inequalities are analogous, so we prove only the first. Since $\left\|\bL_{k,\varepsilon}^{-\frac{1}{2}}\M_k\right\|_{op}^2=\left\|\bL_{k,\varepsilon}^{-\frac{1}{2}}\M_k\M_k^T\bL_{k,\varepsilon}^{-\frac{1}{2}}\right\|_{op}$, we only need to prove
\begin{eqnarray}
\begin{aligned}\notag
\bL_{k,\varepsilon}^{-\frac{1}{2}}\M_k\M_k^T\bL_{k,\varepsilon}^{-\frac{1}{2}} \preceq 4\I_m.
\end{aligned}
\end{eqnarray}
Since $\bL_{k,\varepsilon}$ is invertible, the above inequality is equivalent to
\begin{eqnarray}
\begin{aligned}\notag
\M_k\M_k^T \preceq 4\bL_{k,\varepsilon} \quad\mbox{and}\quad
\y^T\M_k\M_k^T\y \leq 4\y^T\bL_{k,\varepsilon}\y,\hspace*{0.03cm}\forall \y\in\mathbb R^m.
\end{aligned}
\end{eqnarray}
From the recursions of $\M_k$ and $\bL_k$, we have
\begin{eqnarray}
\begin{aligned}\notag
&\y^T\M_k= (1-\theta)\sum_{t=1}^k\theta^{k-t}\y^T\G_t,\\
&\y^T\bL_k\y= (1-\beta)\sum_{t=1}^k\beta^{k-t}\y^T\G_t\G_t^T\y = (1-\beta)\sum_{t=1}^k\beta^{k-t}\left\|\y^T\G_t\right\|^2.
\end{aligned}
\end{eqnarray}
From Holder's inequality, we have
\begin{eqnarray}
\begin{aligned}\notag
\y^T\M_k\M_k^T\y =& (1-\theta)^2\left\|\sum_{t=1}^k\theta^{k-t}\y^T\G_t \right\|^2 \leq (1-\theta)^2\left(\sum_{t=1}^k\theta^{k-t}\left\|\y^T\G_t \right\|\right)^2 \\
\leq& (1-\theta)^2\left(\sum_{t=1}^k\beta^{k-t}\left\|\y^T\G_t\right\|^2 \right)\left(\sum_{t=1}^k\left(\frac{\theta^2}{\beta}\right)^{k-t}\right)\\
=&\frac{(1-\theta)^2}{1-\beta}\y^T\bL_k\y\left(\sum_{t=1}^k\left(\frac{\theta^2}{\beta}\right)^{k-t}\right)\\
\leq& \frac{(1-\theta)^2}{1-\beta}\frac{1}{1-\frac{\theta^2}{\beta}}\y^T\bL_k\y\overset{a}\leq \frac{(1-\theta)^2}{(1-\beta)^2}\y^T\bL_k\y \\
\overset{b}\leq& \frac{(1-\sqrt{\theta})^2(1+\sqrt{\theta})^2}{(1-\sqrt{\theta})^2}\y^T\bL_k\y\leq 4\y^T\bL_k\y\leq 4\y^T\bL_{k,\varepsilon}\y,
\end{aligned}
\end{eqnarray}
where we use $\theta\leq\beta$ in $\overset{a}\leq$ and $\beta\leq\sqrt{\theta}$ in $\overset{b}\leq$.
\end{proof}
\begin{remark}
The requirement $\beta\leq\sqrt{\theta}$ excludes the common setting $(\theta,\beta)=(0.9,0.999)$. In fact, we can relax this condition to $\beta\leq\theta^{1/128}$, noting that $0.9^{1/128}\geq 0.999$. In this case, the step marked $\overset{b}\leq$ should be replaced by $\frac{(1-\theta)^2}{(1-\beta)^2}\leq\frac{(1-\theta^{1/2^7})^2\Pi_{r=1}^7(1+\theta^{1/2^r})^2}{(1-\theta^{1/2^7})^2}\leq 4^7$, at the cost of a larger constant term in the convergence rate.
\end{remark}

By leveraging Lemma \ref{bound-update-lemma} and the distinct properties of decoupled weight decay, we ultimately establish the following lemma, which extends \citep[Lemma 3]{Li-2025-nips} but replaces the $\ell_{\infty}$ norm of vectors by the spectral norm of matrices.
\begin{lemma}\label{spectral-norm-bound}
Let $\eta\lambda\leq \frac{\sqrt{\nu}}{2K^{5/4}}$, $\|\X_1\|_{op}\leq \frac{\sqrt{\nu}}{K^{1/4}\lambda}$, $\frac{\sqrt{\nu}}{K^{1/4}}\leq 1$, $\theta\leq\beta\leq\sqrt{\theta}<1$, and $\frac{1}{p}+\frac{1}{q}=1$ for some constant $\nu$. Then for Algorithm \ref{alg1}, we have
\begin{eqnarray}
\begin{aligned}\label{x-bound}
\lambda\|\X_k\|_{op}\leq \frac{3\sqrt{\nu}}{K^{1/4}},\quad \forall k=1,2,\cdots,K.
\end{aligned}
\end{eqnarray}
\end{lemma} 
\begin{proof}
When $\lambda=0$, (\ref{x-bound}) trivially holds, so we only consider $\lambda\neq 0$. From the update of $\X_{k+1}$, we have
\begin{eqnarray}
\begin{aligned}\notag
\|\X_{k+1}\|_{op}-\frac{2}{\lambda}=&\left\|(1-\lambda\eta)\X_k - \eta\bL_{k,\varepsilon}^{-\frac{1}{2p}}\M_k\bR_{k,\varepsilon}^{-\frac{1}{2q}}\right\|_{op}-\frac{2}{\lambda}\\
\leq&(1-\lambda\eta)\left\|\X_k\right\|_{op} + \eta\left\|\bL_{k,\varepsilon}^{-\frac{1}{2p}}\M_k\bR_{k,\varepsilon}^{-\frac{1}{2q}}\right\|_{op}-\frac{2}{\lambda}\\
\overset{a}\leq&(1-\lambda\eta)\left\|\X_k\right\|_{op} + 2\eta - \frac{2}{\lambda}=(1-\lambda\eta)\left(\left\|\X_k\right\|_{op}-\frac{2}{\lambda}\right)\\
\leq&(1-\lambda\eta)^k\left(\left\|\X_1\right\|_{op}-\frac{2}{\lambda}\right)\leq-\frac{1}{\lambda}(1-\lambda\eta)^k\left(2-\frac{\sqrt{\nu}}{K^{1/4}}\right),
\end{aligned}
\end{eqnarray}
where we use Lemma \ref{bound-update-lemma} in $\overset{a}\leq$. Since $\ln x\leq x-1$ and $e^x\geq x+1$ for any $x>0$ and $\eta\lambda\leq\frac{\sqrt{\nu}}{2K^{5/4}}\leq\frac{1}{2}$, we have for any $k\leq K$ that
\begin{eqnarray}
\begin{aligned}\notag
k\ln(1-\eta\lambda)=-k\ln\frac{1}{1-\eta\lambda}\geq -K\left(\frac{1}{1-\eta\lambda}-1\right)=-\frac{K\eta\lambda}{1-\eta\lambda}\geq -\frac{\sqrt{\nu}}{K^{1/4}},
\end{aligned}
\end{eqnarray}
\begin{eqnarray}
\begin{aligned}\notag
(1-\eta\lambda)^k\geq e^{-\frac{\sqrt{\nu}}{K^{1/4}}}\geq 1-\frac{\sqrt{\nu}}{K^{1/4}},
\end{aligned}
\end{eqnarray}
and
\begin{eqnarray}
\begin{aligned}\notag
\|\X_{k+1}\|_{op}-\frac{2}{\lambda}\leq& -\frac{1}{\lambda}\left(1-\frac{\sqrt{\nu}}{K^{1/4}}\right)\left(2-\frac{\sqrt{\nu}}{K^{1/4}}\right)\leq-\frac{2}{\lambda}+\frac{3}{\lambda}\frac{\sqrt{\nu}}{K^{1/4}}.
\end{aligned}
\end{eqnarray}
\end{proof}

\subsection{Proof of Theorem \ref{main-theorem}}\label{main-theorem-proof-sec}
With the supporting lemmas in Sections \ref{sec-two-sided-preconditioning1}, \ref{sec-two-sided-preconditioning2}, and \ref{sec:supprt-lemmas}, we can prove Theorem \ref{main-theorem}.
\begin{proof}
As the gradient is $L$-Lipschitz, we have
\begin{eqnarray}
\begin{aligned}\label{equ1}
&\E_k\left[f(\X_{k+1})\big|\F_{k-1}\right]-f(\X_k)\\
\leq& \E_k\left[\<\nabla f(\X_k),\X_{k+1}-\X_k\>+\frac{L}{2}\|\X_{k+1}-\X_k\|_F^2 \big|\F_{k-1}\right]\\
=& \E_k\left[-\eta\<\nabla f(\X_k),\lambda\X_k+\bL_{k,\varepsilon}^{-\frac{1}{2p}}\M_k\bR_{k,\varepsilon}^{-\frac{1}{2q}}\>+\frac{L\eta^2}{2}\left\|\lambda\X_k+\bL_{k,\varepsilon}^{-\frac{1}{2p}}\M_k\bR_{k,\varepsilon}^{-\frac{1}{2q}}\right\|_F^2 \big|\F_{k-1}\right]\\
=& \E_k\left[-\eta\<\bL_{k,\varepsilon}^{-\frac{1}{4p}}\nabla f(\X_k)\bR_{k,\varepsilon}^{-\frac{1}{4q}},\lambda\bL_{k,\varepsilon}^{\frac{1}{4p}}\X_k\bR_{k,\varepsilon}^{\frac{1}{4q}}+\bL_{k,\varepsilon}^{-\frac{1}{4p}}\M_k\bR_{k,\varepsilon}^{-\frac{1}{4q}}\>\right.\\
&\qquad\left.+\frac{L\eta^2}{2}\left\|\lambda\X_k+\bL_{k,\varepsilon}^{-\frac{1}{2p}}\M_k\bR_{k,\varepsilon}^{-\frac{1}{2q}}\right\|_F^2 \big|\F_{k-1}\right]\\
=& \E_k\left[-\frac{\eta}{2}\left\|\bL_{k,\varepsilon}^{-\frac{1}{4p}}\nabla f(\X_k)\bR_{k,\varepsilon}^{-\frac{1}{4q}}\right\|_F^2 - \frac{\eta}{2}\left\|\lambda\bL_{k,\varepsilon}^{\frac{1}{4p}}\X_k\bR_{k,\varepsilon}^{\frac{1}{4q}}+\bL_{k,\varepsilon}^{-\frac{1}{4p}}\M_k\bR_{k,\varepsilon}^{-\frac{1}{4q}}\right\|_F^2\right.\\
&\quad\left. + \frac{\eta}{2}\left\|\bL_{k,\varepsilon}^{-\frac{1}{4p}}\hspace*{-0.08cm}\left(\nabla f(\X_k)\hspace*{-0.08cm}-\hspace*{-0.08cm}\M_k\right)\hspace*{-0.08cm}\bR_{k,\varepsilon}^{-\frac{1}{4q}}\hspace*{-0.08cm}-\hspace*{-0.08cm}\lambda\bL_{k,\varepsilon}^{\frac{1}{4p}}\X_k\bR_{k,\varepsilon}^{\frac{1}{4q}}\right\|_F^2 \hspace*{-0.08cm}+\hspace*{-0.08cm}\frac{L\eta^2}{2}\hspace*{-0.08cm}\left\|\lambda\X_k\hspace*{-0.08cm}+\hspace*{-0.08cm}\bL_{k,\varepsilon}^{-\frac{1}{2p}}\M_k\bR_{k,\varepsilon}^{-\frac{1}{2q}}\right\|_F^2 \big|\F_{k-1}\right]\hspace*{-0.75cm}\\
\leq& \E_k\left[-\frac{\eta}{2}\left\|\bL_{k,\varepsilon}^{-\frac{1}{4p}}\nabla f(\X_k)\bR_{k,\varepsilon}^{-\frac{1}{4q}}\right\|_F^2 \hspace*{-0.08cm}-\hspace*{-0.08cm} \frac{\eta}{2}\left\|\lambda\bL_{k,\varepsilon}^{\frac{1}{4p}}\X_k\bR_{k,\varepsilon}^{\frac{1}{4q}}\hspace*{-0.08cm}+\hspace*{-0.08cm}\bL_{k,\varepsilon}^{-\frac{1}{4p}}\M_k\bR_{k,\varepsilon}^{-\frac{1}{4q}}\right\|_F^2\hspace*{-0.08cm}+\hspace*{-0.08cm}\eta\lambda^2\hspace*{-0.05cm}\underbrace{\left\|\bL_{k,\varepsilon}^{\frac{1}{4p}}\X_k\bR_{k,\varepsilon}^{\frac{1}{4q}}\right\|_F^2}_{\text{\rm term (a)}}\right.\hspace*{-1cm}\\
&\qquad\left. + \eta\underbrace{\left\|\bL_{k,\varepsilon}^{-\frac{1}{4p}}\left(\nabla f(\X_k)-\M_k\right)\bR_{k,\varepsilon}^{-\frac{1}{4q}}\right\|_F^2}_{\text{\rm term (b)}} +\frac{L\eta^2}{2}\underbrace{\left\|\lambda\X_k+\bL_{k,\varepsilon}^{-\frac{1}{2p}}\M_k\bR_{k,\varepsilon}^{-\frac{1}{2q}}\right\|_F^2}_{\text{\rm term (c)}} \big|\F_{k-1}\right].
\end{aligned}
\end{eqnarray}
For term (a), from Lemma \ref{bound-X} and \ref{spectral-norm-bound}, we have
\begin{eqnarray}
\begin{aligned}\notag
\left\|\bL_{k,\varepsilon}^{\frac{1}{4p}}\X_k\bR_{k,\varepsilon}^{\frac{1}{4q}}\right\|_F^2\leq&\|\X_k\|_{op}^2 \left(\tr\left(\bL_{k,\varepsilon}^{1/2}\right)+\tr\left(\bR_{k,\varepsilon}^{1/2}\right)\right)\\
\leq&\frac{9\nu}{\lambda^2K^{1/2}} \left(\tr\left(\bL_{k,\varepsilon}^{1/2}\right)+\tr\left(\bR_{k,\varepsilon}^{1/2}\right)\right).
\end{aligned}
\end{eqnarray}
For terms (b) and (c), similar to the induction in (\ref{equ14}), we have
\begin{eqnarray}
\begin{aligned}\notag
\left\|\bL_{k,\varepsilon}^{-\frac{1}{4p}}\left(\nabla f(\X_k)-\M_k\right)\bR_{k,\varepsilon}^{-\frac{1}{4q}}\right\|_F^2=&\tr\left(\bR_{k,\varepsilon}^{-\frac{1}{4q}}\left(\nabla f(\X_k)-\M_k\right)^T\bL_{k,\varepsilon}^{-\frac{1}{2p}}\left(\nabla f(\X_k)-\M_k\right)\bR_{k,\varepsilon}^{-\frac{1}{4q}}\right)\\
\leq&\frac{1}{\hat\varepsilon^{\frac{1}{2p}}}\tr\left(\bR_{k,\varepsilon}^{-\frac{1}{4q}}\left(\nabla f(\X_k)-\M_k\right)^T\left(\nabla f(\X_k)-\M_k\right)\bR_{k,\varepsilon}^{-\frac{1}{4q}}\right)\\
=&\frac{1}{\hat\varepsilon^{\frac{1}{2p}}}\tr\left(\left(\nabla f(\X_k)-\M_k\right)\bR_{k,\varepsilon}^{-\frac{1}{2q}}\left(\nabla f(\X_k)-\M_k\right)^T\right)\\
\leq&\frac{1}{\hat\varepsilon^{\frac{1}{2p}+\frac{1}{2q}}}\tr\left(\left(\nabla f(\X_k)-\M_k\right)\left(\nabla f(\X_k)-\M_k\right)^T\right)\\
=&\frac{1}{\sqrt{\hat\varepsilon}}\left\|\nabla f(\X_k)-\M_k\right\|_F^2,
\end{aligned}
\end{eqnarray}
and
\begin{eqnarray}
\begin{aligned}\label{equ2}
\left\|\lambda\X_k+\bL_{k,\varepsilon}^{-\frac{1}{2p}}\M_k\bR_{k,\varepsilon}^{-\frac{1}{2q}}\right\|_F^2=&\left\|\bL_{k,\varepsilon}^{-\frac{1}{4p}}\underbrace{\left(\lambda\bL_{k,\varepsilon}^{\frac{1}{4p}}\X_k\bR_{k,\varepsilon}^{\frac{1}{4q}}+\bL_{k,\varepsilon}^{-\frac{1}{4p}}\M_k\bR_{k,\varepsilon}^{-\frac{1}{4q}}\right)}_{\A}\bR_{k,\varepsilon}^{-\frac{1}{4q}}\right\|_F^2\\
=&\tr\left(\bR_{k,\varepsilon}^{-\frac{1}{4q}}\A^T\bL_{k,\varepsilon}^{-\frac{1}{2p}}\A\bR_{k,\varepsilon}^{-\frac{1}{4q}}\right)\leq\frac{1}{\hat\varepsilon^{\frac{1}{2p}}}\tr\left(\bR_{k,\varepsilon}^{-\frac{1}{4q}}\A^T\A\bR_{k,\varepsilon}^{-\frac{1}{4q}}\right)\\
=&\frac{1}{\hat\varepsilon^{\frac{1}{2p}}}\tr\left(\A\bR_{k,\varepsilon}^{-\frac{1}{2q}}\A^T\right)\leq\frac{1}{\hat\varepsilon^{\frac{1}{2p}+\frac{1}{2q}}}\tr\left(\A\A^T\right)=\frac{1}{\sqrt{\hat\varepsilon}}\|\A\|_F^2\\
=&\frac{1}{\sqrt{\hat\varepsilon}}\left\|\lambda\bL_{k,\varepsilon}^{\frac{1}{4p}}\X_k\bR_{k,\varepsilon}^{\frac{1}{4q}}+\bL_{k,\varepsilon}^{-\frac{1}{4p}}\M_k\bR_{k,\varepsilon}^{-\frac{1}{4q}}\right\|_F^2,
\end{aligned}
\end{eqnarray}
where we use $\frac{1}{p}+\frac{1}{q}=1$. Plugging into (\ref{equ1}) and letting $\eta\leq\frac{\sqrt{\hat\varepsilon}}{2L}$, we have
\begin{eqnarray}
\begin{aligned}\label{equ3}
&\E_k\left[f(\X_{k+1})\big|\F_{k-1}\right]-f(\X_k)\\
\leq& \E_k\left[-\frac{\eta}{2}\hspace*{-0.06cm}\left\|\bL_{k,\varepsilon}^{-\frac{1}{4p}}\nabla \hspace*{-0.03cm}f(\X_k)\bR_{k,\varepsilon}^{-\frac{1}{4q}}\right\|_F^2 \hspace*{-0.08cm}-\hspace*{-0.08cm} \frac{\eta}{2}\hspace*{-0.06cm}\left\|\lambda\bL_{k,\varepsilon}^{\frac{1}{4p}}\X_k\bR_{k,\varepsilon}^{\frac{1}{4q}}\hspace*{-0.08cm}+\hspace*{-0.08cm}\bL_{k,\varepsilon}^{-\frac{1}{4p}}\M_k\bR_{k,\varepsilon}^{-\frac{1}{4q}}\right\|_F^2\hspace*{-0.08cm}+\hspace*{-0.08cm} \frac{\eta}{\sqrt{\hat\varepsilon}}\hspace*{-0.06cm}\left\|\nabla \hspace*{-0.03cm}f(\X_k)\hspace*{-0.08cm}-\hspace*{-0.08cm}\M_k\right\|_F^2\right.\hspace*{-0.75cm}\\
&\quad\left.+\frac{9\eta\nu}{K^{1/2}} \left(\tr\left(\bL_{k,\varepsilon}^{1/2}\right)+\tr\left(\bR_{k,\varepsilon}^{1/2}\right)\right) +\frac{L\eta^2}{2\sqrt{\hat\varepsilon}}\left\|\lambda\bL_{k,\varepsilon}^{\frac{1}{4p}}\X_k\bR_{k,\varepsilon}^{\frac{1}{4q}}+\bL_{k,\varepsilon}^{-\frac{1}{4p}}\M_k\bR_{k,\varepsilon}^{-\frac{1}{4q}}\right\|_F^2 \big|\F_{k-1}\right]\\
\leq& \E_k\left[-\frac{\eta}{2}\left\|\bL_{k,\varepsilon}^{-\frac{1}{4p}}\nabla f(\X_k)\bR_{k,\varepsilon}^{-\frac{1}{4q}}\right\|_F^2 - \frac{\eta}{4}\left\|\lambda\bL_{k,\varepsilon}^{\frac{1}{4p}}\X_k\bR_{k,\varepsilon}^{\frac{1}{4q}}+\bL_{k,\varepsilon}^{-\frac{1}{4p}}\M_k\bR_{k,\varepsilon}^{-\frac{1}{4q}}\right\|_F^2\right.\\
&\quad\left. + \frac{\eta}{\sqrt{\hat\varepsilon}}\left\|\nabla f(\X_k)-\M_k\right\|_F^2+ \frac{9\eta\nu}{K^{1/2}} \left(\tr\left(\bL_{k,\varepsilon}^{1/2}\right)+\tr\left(\bR_{k,\varepsilon}^{1/2}\right)\right) \big|\F_{k-1}\right].
\end{aligned}
\end{eqnarray}
Multiplying both sides of (\ref{GM-dif-equ}) with $\omega=1$ by $\frac{\eta}{\sqrt{\hat\varepsilon}(1-\theta)}$, adding it to (\ref{equ3}), and arranging the terms, we have
\begin{eqnarray}
\begin{aligned}\label{equ4}
&\E_k\hspace*{-0.1cm}\left[\hspace*{-0.05cm}f(\X_{k+1})\hspace*{-0.07cm}-\hspace*{-0.07cm}f^* \hspace*{-0.07cm}+\hspace*{-0.07cm} \frac{\eta}{4}\hspace*{-0.05cm}\left\|\lambda\bL_{k,\varepsilon}^{\frac{1}{4p}}\X_k\bR_{k,\varepsilon}^{\frac{1}{4q}}\hspace*{-0.07cm}+\hspace*{-0.07cm}\bL_{k,\varepsilon}^{-\frac{1}{4p}}\M_k\bR_{k,\varepsilon}^{-\frac{1}{4q}}\right\|_F^2 \hspace*{-0.07cm}+\hspace*{-0.07cm} \frac{\eta\theta}{\sqrt{\hat\varepsilon}(1\hspace*{-0.05cm}-\hspace*{-0.05cm}\theta)}\hspace*{-0.05cm}\left\|\nabla\hspace*{-0.05cm} f(\X_k)\hspace*{-0.07cm}-\hspace*{-0.07cm}\M_k\right\|_F^2 \hspace*{-0.05cm}\big|\F_{k-1}\right]\hspace*{-0.75cm}\\
\leq& f(\X_k)-f^* + \E_k\left[-\frac{\eta}{2}\left\|\bL_{k,\varepsilon}^{-\frac{1}{4p}}\nabla f(\X_k)\bR_{k,\varepsilon}^{-\frac{1}{4q}}\right\|_F^2 + \frac{9\eta\nu}{K^{1/2}} \left(\tr\left(\bL_{k,\varepsilon}^{1/2}\right)+\tr\left(\bR_{k,\varepsilon}^{1/2}\right)\right) \big|\F_{k-1}\right]\hspace*{-0.75cm}\\
&+\frac{ L^2\eta^3}{\hat\varepsilon(1-\theta)^2}\left\|\lambda\bL_{k-1,\varepsilon}^{\frac{1}{4p}}\X_{k-1}\bR_{k-1,\varepsilon}^{\frac{1}{4q}}+\bL_{k-1,\varepsilon}^{-\frac{1}{4p}}\M_{k-1}\bR_{k-1,\varepsilon}^{-\frac{1}{4q}}\right\|_F^2 \\
&+\frac{\eta\theta}{\sqrt{\hat\varepsilon}(1-\theta)}\left\|\nabla f(\X_{k-1})-\M_{k-1}\right\|_F^2+ \frac{\eta(1-\theta)}{\sqrt{\hat\varepsilon}}\sigma^2\\
\leq& f(\X_k)-f^*+ \E_k\left[-\frac{\eta}{2}\left\|\bL_{k,\varepsilon}^{-\frac{1}{4p}}\nabla f(\X_k)\bR_{k,\varepsilon}^{-\frac{1}{4q}}\right\|_F^2 + \frac{9\eta\nu}{K^{1/2}} \left(\tr\left(\bL_{k,\varepsilon}^{1/2}\right)+\tr\left(\bR_{k,\varepsilon}^{1/2}\right)\right) \big|\F_{k-1}\right]\hspace*{-0.75cm}\\
&+\hspace*{-0.08cm} \frac{\eta}{4}\hspace*{-0.08cm}\left\|\hspace*{-0.04cm}\lambda\bL_{k-1,\varepsilon}^{\frac{1}{4p}}\X_{k-1}\hspace*{-0.05cm}\bR_{k-1,\varepsilon}^{\frac{1}{4q}}\hspace*{-0.08cm}+\hspace*{-0.08cm}\bL_{k-1,\varepsilon}^{-\frac{1}{4p}}\M_{k-1}\hspace*{-0.05cm}\bR_{k-1,\varepsilon}^{-\frac{1}{4q}}\hspace*{-0.04cm}\right\|_F^2 \hspace*{-0.09cm}+\hspace*{-0.09cm}\frac{\eta\theta}{\sqrt{\hat\varepsilon}(\hspace*{-0.05cm}1\hspace*{-0.09cm}-\hspace*{-0.09cm}\theta\hspace*{-0.04cm})}\hspace*{-0.08cm}\left\|\hspace*{-0.04cm}\nabla \hspace*{-0.07cm}f\hspace*{-0.04cm}(\X_{k-1}\hspace*{-0.05cm})\hspace*{-0.08cm}-\hspace*{-0.08cm}\M_{k-1}\hspace*{-0.05cm}\right\|_F^2 \hspace*{-0.08cm}+\hspace*{-0.09cm} \frac{\eta(\hspace*{-0.05cm}1\hspace*{-0.08cm}-\hspace*{-0.08cm}\theta)\hspace*{-0.04cm}}{\sqrt{\hat\varepsilon}}\sigma^2\hspace*{-0.08cm},\hspace*{-0.75cm}
\end{aligned}
\end{eqnarray}
where we let $\eta^2\leq\frac{\hat\varepsilon(1-\theta)^2}{4L^2}$ in the last inequality. Taking expectation with respect to $\F_{k-1}$ and summing (\ref{equ3}) with $k=1$ and (\ref{equ4}) over $k = 2, \cdots , K$, we have
\begin{eqnarray}
\begin{aligned}\label{equ5}
&\E_{\F_K}\hspace*{-0.12cm}\left[\hspace*{-0.05cm}f(\X_{K+1})\hspace*{-0.07cm}-\hspace*{-0.07cm}f^* \hspace*{-0.07cm}+\hspace*{-0.07cm} \frac{\eta}{4}\hspace*{-0.07cm}\left\|\lambda\bL_{K,\varepsilon}^{\frac{1}{4p}}\X_K\bR_{K,\varepsilon}^{\frac{1}{4q}}\hspace*{-0.07cm}+\hspace*{-0.07cm}\bL_{K,\varepsilon}^{-\frac{1}{4p}}\M_K\bR_{K,\varepsilon}^{-\frac{1}{4q}}\right\|_F^2 \hspace*{-0.07cm}+\hspace*{-0.07cm} \frac{\eta\theta}{\sqrt{\hat\varepsilon}(1\hspace*{-0.07cm}-\hspace*{-0.07cm}\theta)}\left\|\nabla f(\X_K)\hspace*{-0.07cm}-\hspace*{-0.07cm}\M_K\right\|_F^2 \hspace*{-0.05cm}\right]\hspace*{-0.75cm}\\
\leq& f(\X_1)-f^* + \sum_{k=1}^K\E_{\F_k}\left[-\frac{\eta}{2}\left\|\bL_{k,\varepsilon}^{-\frac{1}{4p}}\nabla f(\X_k)\bR_{k,\varepsilon}^{-\frac{1}{4q}}\right\|_F^2+ \frac{9\eta\nu}{K^{1/2}} \left(\tr\left(\bL_{k,\varepsilon}^{1/2}\right)+\tr\left(\bR_{k,\varepsilon}^{1/2}\right)\right)\right]\\
&+\left(\frac{\eta\theta}{\sqrt{\hat\varepsilon}(1-\theta)}+\frac{\eta}{\sqrt{\hat\varepsilon}}\right)\E_{\F_1}\left[\left\|\nabla f(\X_1)-\M_1\right\|_F^2\right] + \frac{(K-1)\eta(1-\theta)}{\sqrt{\hat\varepsilon}}\sigma^2\\
=& f(\X_1)-f^*+ \sum_{k=1}^K\E_{\F_k}\left[-\frac{\eta}{2}\left\|\bL_{k,\varepsilon}^{-\frac{1}{4p}}\nabla f(\X_k)\bR_{k,\varepsilon}^{-\frac{1}{4q}}\right\|_F^2+ \frac{9\eta\nu}{K^{1/2}} \left(\tr\left(\bL_{k,\varepsilon}^{1/2}\right)+\tr\left(\bR_{k,\varepsilon}^{1/2}\right)\right)\right]\\
&+\frac{\eta}{\sqrt{\hat\varepsilon}(1-\theta)}\E_{\F_1}\left[\left\|\nabla f(\X_1)-\M_1\right\|_F^2\right] + \frac{(K-1)\eta(1-\theta)}{\sqrt{\hat\varepsilon}}\sigma^2.
\end{aligned}
\end{eqnarray}
As the gradient is $L$-Lipschitz, we have
\begin{eqnarray}
\begin{aligned}\notag
&f^*\hspace*{-0.07cm}\leq\hspace*{-0.07cm} f\hspace*{-0.07cm}\left(\hspace*{-0.07cm}\X\hspace*{-0.07cm}-\hspace*{-0.07cm}\frac{1}{L}\nabla f(\X)\hspace*{-0.1cm}\right)\hspace*{-0.07cm}\leq\hspace*{-0.07cm} f(\X)\hspace*{-0.07cm}-\hspace*{-0.07cm}\frac{1}{L}\hspace*{-0.07cm}\<\nabla f(\X),\hspace*{-0.07cm}\nabla f(\X)\>\hspace*{-0.07cm}+\hspace*{-0.07cm}\frac{L}{2}\hspace*{-0.07cm}\left\|\frac{1}{L}\nabla f(\X)\right\|_F^2\hspace*{-0.07cm}=\hspace*{-0.07cm}f(\X)\hspace*{-0.07cm}-\hspace*{-0.07cm}\frac{1}{2L}\hspace*{-0.07cm}\left\|\nabla f(\X)\right\|_F^2.
\end{aligned}
\end{eqnarray}
Using the recursion of $\M_1$ and $\M_0=\0$, we have
\begin{eqnarray}
\begin{aligned}\notag
\E_{\F_1}\left[\left\|\nabla f(\X_1)-\M_1\right\|_F^2\right]=&\E_{\F_1}\left[\left\|\theta\nabla f(\X_1)+(1-\theta)\left(\nabla f(\X_1)-\G_1\right)\right\|_F^2\right]\\
=&\theta^2\left\|\nabla f(\X_1)\right\|_F^2+(1-\theta)^2\E_{\F_1}\left[\left\|\nabla f(\X_1)-\G_1\right\|_F^2\right]\\
\leq& 2L\left(f(\X_1)-f^*\right) + (1-\theta)^2\sigma^2.
\end{aligned}
\end{eqnarray}
Plugging into (\ref{equ5}), we have
\begin{eqnarray}
\begin{aligned}\notag
&\E_{\F_K}\hspace*{-0.12cm}\left[\hspace*{-0.05cm}f(\X_{K+1})\hspace*{-0.07cm}-\hspace*{-0.07cm}f^* \hspace*{-0.07cm}+\hspace*{-0.07cm} \frac{\eta}{4}\hspace*{-0.07cm}\left\|\lambda\bL_{K,\varepsilon}^{\frac{1}{4p}}\X_K\bR_{K,\varepsilon}^{\frac{1}{4q}}\hspace*{-0.07cm}+\hspace*{-0.07cm}\bL_{K,\varepsilon}^{-\frac{1}{4p}}\M_K\bR_{K,\varepsilon}^{-\frac{1}{4q}}\right\|_F^2 \hspace*{-0.07cm}+\hspace*{-0.07cm} \frac{\eta\theta}{\sqrt{\hat\varepsilon}(1\hspace*{-0.07cm}-\hspace*{-0.07cm}\theta)}\left\|\nabla f(\X_K)\hspace*{-0.07cm}-\hspace*{-0.07cm}\M_K\right\|_F^2 \hspace*{-0.05cm}\right]\\
\leq& f(\X_1)-f^*+ \sum_{k=1}^K\E_{\F_k}\left[-\frac{\eta}{2}\left\|\bL_{k,\varepsilon}^{-\frac{1}{4p}}\nabla f(\X_k)\bR_{k,\varepsilon}^{-\frac{1}{4q}}\right\|_F^2+ \frac{9\eta\nu}{K^{1/2}} \left(\tr\left(\bL_{k,\varepsilon}^{1/2}\right)+\tr\left(\bR_{k,\varepsilon}^{1/2}\right)\right)\right]\\
&+\frac{2L\eta}{\sqrt{\hat\varepsilon}(1-\theta)}\left(f(\X_1)-f^*\right) + \frac{K\eta(1-\theta)}{\sqrt{\hat\varepsilon}}\sigma^2
\end{aligned}
\end{eqnarray}
and
\begin{eqnarray}
\begin{aligned}\label{equ7}
&\sum_{k=1}^K\E_{\F_k}\left[\left\|\bL_{k,\varepsilon}^{-\frac{1}{4p}}\nabla f(\X_k)\bR_{k,\varepsilon}^{-\frac{1}{4q}}\right\|_F^2\right]\\
\leq& \frac{18\nu}{K^{1/2}}\hspace*{-0.075cm}\sum_{k=1}^K\hspace*{-0.075cm}\E_{\F_k}\hspace*{-0.13cm}\left[\hspace*{-0.05cm}\tr\hspace*{-0.075cm}\left(\hspace*{-0.075cm}\bL_{k,\varepsilon}^{1/2}\right)\hspace*{-0.075cm}+\hspace*{-0.075cm}\tr\hspace*{-0.075cm}\left(\hspace*{-0.075cm}\bR_{k,\varepsilon}^{1/2}\right)\hspace*{-0.05cm}\right] \hspace*{-0.08cm}+\hspace*{-0.075cm}\frac{2\hspace*{-0.05cm}\left(f(\X_1\hspace*{-0.05cm})\hspace*{-0.075cm}-\hspace*{-0.075cm}f^*\hspace*{-0.05cm}\right)}{\eta}\hspace*{-0.075cm}+\hspace*{-0.075cm}\frac{4L}{\sqrt{\hat\varepsilon}(\hspace*{-0.05cm}1\hspace*{-0.075cm}-\hspace*{-0.075cm}\theta)}\hspace*{-0.05cm}\left(f\hspace*{-0.05cm}(\X_1\hspace*{-0.05cm})\hspace*{-0.075cm}-\hspace*{-0.075cm}f^*\hspace*{-0.05cm}\right) \hspace*{-0.075cm}+\hspace*{-0.075cm} \frac{2K\hspace*{-0.05cm}(1\hspace*{-0.075cm}-\hspace*{-0.075cm}\theta)}{\sqrt{\hat\varepsilon}}\sigma^2\\
\leq& \frac{18\nu}{K^{1/2}}\hspace*{-0.075cm}\sum_{k=1}^K\hspace*{-0.075cm}\E_{\F_k}\hspace*{-0.13cm}\left[\hspace*{-0.05cm}\tr\hspace*{-0.075cm}\left(\hspace*{-0.075cm}\bL_{k,\varepsilon}^{1/2}\right)\hspace*{-0.075cm}+\hspace*{-0.075cm}\tr\hspace*{-0.075cm}\left(\hspace*{-0.075cm}\bR_{k,\varepsilon}^{1/2}\right)\hspace*{-0.05cm}\right] \hspace*{-0.08cm}+\underbrace{\hspace*{-0.075cm}\frac{2\hspace*{-0.05cm}\left(f(\X_1\hspace*{-0.05cm})\hspace*{-0.075cm}-\hspace*{-0.075cm}f^*\hspace*{-0.05cm}\right)}{\eta}\hspace*{-0.075cm}+\hspace*{-0.075cm}\frac{4L}{\sqrt{\hat\varepsilon}(\hspace*{-0.05cm}1\hspace*{-0.075cm}-\hspace*{-0.075cm}\theta)}\hspace*{-0.05cm}\left(f\hspace*{-0.05cm}(\X_1\hspace*{-0.05cm})\hspace*{-0.075cm}-\hspace*{-0.075cm}f^*\hspace*{-0.05cm}\right) \hspace*{-0.075cm}+\hspace*{-0.075cm} \frac{2K\hspace*{-0.05cm}(1\hspace*{-0.075cm}-\hspace*{-0.075cm}\theta)}{\sqrt{\hat\varepsilon}}\hat\sigma^2}_{\widetilde C},
\end{aligned}
\end{eqnarray}
where we denote $\hat\sigma^2=\max\left\{\sigma^2,\frac{L\left(f(\X_1)-f^*\right)}{K\gamma^2}\right\}$ with any $\gamma\in(0,1)$. From Lemma \ref{second-moment-bound} with $\omega=1$, we have
\begin{eqnarray}
\begin{aligned}\notag
&\sum_{k=1}^K\E\left[\tr\left(\bL_{k,\varepsilon}^{1/2}\right)+\tr\left(\bR_{k,\varepsilon}^{1/2}\right)\right]\\
\leq& K\left(\underbrace{\tr\left(\Sigma_L^{1/2}\right)+\tr\left(\Sigma_R^{1/2}\right)+(m+n)\sqrt{\varepsilon}}_{C}\right)+\frac{4}{\sqrt{1-\beta}}\sum_{t=1}^K\E\left[\left\|\nabla f(\X_t)\right\|_*\right]. 
\end{aligned}
\end{eqnarray}
From Lemma \ref{gradient-bound} with $\omega=1$, we have
\begin{eqnarray}
\begin{aligned}\notag
&\sum_{k=1}^K\E\left[\left\|\nabla f(\X_k)\right\|_*\right]\\
\leq&\sqrt{\left(\sum_{k=1}^K\E\left[\tr\left(\bL_{k,\varepsilon}^{1/2}\right)+\tr\left(\bR_{k,\varepsilon}^{1/2}\right)\right]\right)\left(\sum_{k=1}^K\E\left[\left\|\bL_{k,\varepsilon}^{-\frac{1}{4p}}\nabla f(\X_k)\bR_{k,\varepsilon}^{-\frac{1}{4q}}\right\|_F^2\right]\right)}\\
\leq&\sqrt{\left(\sum_{k=1}^K\E\left[\tr\left(\bL_{k,\varepsilon}^{1/2}\right)+\tr\left(\bR_{k,\varepsilon}^{1/2}\right)\right]\right)\left( \frac{18\nu}{K^{1/2}}\sum_{k=1}^K\E\left[\tr\left(\bL_{k,\varepsilon}^{1/2}\right)+\tr\left(\bR_{k,\varepsilon}^{1/2}\right)\right] + \widetilde C\right)}\\
\leq&\sqrt{\hspace*{-0.04cm}\left(\hspace*{-0.04cm}\frac{4}{\sqrt{1\hspace*{-0.04cm}-\hspace*{-0.04cm}\beta}}\hspace*{-0.04cm}\sum_{t=1}^K\hspace*{-0.04cm}\E\hspace*{-0.04cm}\left[\left\|\nabla f(\X_t)\right\|_*\right]\hspace*{-0.04cm}+\hspace*{-0.04cm}KC\hspace*{-0.04cm}\right)\hspace*{-0.04cm}\left( \hspace*{-0.04cm}\frac{18\nu}{K^{1/2}}\frac{4}{\sqrt{1\hspace*{-0.04cm}-\hspace*{-0.04cm}\beta}}\hspace*{-0.04cm}\sum_{t=1}^K\hspace*{-0.04cm}\E\hspace*{-0.04cm}\left[\left\|\nabla f(\X_t)\right\|_*\right] \hspace*{-0.04cm}+\hspace*{-0.04cm} 18\nu CK^{1/2} \hspace*{-0.04cm}+\hspace*{-0.04cm} \widetilde C\right)}.
\end{aligned}
\end{eqnarray}
So we have
\begin{eqnarray}
\begin{aligned}\notag
&\left(\sum_{k=1}^K\E\left[\left\|\nabla f(\X_k)\right\|_*\right]\right)^2\leq\frac{288\nu}{K^{1/2}(1-\beta)}\left(\sum_{k=1}^K\E\left[\left\|\nabla f(\X_k)\right\|_*\right]\right)^2 \\
&\qquad+ \frac{4}{\sqrt{1-\beta}}\left(36\nu CK^{1/2} + \widetilde C\right)\sum_{k=1}^K\E\left[\left\|\nabla f(\X_k)\right\|_*\right] + 18\nu C^2K^{3/2} + K\widetilde CC.\hspace*{-0.8cm}
\end{aligned}
\end{eqnarray}
Next, we consider the constants. From Definition \ref{def-X-p} and Fact \ref{PSD-eig-singular}, we have
\begin{eqnarray}
\begin{aligned}\notag
\tr\left(\Sigma_L^{1/2}\right)+\tr\left(\Sigma_R^{1/2}\right)=&\sum_{i=1}^m\sqrt{\sigma_i(\Sigma_L)} + \sum_{i=1}^n\sqrt{\sigma_i(\Sigma_R)}\\
\leq& \sqrt{(m+n)\left(\sum_{i=1}^m \sigma_i(\Sigma_L)+\sum_{i=1}^n \sigma_i(\Sigma_R)\right)}\\
=&\sqrt{(m+n)\left(\tr\left(\Sigma_L\right)+\tr\left(\Sigma_R\right)\right)}\\
=&\sigma\sqrt{2(m+n)}\leq\hat\sigma\sqrt{2(m+n)}.
\end{aligned}
\end{eqnarray} 
Letting $\varepsilon=\frac{\tau\hat\sigma^2}{m+n}$ with any $\tau\leq 1$, we have $(m+n)\sqrt{\varepsilon}\leq \hat\sigma\sqrt{m+n}$ and 
\begin{eqnarray}
\begin{aligned}\notag
C\leq 2.5\hat\sigma\sqrt{m+n}.
\end{aligned}
\end{eqnarray} 
Recall that we require the parameters satisfying the following relations in the above proof
\begin{eqnarray}
\begin{aligned}\notag
\eta\leq \frac{\sqrt{\hat\varepsilon}}{2L},\quad \eta^2\leq\frac{\hat\varepsilon(1-\theta)^2}{4L^2}
\end{aligned}
\end{eqnarray} 
and
\begin{eqnarray}
\begin{aligned}\notag
\eta\lambda\leq \frac{\sqrt{\nu}}{2K^{5/4}}, \quad\|\X_1\|_{op}\leq \frac{\sqrt{\nu}}{K^{1/4}\lambda}, \quad\frac{\sqrt{\nu}}{K^{1/4}}\leq 1, \quad\theta\leq\beta\leq\sqrt{\theta}<1
\end{aligned}
\end{eqnarray} 
in Lemma \ref{spectral-norm-bound}. Letting
\begin{eqnarray}
\begin{aligned}\notag
&1-\theta=\sqrt{\frac{L\left(f(\X_1)-f^*\right)}{K\hat\sigma^2}},\quad \eta=\sqrt{\frac{\hat\varepsilon\left(f(\X_1)-f^*\right)}{4LK\hat\sigma^2}},\quad\nu=\frac{1}{1152}\sqrt{\frac{L(f(\X_1)-f^*)}{\hat\sigma^2}},\\
&\hspace*{1.3cm}\lambda\leq\frac{1}{\sqrt{1152\hat\varepsilon}K^{3/4}}\sqrt[4]{\frac{L^3\hat\sigma^2}{f(\X_1)-f^*}},\quad\|\X_1\|_{op}\leq\sqrt{\frac{\hat\varepsilon K\left(f(\X_1)-f^*\right)}{L\hat\sigma^2}},
\end{aligned}
\end{eqnarray} 
the above requirements are satisfied by the definition of $\hat\sigma^2$. We also have
\begin{eqnarray}
\begin{aligned}\notag
&\frac{1}{1-\beta}\leq\frac{1}{1-\sqrt{\theta}}\leq\frac{2}{1-\theta}=2\sqrt{\frac{K\hat\sigma^2}{L\left(f(\X_1)-f^*\right)}}, \quad\frac{288\nu}{K^{1/2}(1-\beta)}\leq \frac{1}{2},\\
& \widetilde C\leq10\sqrt{\frac{K\hat\sigma^2L\left(f(\X_1)-f^*\right)}{\hat\varepsilon}},\quad \frac{\widetilde C}{\sqrt{1-\beta}}\leq  \frac{14.2\hat\sigma}{\sqrt{\hat\varepsilon}}\sqrt[4]{K^3\hat\sigma^2L\left(f(\X_1)-f^*\right)},\\
& \frac{\nu CK^{1/2}}{\sqrt{1-\beta}}\leq \frac{\sqrt{m+n}}{288}\sqrt[4]{K^3\hat\sigma^2L\left(f(\X_1)-f^*\right)},\\
&\frac{4}{\sqrt{1-\beta}}\left(36\nu CK^{1/2} + \widetilde C\right)\leq \left(\sqrt{m+n}+\frac{57\hat\sigma}{\sqrt{\hat\varepsilon}}\right)\sqrt[4]{K^3\hat\sigma^2L\left(f(\X_1)-f^*\right)},\\
&\nu C^2K^{3/2}\leq \frac{m+n}{144}\sqrt{K^3\hat\sigma^2L\left(f(\X_1)-f^*\right)},\quad K\widetilde CC\leq 25\hat\sigma\sqrt{\frac{m+n}{\hat\varepsilon}}\sqrt{K^3\hat\sigma^2L\left(f(\X_1)-f^*\right)},\\
&18\nu C^2K^{3/2} + K\widetilde CC\leq \left(m+n+25\hat\sigma\sqrt{\frac{m+n}{\hat\varepsilon}}\right)\sqrt{K^3\hat\sigma^2L\left(f(\X_1)-f^*\right)}\\
&\hspace*{3.05cm}\leq \left(\frac{27}{2}(m+n)+\frac{25}{2}\frac{\hat\sigma^2}{\hat\varepsilon}\right)\sqrt{K^3\hat\sigma^2L\left(f(\X_1)-f^*\right)}.
\end{aligned}
\end{eqnarray} 
So we have
\begin{eqnarray}
\begin{aligned}\notag
\frac{1}{2}\left(\sum_{k=1}^K\E\left[\left\|\nabla f(\X_k)\right\|_*\right]\right)^2\leq&\left(\sqrt{m+n}+\frac{57\hat\sigma}{\sqrt{\hat\varepsilon}}\right)\sqrt[4]{K^3\hat\sigma^2L\left(f(\X_1)-f^*\right)}\sum_{k=1}^K\E\left[\left\|\nabla f(\X_k)\right\|_*\right]\\
&+\left(\frac{27}{2}(m+n)+\frac{25}{2}\frac{\hat\sigma^2}{\hat\varepsilon}\right)\sqrt{K^3\hat\sigma^2L\left(f(\X_1)-f^*\right)}.
\end{aligned}
\end{eqnarray}
Solving inequality $x^2-ax-b\leq 0$, we have $x\leq\frac{a+\sqrt{a^2+4b}}{2}\leq a+\sqrt{b}$ and
\begin{eqnarray}
\begin{aligned}\notag
\sum_{k=1}^K\E\left[\left\|\nabla f(\X_k)\right\|_*\right]\leq& \left(2\sqrt{m+n}+\frac{114\hat\sigma}{\sqrt{\hat\varepsilon}} +\sqrt{27(m+n)+25\frac{\hat\sigma^2}{\hat\varepsilon}}\right)\sqrt[4]{K^3\hat\sigma^2L\left(f(\X_1)-f^*\right)}\\
\leq& \left(8\sqrt{m+n}+\frac{119\hat\sigma}{\sqrt{\hat\varepsilon}} \right)\sqrt[4]{K^3\hat\sigma^2L\left(f(\X_1)-f^*\right)}\\
=& \left(8\sqrt{m+n}+\frac{119\hat\sigma}{\sqrt{\hat\varepsilon}} \right)\hspace*{-0.08cm}\max\hspace*{-0.08cm}\left\{\hspace*{-0.12cm}\sqrt[4]{K^3\sigma^2L\hspace*{-0.08cm}\left(f(\X_1)\hspace*{-0.08cm}-\hspace*{-0.08cm}f^*\right)},  \sqrt{ \frac{K L\hspace*{-0.08cm}\left(f(\X_1)\hspace*{-0.08cm}-\hspace*{-0.08cm}f^*\right)}{\gamma} }\right\}\hspace*{-0.08cm}. 
\end{aligned}
\end{eqnarray} 
Dividing both sides by $K$, we have the desired bound. Finally, Lemma \ref{spectral-norm-bound} guarantees
\begin{eqnarray}
\begin{aligned}\notag
\lambda\|\X_k\|_{op}\leq \frac{3\sqrt{\nu}}{K^{1/4}}=\frac{3}{\sqrt{1152}}\sqrt[4]{\frac{L\left(f(\X_1)-f^*\right)}{K\hat\sigma^2}}<1
\end{aligned}
\end{eqnarray} 
by the setting of $\hat\sigma^2$.
\end{proof}

\subsection{Proof of Theorem \ref{main-theorem2}}\label{sec:proof2}
Since we do not consider decoupled weight decay when $\omega\in(0,2]$, Theorem \ref{main-theorem2} does not rely on Lemmas \ref{bound-X}, \ref{bound-update-lemma}, and \ref{spectral-norm-bound}. Consequently, the condition $\theta\leq\beta\leq\sqrt{\theta}$ is no longer required, and $\beta$ can be any constant in  $(0,1)$ that does not depend on $K$.
\begin{proof}
When $\lambda=0$, term (a) in (\ref{equ1}) disappears. Following a similar proof to that of Theorem \ref{main-theorem} and replacing $\frac{1}{p}+\frac{1}{q}=1$ by $\frac{1}{p}+\frac{1}{q}=2-\omega$, we have
\begin{eqnarray}
\begin{aligned}\notag
&\left\|\bL_{k,\varepsilon}^{-\frac{1}{4p}}\left(\nabla f(\X_k)-\M_k\right)\bR_{k,\varepsilon}^{-\frac{1}{4q}}\right\|_F^2\leq\frac{1}{\hat\varepsilon^{1-\omega/2}}\left\|\nabla f(\X_k)-\M_k\right\|_F^2,\\
&\left\|\bL_{k,\varepsilon}^{-\frac{1}{2p}}\M_k\bR_{k,\varepsilon}^{-\frac{1}{2q}}\right\|_F^2\leq\frac{1}{\hat\varepsilon^{1-\omega/2}}\left\|\bL_{k,\varepsilon}^{-\frac{1}{4p}}\M_k\bR_{k,\varepsilon}^{-\frac{1}{4q}}\right\|_F^2,\\
\end{aligned}
\end{eqnarray}
and
\begin{eqnarray}
\begin{aligned}\notag
&\E_k\left[f(\X_{k+1})\big|\F_{k-1}\right]-f(\X_k)\\
\leq& \E_k\left[-\frac{\eta}{2}\left\|\bL_{k,\varepsilon}^{-\frac{1}{4p}}\nabla f(\X_k)\bR_{k,\varepsilon}^{-\frac{1}{4q}}\right\|_F^2 - \frac{\eta}{2}\left\|\bL_{k,\varepsilon}^{-\frac{1}{4p}}\M_k\bR_{k,\varepsilon}^{-\frac{1}{4q}}\right\|_F^2\right.\\
&\quad\left.+ \frac{\eta}{\hat\varepsilon^{1-\omega/2}}\left\|\nabla f(\X_k)-\M_k\right\|_F^2 +\frac{L\eta^2}{2\hat\varepsilon^{1-\omega/2}}\left\|\bL_{k,\varepsilon}^{-\frac{1}{4p}}\M_k\bR_{k,\varepsilon}^{-\frac{1}{4q}}\right\|_F^2 \big|\F_{k-1}\right]\\
\leq& \E_k\left[-\frac{\eta}{2}\left\|\bL_{k,\varepsilon}^{-\frac{1}{4p}}\nabla f(\X_k)\bR_{k,\varepsilon}^{-\frac{1}{4q}}\right\|_F^2 \hspace*{-0.05cm}-\hspace*{-0.05cm} \frac{\eta}{4}\left\|\bL_{k,\varepsilon}^{-\frac{1}{4p}}\M_k\bR_{k,\varepsilon}^{-\frac{1}{4q}}\right\|_F^2 \hspace*{-0.05cm}+\hspace*{-0.05cm} \frac{\eta}{\hat\varepsilon^{1-\omega/2}}\left\|\nabla f(\X_k)-\M_k\right\|_F^2\big|\F_{k-1}\hspace*{-0.05cm}\right]
\end{aligned}
\end{eqnarray}
by letting $\eta\leq\frac{\hat\varepsilon^{1-\omega/2}}{2L}$. Multiplying both sides of (\ref{GM-dif-equ}) with $\lambda=0$ by $\frac{\eta}{\hat\varepsilon^{1-\omega/2}(1-\theta)}$, adding it to the above inequality, and arranging the terms, we have
\begin{eqnarray}
\begin{aligned}\notag
&\E_k\left[f(\X_{k+1})-f^* + \frac{\eta}{4}\left\|\bL_{k,\varepsilon}^{-\frac{1}{4p}}\M_k\bR_{k,\varepsilon}^{-\frac{1}{4q}}\hspace*{-0.05cm}\right\|_F^2 + \frac{\eta\theta}{\hat\varepsilon^{1-\omega/2}(1-\theta)}\left\|\nabla f(\X_k)-\M_k\right\|_F^2 \big|\F_{k-1}\right]\hspace*{-0.75cm}\\
\leq& f(\X_k)\hspace*{-0.04cm}-\hspace*{-0.04cm}f^* \hspace*{-0.04cm}-\hspace*{-0.04cm} \frac{\eta}{2}\E_k\left[\left\|\bL_{k,\varepsilon}^{-\frac{1}{4p}}\nabla f(\X_k)\bR_{k,\varepsilon}^{-\frac{1}{4q}}\right\|_F^2 \big|\F_{k-1}\right]\hspace*{-0.04cm}+\hspace*{-0.04cm}\frac{\eta\theta}{\hat\varepsilon^{1-\omega/2}(1-\theta)}\left\|\nabla f(\X_{k-1})-\M_{k-1}\right\|_F^2\\
&+\frac{ L^2\eta^3}{\hat\varepsilon^{2-\omega}(1-\theta)^2}\left\|\bL_{k-1,\varepsilon}^{-\frac{1}{4p}}\M_{k-1}\bR_{k-1,\varepsilon}^{-\frac{1}{4q}}\right\|_F^2 + \frac{\eta(1-\theta)}{\hat\varepsilon^{1-\omega/2}}\sigma^2\\
\leq& f(\X_k)\hspace*{-0.04cm}-\hspace*{-0.04cm}f^* \hspace*{-0.04cm}-\hspace*{-0.04cm} \frac{\eta}{2}\E_k\left[\left\|\bL_{k,\varepsilon}^{-\frac{1}{4p}}\nabla f(\X_k)\bR_{k,\varepsilon}^{-\frac{1}{4q}}\right\|_F^2 \big|\F_{k-1}\right]\hspace*{-0.04cm}+\hspace*{-0.04cm}\frac{\eta\theta}{\hat\varepsilon^{1-\omega/2}(1-\theta)}\left\|\nabla f(\X_{k-1})-\M_{k-1}\right\|_F^2\\
&+\frac{ \eta}{4}\left\|\bL_{k-1,\varepsilon}^{-\frac{1}{4p}}\M_{k-1}\bR_{k-1,\varepsilon}^{-\frac{1}{4q}}\right\|_F^2 + \frac{\eta(1-\theta)}{\hat\varepsilon^{1-\omega/2}}\sigma^2,
\end{aligned}
\end{eqnarray}
where we let $\eta^2\leq\frac{\hat\varepsilon^{2-\omega}(1-\theta)^2}{4L^2}$ in the last inequality. Similar to the proof of Theorem \ref{main-theorem}, we have
\begin{eqnarray}
\begin{aligned}\label{equ13}
&\sum_{k=1}^K\E\left[\left\|\bL_{k,\varepsilon}^{-\frac{1}{4p}}\nabla f(\X_k)\bR_{k,\varepsilon}^{-\frac{1}{4q}}\right\|_F^2\right]\\
\leq& \underbrace{\frac{2\left(f(\X_1)-f^*\right)}{\eta}+\frac{4L}{\hat\varepsilon^{1-\omega/2}(1-\theta)}\left(f(\X_1)-f^*\right) + \frac{2K(1-\theta)}{\hat\varepsilon^{1-\omega/2}}\hat\sigma^2}_{\widetilde C}.
\end{aligned}
\end{eqnarray}
Indeed, the above proof is essentially the same as the proof of Theorem \ref{main-theorem}, except that we replace $\sqrt{\hat\varepsilon}$ with $\hat\varepsilon^{1-\omega/2}$ and set $\lambda=0$. The rest of the proof differs from that of Theorem \ref{main-theorem}. 

Plugging the above inequality into (\ref{gradient-bound-final}), we have
\begin{eqnarray}
\begin{aligned}\notag
\sum_{k=1}^K\E\left[\left\||\nabla f(\X_k)|^{\omega}\right\|_*\right]\leq&\widetilde C^{\omega/2}\left(KC+\frac{2(1-\beta)^{\omega/2}}{1-\beta^{\omega/2}}\sum_{t=1}^K\E\left[\left\||\nabla f(\X_t)|^{\omega}\right\|_*\right]\right)^{1-\omega/2}\\
\overset{a}\leq&\widetilde C^{\omega/2}\left((KC)^{1-\omega/2}+\left(\frac{2(1-\beta)^{\omega/2}}{1-\beta^{\omega/2}}\sum_{t=1}^K\E\left[\left\||\nabla f(\X_t)|^{\omega}\right\|_*\right]\right)^{1-\omega/2}\right),
\end{aligned}
\end{eqnarray} 
where $C=\tr\left(\Sigma_L^{\omega/2}\right)+\tr\left(\Sigma_R^{\omega/2}\right)+(m+n)\varepsilon^{\omega/2}$ and we use $(x+y)^{\alpha}\leq x^{\alpha}+y^{\alpha}$ for $0\leq\alpha\leq1$ and the fact $\omega\in(0,2]$ in $\overset{a}\leq$. From Young's inequality $xy\leq\frac{x^t}{t}+\frac{y^s}{s}$ with $\frac{1}{t}+\frac{1}{s}=1$ and $t,s>0$, we have
\begin{eqnarray}
\begin{aligned}\notag
&\widetilde C^{\omega/2}\left(\frac{2(1-\beta)^{\omega/2}}{1-\beta^{\omega/2}}\right)^{1-\omega/2}\left(\sum_{t=1}^K\E\left[\left\||\nabla f(\X_t)|^{\omega}\right\|_*\right]\right)^{1-\omega/2}\\
\leq&\frac{\omega}{2}\left(\widetilde C^{\omega/2}\left(\frac{2(1-\beta)^{\omega/2}}{1-\beta^{\omega/2}}\right)^{1-\omega/2}\right)^{2/\omega} + \left(1-\frac{\omega}{2}\right)\left(\sum_{t=1}^K\E\left[\left\||\nabla f(\X_t)|^{\omega}\right\|_*\right]\right)
\end{aligned}
\end{eqnarray} 
for $0<\omega<2$. The above inequality also holds trivially for $\omega=2$. So we have
\begin{eqnarray}
\begin{aligned}\label{equ8}
\sum_{k=1}^K\E\left[\left\||\nabla f(\X_k)|^{\omega}\right\|_*\right]\leq& \frac{2}{\omega}\widetilde C^{\omega/2}(KC)^{1-\omega/2} + \left(\widetilde C^{\omega/2}\left(\frac{2(1-\beta)^{\omega/2}}{1-\beta^{\omega/2}}\right)^{1-\omega/2}\right)^{2/\omega}\\
=& \frac{2}{\omega}\widetilde C^{\omega/2}(KC)^{1-\omega/2} + \widetilde C\underbrace{\left(\frac{2(1-\beta)^{\omega/2}}{1-\beta^{\omega/2}}\right)^{2/\omega-1}}_{\hat C}.
\end{aligned}
\end{eqnarray}
Next, we consider the constants. Choosing $\beta$ as a constant independent of $K$, it follows that $\hat C$ is also independent of $K$. Letting 
\begin{eqnarray}
\begin{aligned}\notag
&1-\theta=\sqrt{\frac{L\left(f(\X_1)-f^*\right)}{K\hat\sigma^2}},\quad \eta=\hat\varepsilon^{1-\omega/2}\sqrt{\frac{f(\X_1)-f^*}{4LK\hat\sigma^2}},\quad \varepsilon=\frac{\tau\hat\sigma^2}{m+n},\quad \tau\leq 1,
\end{aligned}
\end{eqnarray} 
we have
\begin{eqnarray}
\begin{aligned}\label{equ9}
\widetilde C\leq\frac{10\sqrt{K\hat\sigma^2L\left(f(\X_1)-f^*\right)}}{\hat\varepsilon^{1-\omega/2}}.
\end{aligned}
\end{eqnarray} 
From Definition \ref{def-X-p}, Fact \ref{PSD-eig-singular}, $\omega\in(0,2]$, and Jensen's inequality, we have
\begin{eqnarray}
\begin{aligned}\notag
\frac{\tr\left(\Sigma_L^{\omega/2}\right)+\tr\left(\Sigma_R^{\omega/2}\right)}{m+n}=&\frac{\sum_{i=1}^m\left(\sigma_i(\Sigma_L)\right)^{\omega/2} + \sum_{i=1}^n\left(\sigma_i(\Sigma_R)\right)^{\omega/2}}{m+n}\\
\leq& \left(\frac{\sum_{i=1}^m\sigma_i(\Sigma_L) + \sum_{i=1}^n\sigma_i(\Sigma_R)}{m+n}\right)^{\omega/2}\\
=& \left(\frac{\tr(\Sigma_L) +\tr(\Sigma_R)}{m+n}\right)^{\omega/2}=\left(\frac{2\hat\sigma^2}{m+n}\right)^{\omega/2}
\end{aligned}
\end{eqnarray} 
and 
\begin{eqnarray}
\begin{aligned}\notag
\tr\left(\Sigma_L^{\omega/2}\right)+\tr\left(\Sigma_R^{\omega/2}\right)\leq2(m+n)^{1-\omega/2}\hat\sigma^{\omega}.
\end{aligned}
\end{eqnarray} 
From the setting of $\varepsilon=\frac{\tau\hat\sigma^2}{m+n}$, we have $(m+n)\varepsilon^{\omega/2}\leq (m+n)^{1-\omega/2}\hat\sigma^{\omega}$ and 
\begin{eqnarray}
\begin{aligned}\label{equ10}
C\leq 3(m+n)^{1-\omega/2}\hat\sigma^{\omega}.
\end{aligned}
\end{eqnarray} 
Plugging (\ref{equ9}) and (\ref{equ10}) into (\ref{equ8}), we have
\begin{eqnarray}
\begin{aligned}\notag
&\sum_{k=1}^K\E\left[\left\||\nabla f(\X_k)|^{\omega}\right\|_*\right]\\
\leq& \frac{20\hat\sigma^{\frac{3\omega}{2}-\frac{\omega^2}{2}}(m+n)^{(1-\frac{\omega}{2})^2}}{\omega}K^{1-\frac{\omega}{4}}\frac{\left(L\left(f(\X_1)-f^*\right)\right)^{\frac{\omega}{4}}}{\hat\varepsilon^{\frac{\omega}{2}-\frac{\omega^2}{4}}} + \frac{10\hat C\sqrt{K\hat\sigma^2L\left(f(\X_1)-f^*\right)}}{\hat\varepsilon^{1-\frac{\omega}{2}}}.
\end{aligned}
\end{eqnarray}
Specifically, when $\hat\varepsilon=\varepsilon=\frac{\tau\hat\sigma^2}{m+n}$, we have
\begin{eqnarray}
\begin{aligned}\notag
&\sum_{k=1}^K\E\left[\left\||\nabla f(\X_k)|^{\omega}\right\|_*\right]\\
\leq& \frac{20\hat\sigma^{\frac{\omega}{2}}(m+n)^{1-\frac{\omega}{2}}}{\omega}K^{1-\frac{\omega}{4}}\frac{\left(L\left(f(\X_1)-f^*\right)\right)^{\frac{\omega}{4}}}{\tau^{\frac{\omega}{2}-\frac{\omega^2}{4}}} + \frac{10\hat C(m+n)^{1-\frac{\omega}{2}}\hat\sigma^{\omega-1}\sqrt{KL\left(f(\X_1)-f^*\right)}}{\tau^{1-\frac{\omega}{2}}}\\
=& \frac{20(m+n)^{1-\frac{\omega}{2}}}{\omega\tau^{\frac{\omega}{2}-\frac{\omega^2}{4}}}\max\left\{\sigma^{\frac{\omega}{2}}K\left(\frac{L\left(f(\X_1)-f^*\right)}{K}\right)^{\frac{\omega}{4}},K\left(\frac{L\left(f(\X_1)-f^*\right)}{K\gamma}\right)^{\frac{\omega}{2}}\right\} \\
& + \frac{10\hat C(m+n)^{1-\frac{\omega}{2}}\hat\sigma^{\omega-1}\sqrt{KL\left(f(\X_1)-f^*\right)}}{\tau^{1-\frac{\omega}{2}}},
\end{aligned}
\end{eqnarray}
where we use the setting $\hat\sigma^2=\max\left\{\sigma^2,\frac{L\left(f(\X_1)-f^*\right)}{K\gamma^2}\right\}$ for the first part in the last equation. When $\omega\in(0,1)$, from $\hat\sigma^2\geq\frac{L\left(f(\X_1)-f^*\right)}{K\gamma^2}$, the second part can be further bounded as 
\begin{eqnarray}
\begin{aligned}\notag
\frac{10\hat C(m+n)^{1-\frac{\omega}{2}}\sqrt{KL\left(f(\X_1)-f^*\right)}}{\hat\sigma^{1-\omega}\tau^{1-\frac{\omega}{2}}}\leq \frac{10\hat C(m+n)^{1-\frac{\omega}{2}}K}{\tau^{1-\frac{\omega}{2}}\gamma^{\omega-1}}\left(\frac{L\left(f(\X_1)-f^*\right)}{K}\right)^{\omega/2}.
\end{aligned}
\end{eqnarray}
When $\omega\in[1,2]$, it can be written equivalently as
\begin{eqnarray}
\begin{aligned}\notag
&\frac{10\hat C(m+n)^{1-\frac{\omega}{2}}\hat\sigma^{\omega-1}\sqrt{KL\left(f(\X_1)-f^*\right)}}{\tau^{1-\frac{\omega}{2}}}\\
= &\frac{10\hat C(m+n)^{1-\frac{\omega}{2}}}{\tau^{1-\frac{\omega}{2}}}\max\left\{\sigma^{\omega-1}\sqrt{KL\left(f(\X_1)-f^*\right)},\frac{K}{\gamma^{\omega-1}}\left(\frac{L\left(f(\X_1)-f^*\right)}{K}\right)^{\omega/2}\right\}.
\end{aligned}
\end{eqnarray}
Dividing both sides by $K$, we have the desired bound.

\end{proof}

\subsection{Supporting Lemmas}\label{sec:supprt-lemmas}
The following lemma is similar to \citep[Lemma 4]{Li-2025-nips} and we include the proof for completeness. 
\begin{lemma}\label{momentum-bound-lemma}
Suppose that Assumptions 1-3 and condition (\ref{L-assump}) hold and let $\frac{1}{p}+\frac{1}{q}=2-\omega$. Then for Algorithm \ref{alg1}, we have
\begin{eqnarray}
\begin{aligned}\label{GM-dif-equ}
&\E_k\left[\left\|\M_k-\nabla f(\X_k)\right\|_F^2\big|\F_{k-1}\right]\leq\theta\left\|\M_{k-1}-\nabla f(\X_{k-1})\right\|_F^2 \\
&\qquad+ \frac{L^2\eta^2}{(1-\theta)\hat\varepsilon^{1-\omega/2}}\left\|\lambda\bL_{k-1,\varepsilon}^{\frac{1}{4p}}\X_{k-1}\bR_{k-1,\varepsilon}^{\frac{1}{4q}}+\bL_{k-1,\varepsilon}^{-\frac{1}{4p}}\M_{k-1}\bR_{k-1,\varepsilon}^{-\frac{1}{4q}}\right\|_F^2 + (1-\theta)^2\sigma^2.
\end{aligned}
\end{eqnarray}
\end{lemma}
\begin{proof}
Denoting $\Gamma_k=\G_k-\nabla f(\X_k)$, we have $\E_k\left[\Gamma_k\big|\F_{k-1}\right]=0$ and $\E_k\left[\left\|\Gamma_k\right\|_F^2\big|\F_{k-1}\right]\leq\sigma^2$ from (\ref{G-assump1}). From the update of $\M_k$, we have
\begin{eqnarray}
\begin{aligned}\notag
\M_k-\nabla f(\X_k)=& \theta \M_{k-1} + (1-\theta)\G_k - \nabla f(\X_k)\\
=& \theta \left(\M_{k-1}-\nabla f(\X_{k-1})\right) + (1-\theta)\left(\nabla f(\X_k)+\Gamma_k\right) - \nabla f(\X_k) + \theta\nabla f(\X_{k-1})\\
=& \theta \left(\M_{k-1}-\nabla f(\X_{k-1})\right) + (1-\theta)\Gamma_k - \theta\left(\nabla f(\X_k)-\nabla f(\X_{k-1})\right)
\end{aligned}
\end{eqnarray}
and
\begin{eqnarray}
\begin{aligned}\notag
&\E_k\left[\left\|\M_k-\nabla f(\X_k)\right\|_F^2\big|\F_{k-1}\right]\\
=&\left\|\theta \left(\M_{k-1}-\nabla f(\X_{k-1})\right) - \theta\left(\nabla f(\X_k)-\nabla f(\X_{k-1})\right)\right\|_F^2 + (1-\theta)^2\E_k\left[\left\|\Gamma_k\right\|_F^2\big|\F_{k-1}\right]\\
\leq& \theta^2\left(1+\frac{1-\theta}{\theta}\right)\left\|\M_{k-1}-\nabla f(\X_{k-1})\right\|_F^2 + \theta^2\left(1+\frac{\theta}{1-\theta}\right)\left\|\nabla f(\X_k)-\nabla f(\X_{k-1})\right\|_F^2\\
& + (1-\theta)^2\E_k\left[\left\|\Gamma_k\right\|_F^2\big|\F_{k-1}\right]\\
\leq& \theta\left\|\M_{k-1}-\nabla f(\X_{k-1})\right\|_F^2 + \frac{1}{1-\theta}\left\|\nabla f(\X_k)-\nabla f(\X_{k-1})\right\|_F^2 + (1-\theta)^2\sigma^2\\
\leq& \theta\left\|\M_{k-1}-\nabla f(\X_{k-1})\right\|_F^2 + \frac{L^2}{1-\theta}\left\|\X_k-\X_{k-1}\right\|_F^2 + (1-\theta)^2\sigma^2\\
=& \theta\left\|\M_{k-1}-\nabla f(\X_{k-1})\right\|_F^2 + \frac{L^2\eta^2}{1-\theta}\left\|\lambda\X_{k-1}+\bL_{k-1,\varepsilon}^{-\frac{1}{2p}}\M_{k-1}\bR_{k-1,\varepsilon}^{-\frac{1}{2q}}\right\|_F^2 + (1-\theta)^2\sigma^2\\
\leq& \theta\hspace*{-0.08cm}\left\|\hspace*{-0.02cm}\M_{k-1}\hspace*{-0.12cm}-\hspace*{-0.11cm}\nabla\hspace*{-0.05cm} f\hspace*{-0.02cm}(\hspace*{-0.02cm}\X_{k-1}\hspace*{-0.02cm})\hspace*{-0.02cm}\right\|_F^2 \hspace*{-0.08cm}+\hspace*{-0.08cm} \frac{L^2\eta^2}{(\hspace*{-0.05cm}1\hspace*{-0.08cm}-\hspace*{-0.08cm}\theta\hspace*{-0.02cm})\hspace*{-0.02cm}\hat\varepsilon^{1\hspace*{-0.04cm}-\hspace*{-0.02cm}\omega\hspace*{-0.02cm}/\hspace*{-0.02cm}2}}\hspace*{-0.08cm}\left\|\lambda\bL_{k-1,\varepsilon}^{\frac{1}{4p}}\X_{k-1}\hspace*{-0.04cm}\bR_{k-1,\varepsilon}^{\frac{1}{4q}}\hspace*{-0.08cm}+\hspace*{-0.08cm}\bL_{k-1,\varepsilon}^{-\frac{1}{4p}}\M_{k-1}\hspace*{-0.04cm}\bR_{k-1,\varepsilon}^{-\frac{1}{4q}}\right\|_F^2 \hspace*{-0.1cm}+\hspace*{-0.1cm} (\hspace*{-0.05cm}1\hspace*{-0.09cm}-\hspace*{-0.08cm}\theta\hspace*{-0.02cm})^2\hspace*{-0.04cm}\sigma^2\hspace*{-0.04cm},
\end{aligned}
\end{eqnarray}
where we use (\ref{equ2}) in the last inequality but replacing $\frac{1}{p}+\frac{1}{q}=1$ by $\frac{1}{p}+\frac{1}{q}=2-\omega$.
\end{proof} 

The following is the proof of Lemma \ref{condition2-lemma}.
\begin{proof}
When $\G_{i,j}\sim \mathcal N(\mu,\xi^2)$, we have
\begin{eqnarray}
\begin{aligned}\notag
\E\hspace*{-0.04cm}\left[\left(\G\G^T\right)_{p,q}\right]\hspace*{-0.04cm}=\hspace*{-0.04cm}\E\hspace*{-0.04cm}\left[\sum_{j=1}^n\G_{p,j}\G_{q,j}\right]\hspace*{-0.04cm}=\hspace*{-0.04cm}\sum_{j=1}^n\E\hspace*{-0.04cm}\left[\G_{p,j}\G_{q,j}\right]\hspace*{-0.04cm}=\hspace*{-0.04cm}\sum_{j=1}^n\E\hspace*{-0.04cm}\left[\G_{p,j}\right]\E\left[\G_{q,j}\right]\hspace*{-0.04cm}=\hspace*{-0.04cm}n\mu^2\quad\mbox{if}\quad p\neq q
\end{aligned}
\end{eqnarray}
and
\begin{eqnarray}
\begin{aligned}\notag
\E\left[\left(\G\G^T\right)_{p,q}\right]=\sum_{j=1}^n\E\left[\G_{p,j}^2\right]=\sum_{j=1}^n\left(\E\left[\left(\G_{p,j}-\mu\right)^2\right]+\mu^2\right)=n(\xi^2+\mu^2)\quad\mbox{if}\quad p=q.
\end{aligned}
\end{eqnarray}
So 
\begin{eqnarray}
\begin{aligned}\notag
\E\left[\G\G^T\right]=n\mu^2\1_m\1_m^T + n\xi^2\I_m\succeq n\xi^2\I_m,
\end{aligned}
\end{eqnarray}
where $\1_m\in\mathbb R^m$ is the vector with all ones. We also have
\begin{eqnarray}
\begin{aligned}\notag
\E\left[\|\G\|_F^2\right]=\sum_{i=1}^m\sum_{j=1}^n \E\left[\G_{i,j}^2\right]=\sum_{i=1}^m\sum_{j=1}^n \left(\E\left[\left(\G_{i,j}-\mu\right)^2\right]+\mu^2\right)=mn(\xi^2+\mu^2).
\end{aligned}
\end{eqnarray}
So we have
\begin{eqnarray}
\begin{aligned}\notag
\E\left[\G\G^T\right]\succeq n\xi^2\I_m = \frac{\xi^2}{m(\xi^2+\mu^2)}\E\left[\|\G\|_F^2\right]\I_m.
\end{aligned}
\end{eqnarray}
Similarly, we also have
\begin{eqnarray}
\begin{aligned}\notag
\E\left[\G^T\G\right]\succeq m\xi^2\I_n=\frac{\xi^2}{n(\xi^2+\mu^2)}\E\left[\|\G\|_F^2\right]\I_n.
\end{aligned}
\end{eqnarray}
\end{proof}

\section{Experiments}\label{sec-exp}
In this section, we conduct experiments on real-world deep learning tasks to examine whether our theoretical claims are reflected in practical training dynamics. Specifically, we examine the two kinds of relationships that make our convergence rate practically meaningful: the nuclear-to-Frobenius norm ratio of the full gradient and the effective spectral floor $\hat\varepsilon$ of the Shampoo preconditioners relative to the noise-dependent scale $\hat{\sigma}^2/(m+n)$ and the numerical regularizer $\varepsilon$. We pretrain the GPT-2 \citep{radford2019language} model from scratch on the OpenWebText dataset \citep{Gokaslan2019OpenWeb} to validate these claims\footnote{Our code is available at https://github.com/adonis-dym/Convergence-Rate-AdamW-Style-Shampoo}.


A central quantity in our empirical evaluation is the full gradient $\nabla f(\X)$, which is required both for computing the gradient norm ratio and for estimating the stochastic gradient noise $\sigma^2$. Following the protocol of \citet[Section~E]{Li-2025-nips}, we alternately switch between \textit{training phases} and \textit{logging phases}, and use the logging phases to periodically approximate $\nabla f(\X)$ via accumulated large batch gradients. This interleaved design allows us to measure the full gradient related quantities needed in the analysis without interfering with the normal training dynamics.



Following common practice in non-diagonal preconditioning methods
\citep{muon2024,distribute-shampoo-20,distribute-shampoo-23}, we apply the
AdamW-style Shampoo only to non-embedding two-dimensional
parameters, while using AdamW for the remaining parameters. Concretely,
for a Transformer block (omitting the multi-head structure)
the attention and feed-forward modules can be written as
\begin{equation}\notag
    \operatorname{Attn}(\mathbf{X})
    =
    \operatorname{softmax}\left(
    \frac{\mathbf{X}\mathbf{W}_Q (\mathbf{X}\mathbf{W}_K)^\top}{\sqrt{d}}
    \right)\mathbf{X}\mathbf{W}_V\mathbf{W}_O\quad\mbox{and}\quad
    \operatorname{MLP}(\mathbf{X})
    =
    \phi(\mathbf{X}\mathbf{W}_1) \mathbf{W}_2,
\end{equation}
where $\phi$ denotes the activation function. The Shampoo-optimized
parameters therefore consist of four representative matrix classes:
(i) the attention input projection matrices, corresponding to the
merged QKV projection in GPT-2, i.e. $\operatorname{concat}(\mathbf{W}_Q\mathbf{W}_K\mathbf{W}_V$); (ii) the attention output projection
matrix $\mathbf{W}_O$; (iii) the first layer in MLP module $\mathbf{W}_1$; and (iv) the second layer in MLP module $\mathbf{W}_2$. For clarity and brevity, we report the averaged results within each class.

\begin{figure}[t]
  \centering
  \includegraphics[width=\linewidth]{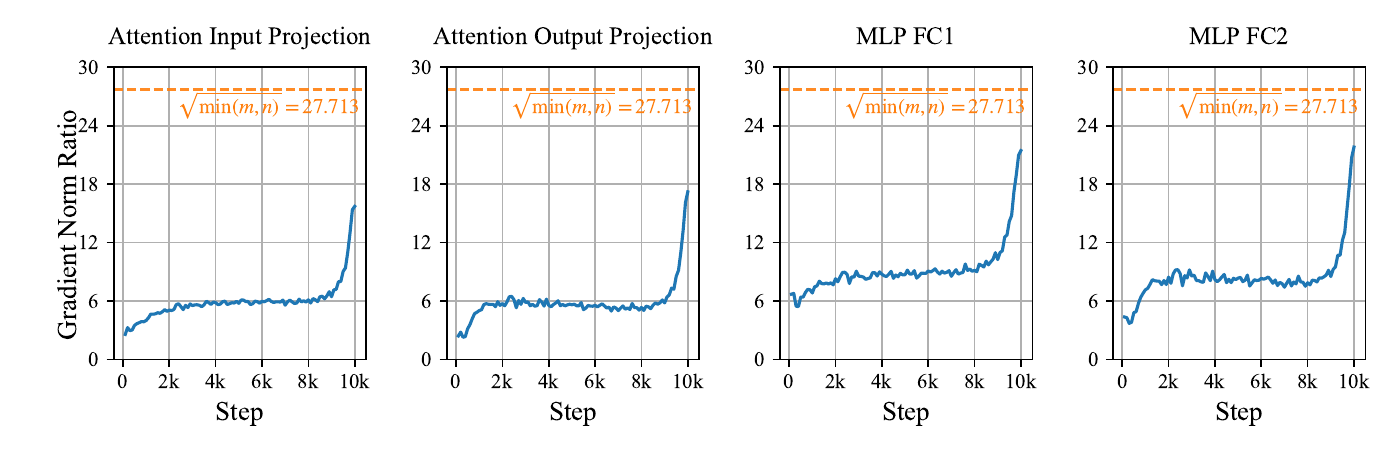}
  \caption{
  The gradient norm ratio 
  $\lVert \nabla f(\mathbf{X}) \rVert_* / \lVert \nabla f(\mathbf{X}) \rVert_F$
  during GPT-2 pretraining on OpenWebText for each class of Shampoo-handled parameters.
  The dashed horizontal line indicates the value of $\sqrt{\min(m,n)}$ for each matrix shape.
  }
  \label{fig:nuc_fro_ratio}
\end{figure}

For the training recipe, we use the standard Megatron-LM GPT-2 Small configuration \citep{shoeybi2019megatron} and follow the related works for shampoo settings \citep{distribute-shampoo-20,distribute-shampoo-23} with minimal modifications. The model contains 12 Transformer layers, hidden size 768, 12 attention heads, and sequence length 1024. We set the learning rate to  $10^{-2}$ and weight decay to 0.05. For the Shampoo-preconditioned parameters, we set $\theta=0.9$, $\beta=0.999$, $\varepsilon=10^{-12}$, and $p=q=2$; for the remaining parameters handled by AdamW, we use $(\beta_1,\beta_2)=(0.9,0.95)$. The learning rate is linearly warmed up over the first 1000 training steps and then decayed according to a cosine schedule over the remaining 9000 steps. We set the global batch size to 640 and apply gradient clipping with threshold 1.0, and complete the training with 8 NVIDIA H200 GPUs.

Figure~\ref{fig:nuc_fro_ratio} reports the value of the gradient norm ratio
$\lVert \nabla f(\mathbf{X}) \rVert_* / \lVert \nabla f(\mathbf{X}) \rVert_F$
for the four classes of Shampoo-handled matrices.
Across all parameter categories, we observe that this ratio stays at the same magnitude as the theoretical upper bound
$\sqrt{\min(m,n)}$ throughout training, indicated by the dashed reference lines.
This behavior suggests that the relationship
$\lVert \nabla f(\mathbf{X}) \rVert_* = \Theta(\sqrt{\min\{m,n\}})\,\lVert \nabla f(\mathbf{X}) \rVert_F$ indeed holds in practice.

Next, we turn to the empirical estimation of the stochastic gradient noise, and compare the noise-dependent scale $\hat{\sigma}^2/(m+n)$ with $\varepsilon$ and $\hat\varepsilon$. To make the noise-dependent scale tractable, we use the approximation $\hat{\sigma}^2 \approx \sigma^2 \approx \|\G-\nabla f(\X)\|_F^2$, where $\G$ is the mini-batch stochastic gradient and $\nabla f(\X)$ is approximated by the accumulated large batch gradient from the logging phase. We then compute the corresponding scale $\hat{\sigma}^2/(m+n)$ for each Shampoo-handled matrix. To estimate the effective spectral floor $\hat{\varepsilon}$, we record the minimum eigenvalues of the regularized left and right preconditioner matrices $\bL_{k,\varepsilon}$ and $\bR_{k,\varepsilon}$, respectively. These quantities are readily available during training, since they are produced as a byproduct of the eigendecomposition used to compute the matrix inverse roots in Algorithm \ref{alg1}.

\begin{figure}[t]
  \centering
  \includegraphics[width=\linewidth]{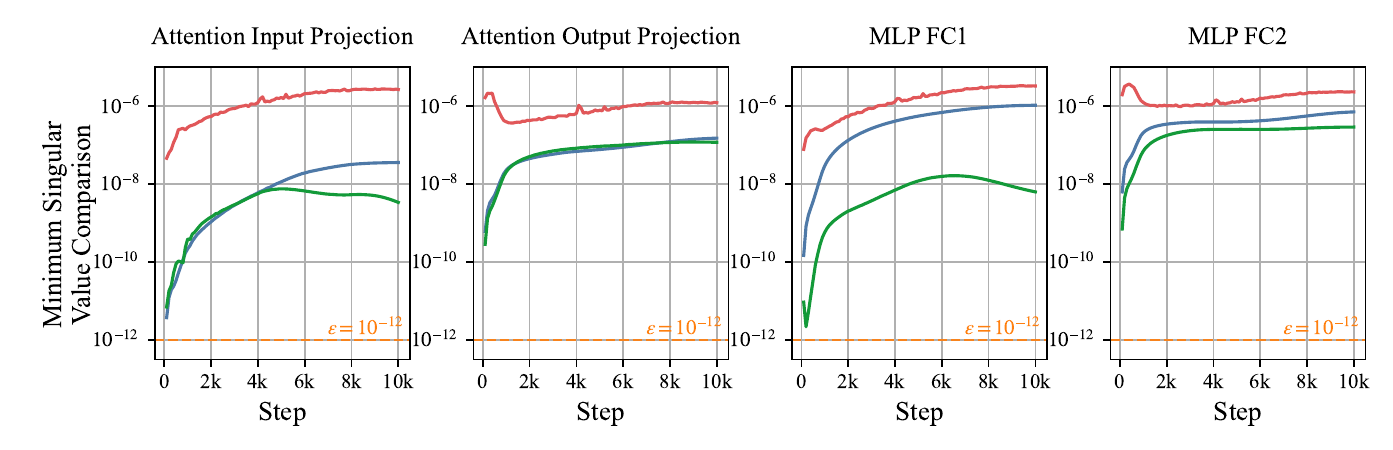}
\caption{
Comparison of the minimum eigenvalues with respect to the stochastic gradient noise during GPT-2 pretraining on OpenWebText dataset.
The red curve indicates the noise-dependent scale $\hat\sigma^2/(m+n)$, while the blue and green curves
correspond to the minimum eigenvalues of the regularized left and right preconditioner matrices
$\bL_{k,\varepsilon}$ and $\bR_{k,\varepsilon}$, respectively.
All quantities are plotted on a logarithmic scale.
}

  \label{fig:min_sval_noise}
\end{figure}


We show the results in Figure~\ref{fig:min_sval_noise}. Across all four classes of Shampoo-handled parameters, the minimum eigenvalues
of both $\bL_{k,\varepsilon}$ and $\bR_{k,\varepsilon}$ are consistently several
orders of magnitude larger than the numerical regularizer
$\varepsilon=10^{-12}$. Moreover, these estimates of $\hat{\varepsilon}$ remain
comparable to the noise-dependent scale $\hat{\sigma}^2/(m+n)$ throughout
training. This indicates that the theoretical bound \eqref{rate1} derived in Theorem \ref{main-theorem} aligns much closer to the favorable regime in \eqref{rate3} than the worst-case rate \eqref{rate2}.

\subsection{Comparison with the Other Optimizers}

The AdamW-style Shampoo, originally implemented in \citep{distribute-shampoo-23}, is a classical optimizer that won the AlgoPerf neural network training competition \citep{AlgoPerf-competition-25}. It has been adopted as a baseline for comparison in the literature. For instance, \citet{Frans-2025} compared different optimizers on GPT‑2 using the OpenWebText dataset. Table 2 of \citep{Frans-2025} shows that Shampoo (in the AdamW-style) outperforms AdamW and achieves performance comparable to Muon \citep{muon2024}. SOAP \citep{soap-2025-iclr} and Splus \citep{splus-2025-nips}, which are both extensions of Shampoo, achieve even better performance in \citep{Frans-2025}. Therefore, we do not include additional optimizer comparisons in this primarily theoretical paper.

\section{Conclusion}

This paper studies the convergence of AdamW-style Shampoo with both one-sided and two-sided preconditioning. When the exponents of the two preconditioners sum to $1/2$, we establish the convergence rate $\frac{1}{K}\sum_{k=1}^K\E\left[\|\nabla f(\X_k)\|_*\right]\leq \bO(\frac{\sqrt{m+n}C}{K^{1/4}})$ measured by nuclear norm. This rate is analogous to the optimal $\frac{1}{K}\sum_{k=1}^K\E\left[\|\nabla f(\X_k)\|_F\right]\leq \bO(\frac{C}{K^{1/4}})$ convergence rate of SGD in the ideal case where $\|\nabla f(\X)\|_*= \Theta(\sqrt{\min\{m,n\}})\|\nabla f(\X)\|_F$ and $m$ and $n$ are of comparable magnitude. Then, we extend our analysis to the setting where the preconditioning exponents do not sum to $1/2$ and establish convergence with a more involved rate.

\bibliography{example_paper}
\bibliographystyle{icml2020}

\end{document}